\documentclass[12pt]{amsart}

\usepackage{amsmath, amssymb, latexsym}
\usepackage{mathrsfs}
\usepackage{enumerate}
\usepackage[all, knot]{xy}
\xyoption{arc}
\usepackage[usenames,dvipsnames]{color}

\usepackage[colorlinks=true, pdfstartview=FitV, linkcolor=blue,%
citecolor=blue, urlcolor=blue]{hyperref}
\usepackage{yfonts}

\numberwithin{equation}{section}

\newcommand{\nc}{\newcommand}

\newtheorem{theorem}{Theorem}[section]
\newtheorem{proposition}[theorem]{Proposition}
\newtheorem{corollary}[theorem]{Corollary}
\newtheorem{conjecture}[theorem]{Conjecture}
\newtheorem{lemma}[theorem]{Lemma}

\theoremstyle{definition}

\newtheorem{definition}[theorem]{Definition}
\newtheorem{example}[theorem]{Example}
\newtheorem{remark}[theorem]{Remark}

\nc{\Prop}{\begin{proposition}}
\nc{\enprop}{\end{proposition}}
\nc{\Conj}{\begin{conjecture}}
\nc{\enconj}{\end{conjecture}}
\nc{\Def}{\begin{definition}}
\nc{\edf}{\end{definition}}

\def\g{\mathfrak g}
\nc{\gt}{\mathfrak{t}}
\def\la{\langle}
\def\ra{\rangle}
\def\lan{\langle}
\def\ran{\rangle}
\def\k{\mathbf k}
\nc{\cor}{\Bbbk}
\def\P{\mathrm{P}}
\def\cQ{\mathcal Q}
\def\cP{\mathcal P}
\def\rtl{\mathrm{Q}}
\def\O{\mathcal O}
\def\A{\mathbb A}
\def\Q{\mathbb Q}
\def\Z{\mathbb Z}
\def\degZ{{\rm deg}_\Z}
\nc\Oint{\mathcal{O}_{{\rm int}}}
\def\Iev{I_{\rm even}}
\def\Iod{I_{\rm odd}}
\def\Pev{\P_{\rm even}}
\def\pa{{\rm p}}
\def\ot{\otimes}
\def\bt{\boxtimes}
\def\wt{{\rm wt}}
\nc{\on}{\operatorname}
\def\Uqg{{\mathsf U}_v(\g)}

\nc{\VB}[1][{\theta,\pa}]{{\rm V}_{#1}}
\nc{\VBe}[1][{\theta,\pa}]{{\rm V}_{#1}}
\nc{\tn}{{\mathring{\theta}}}
\nc{\MB}[1][{\theta,\pa}]{\mathrm{M}_{#1}}
\nc{\MBe}[1][{\theta,\pa}]{\mathrm{M}_{#1}}
\nc{\NB}[1][{\theta,\pa}]{\mathrm{N}_{#1}}
\nc{\NBe}[1][{\theta,\pa}]{\mathrm{N}_{#1}}

\newcommand{\UqBg}[1][{\theta}]{{\rm U}^{q}_{#1}(\g)}
\newcommand{\BKM}[1][{\theta}]{{\rm U}^{#1}_q(\g)}

\def\UqmBg{\BKM^-}
\def\Us{{\mathcal U}}
\def\Vs{{\mathcal V}}
\def\Uqsg{\Us(\g)}
\def\Uqmsg{\Us^-(\g)}
\def\Uqpsg{\Us^+(\g)}
\def\Uqzsg{\Us^0(\g)}
\def\Usu{{\bf U}}
\def\Vsu{{\bf V}}
\def\Msu{{\bf M}}

\def\Uqsug{\Usu(\g)}
\def\Uqmsug{\Usu^-(\g)}
\def\Uqzsug{\Usu^0(\g)}
\def\Uqpsug{\Usu^+(\g)}

\def\super{{\rm super}}
\def\Mod{{\rm Mod}}
\def\Proj{ {\rm Proj}}
\def\Rep{{\rm Rep}}
\def\MOD{{\rm Mod}_\super}
\def\id{{\rm id}}
\def\PROJ{{\rm Proj}_\super}
\def\REP{{\rm Rep}_\super}
\def\ch{{\rm ch}}
\newcommand{\ind}{{\rm Ind}}

\newcommand{\soc}{{\rm soc}}
\newcommand{\mult}{\on{mult}}
\newcommand{\hd}{{\rm hd}}
\newcommand{\ad}{{\rm ad}}

\newcommand{\eps}{\varepsilon}

\newcommand{\al}{\alpha}
\def\Hom{{\rm Hom}}
\def\End{{\rm End}}

\nc{\be}{\begin{enumerate}} \nc{\ee}{\end{enumerate}}
\nc{\bnum}{\be[{\rm(i)}]} \nc{\bna}{\be[{\rm(a)}]}
\nc{\eq}{\begin{eqnarray}} \nc{\eneq}{\end{eqnarray}}
\nc{\seteq}{\mathbin{:=}} \nc{\set}[2]{\left\{#1\mid #2\right\}}
\nc{\ba}{\begin{array}} \nc{\ea}{\end{array}}
\nc{\hs}{\hspace*}
\nc{\car}{\mathrm{A}}
\nc{\cl}{\colon}
\nc{\eqn}{\begin{eqnarray*}} \nc{\eneqn}{\end{eqnarray*}}
\newcommand{\shc}{{\mathscr{C}}}
\newcommand{\Irr}{\mathcal{I}\mathit{rr}}
\nc{\tp}{\mathrm{top}}
\newcommand{\isoto}[1][]{\mathop{\xrightarrow[#1]%
{\rule{0pt}{.9ex}%
{\raisebox{-.6ex}[0ex][-.7ex]{$\mspace{4mu}\sim\mspace{3mu}$}}}}}

\nc{\ro}{{\rm(}}
\nc{\rf}{{\rm)}}
\nc{\bl}{\bigl(}
\nc{\br}{\bigr)}
\nc{\soplus}%
{\mathop{\text{\scriptsize\raisebox{.5ex}{$\displaystyle\bigoplus$}}}}
\nc{\Lemma}{\begin{lemma}}
\nc{\enlemma}{\end{lemma}}
\nc{\Cor}{\begin{corollary}}
\nc{\encor}{\end{corollary}}
\nc{\vs}{\vspace*}
\nc{\rk}{\on{rank}}
\nc{\Proof}{\begin{proof}}
\nc{\QED}{\end{proof}}
\nc{\tens}{\mathop{\otimes}\limits}
\nc{\qintt}[2][i]{{\la #2\ra\mspace{1mu}{}^{\tilde p}_{#1}}}
\nc{\factt}[2][i]{{\la #2\ra\mspace{1mu}{}^{\tilde p}_{#1}}\,!}
\nc{\qintp}[2][i]{[#2]\mspace{1mu}{}^{\pa}_{#1}}
\nc{\factp}[2][i]{{[#2]\mspace{1mu}{}^{\pa}_{#1}}\,!}
\nc{\binm}[2]{\genfrac{[}{]}{0pt}{0}{#1}{#2}}
\nc{\binma}[3][i]{{\genfrac{\langle}{\rangle}{0pt}{0}{#2}{#3}}%
^{\hs{-.5ex}\pi}_{\hs{-.5ex}{#1}}}
\nc{\binmpi}[3][i]{{\genfrac{[}{]}{0pt}{0}{#2}{#3}}%
^{\hs{-.5ex}\pi}_{\hs{-.5ex}{#1}}}
\nc{\binmp}[3][i]{{\genfrac{[}{]}{0pt}{0}{#2}{#3}}^{\pa}_{#1}}
\nc{\lam}{\lambda}
\nc{\dK}{\widetilde{K}}

\nc{\vphi}{\varphi}
\nc{\monoto}{\rightarrowtail}
\nc{\Ocat}[1][{\mathrm{int}}]{\mathcal{O}_{#1}}
\nc{\epi}{\twoheadrightarrow}
\nc{\Ad}{\on{Ad}}
\nc{\tpa}{\widetilde{\pa}}
\nc{\tth}{\widetilde{\theta}}
\nc{\tip}{\widetilde{p}}

\nc{\F}[1][{\theta,\pa}]{\mathcal{F}(#1)}
\nc{\FS}[1][{\theta,\pa}]{U_{#1}(\g)}
\nc{\HF}[1][{\tth,\tpa}]{\mathcal{H}(#1)}
\nc{\HS}[1][{\tth,\tpa}]{{\mathcal{U}}_{#1}(\g)}
\nc{\FSp}[1][{\theta,\pa}]{U_{#1}^+(\g)}
\nc{\FSm}[1][{\theta,\pa}]{U_{#1}^-(\g)}
\nc{\FSz}[1][{\theta,\pa}]{U_{#1}^0(\g)}
\nc{\HSp}[1][{\tth,\tpa}]{{\mathcal{U}}_{#1}^+(\g)}
\nc{\HSm}[1][{\tth,\tpa}]{{\mathcal{U}}_{#1}^-(\g)}
\nc{\HSz}[1][{\tth,\tpa}]{{\mathcal{U}}_{#1}^0(\g)}
\nc{\ev}{\mathrm{ev}}
\nc{\trp}[1]{{{}^\mathrm{t}{#1}}}
\nc{\px}[1]{{(#1)}}
\nc{\pix}[1]{{<#1>}}
\nc{\bal}{\begin{align}}
\nc{\eal}{\end{align}}
\nc{\baln}{\begin{align*}}
\nc{\ealn}{\end{align*}}
\nc{\ssum}{\mathop\sum\limits}
\nc{\Pevp}{\P^+_{\mathrm{even}}}
\nc{\C}{\mathbb{C}}
\nc{\Bgt}[1][{{}^t\tth,\tip}]{\mathrm{B}_{#1}(\g)}
\nc{\Bg}[1][{\tth,\tip}]{\mathrm{B}_{#1}(\g)}
\nc{\BqBg}[1][{\theta}]{\mathrm{B}^{#1}(\g)}
\nc{\Bqsug}{{\mathcal B}(\g)}
\nc{\Lam}{\Lambda}
\nc{\tC}{\widetilde{\shc}}
\nc{\Ob}{\on{Ob}}
\nc{\To}[1][\quad]{\xrightarrow{\;#1\;}}
\nc{\sd}[1]{{#1}{}^{\mathrm{D}}}
\nc{\one}{\mathbf{1}}
\nc{\Fct}{\mathrm{Fct}}
\nc{\Fcts}{\mathrm{Fct}_\super}
\nc{\ct}{\mathrm{{CT}}}
\nc{\Fm}[1][{\theta,\pa}]{\mathcal{F}^-(#1)}
\nc{\Fp}[1][{\theta,\pa}]{\mathcal{F}^+(#1)}
\nc{\noi}{\noindent}
\nc{\qtext}[1]{\quad\text{#1}\quad}
\nc{\Sl}{\mathfrak{sl}}
\nc{\EL}{E^\Lambda}
\nc{\FL}{F^\Lambda}
\nc{\La}{\Lambda}
\nc{\rev}{\mathrm{\mspace{1mu}sr}}
\nc{\Mods}{\MOD}
\nc{\op}{\mathrm{opp}}
\nc{\ex}[1]{\mspace{2mu}\mathrm{e}^{#1}}
\nc{\sop}{\mathrm{superop}}
\nc{\dg}{\mathrm{d}}
\nc{\alm}[1]{\al_{\substack{\rule[.8ex]{0ex}{1.5ex}}#1}}
\nc{\FA}{{\textgoth{A}}}
\nc{\CHom}{{{\mathscr H}\mspace{-4mu}om}}
\nc{\CEnd}{{\mathcal{E}nd}}
\nc{\can}{{\mathrm{can}}}
\nc{\ol}{\overline}
\newcommand{\scbul}{\,\raise.4ex\hbox{$\scriptscriptstyle\bullet$}\,}

\begin{document}
\title[Supercategorification of quantum Kac-Moody algebras II]
{Supercategorification of \\ quantum Kac-Moody algebras II}

\author[Seok-Jin Kang]{Seok-Jin Kang$^1$}
\thanks{$^1$ This work was supported by NRF Grant \# 2012-005700 and NRF Grant \# 2011-0027952.}
\address{Department of Mathematical Sciences and Research Institute of Mathematics,
Seoul National University, 599 Gwanak-ro, Gwanak-gu, Seoul 151-747,
Korea} \email{sjkang@snu.ac.kr}

\author[Masaki Kashiwara]{Masaki Kashiwara$^{2}$}
\thanks{$^2$ This work was supported by Grant-in-Aid for
Scientific Research (B) 22340005, Japan Society for the Promotion of
Science.}
\address{Research Institute for Mathematical Sciences, Kyoto University, Kyoto 606-8502,
Japan, and Department of Mathematical Sciences, Seoul National
University, 599 Gwanak-ro, Gwanak-gu, Seoul 151-747, Korea}
\email{masaki@kurims.kyoto-u.ac.jp}

\author[Se-jin Oh]{Se-jin Oh$^{3}$}
\thanks{$^{3}$  This work was supported by Priority Research Centers
 Program through the National Research Foundation of Korea (NRF) funded by
 the Ministry of Education, Science and Technology \# 2012-047640.}

\address{Pohang Mathematics Institute, Pohang University of Science and
Technology, San31 Hyoja-Dong Nam-Gu, Pohang 790-784, Korea}
\email{sejin092@gmail.com}

\date{March 8, 2013}
\subjclass[2000]{ 05E10, 16G99, 81R10}
\keywords{categorification, quiver Hecke superalgebras, cyclotomic quotient,
quantum Kac-Moody superalgebras}

\begin{abstract}

In this paper, we investigate the supercategories consisting of
supermodules over quiver Hecke superalgebras and cyclotomic quiver
Hecke superalgebras. We prove that these supercategories provide a
supercategorification of a certain family of quantum superalgebras
and their integrable highest weight modules. We show that, by taking
a specialization, we obtain a supercategorification of quantum
Kac-Moody superalgebras and their integrable highest weight modules.

\end{abstract}

\maketitle
\tableofcontents
\vskip 2em

\section*{Introduction}

This is a continuation of our previous work on the
supercategorification of quantum Kac-Moody algebras and their
integrable highest weight modules \cite{KKO12}. We first recall the
main results of \cite{KKO12}.

Let $I$ be an index set, $(\car =(a_{ij})_{i, j \in I}, \P, \Pi,
\Pi^{\vee})$ be a symmetrizable Cartan datum and $U_q(\g)$ be the
corresponding quantum group (or quantum Kac-Moody algebra). Since
$\car$ is symmetrizable, there is a diagonal matrix ${\rm D}$ with
positive integral entries ${\rm d}_i$ $(i\in I)$ such that ${\rm
DA}$ is symmetric. For a dominant integral weight $\Lambda \in
\P^{+}$, we denote by $V(\Lambda)$ the integrable highest weight
$U_q(\g)$-module with highest weight $\Lambda$. The integral forms
of $U_q(\g)$ and $V(\Lambda)$ will be denoted by $U_{\A}(\g)$ and
$V_{\A}(\Lambda)$, where $\A = \Z[q, q^{-1}]$.

In \cite{KL1, KL2, R08}, Khovanov-Lauda and Rouquier independently
introduced a new family of graded algebras, the {\em
Khovanov-Lauda-Rouquier algebras} or {\em quiver Hecke algebras},
that gives a categorification of quantum Kac-Moody algebras.
Furthermore, Khovanov and Lauda conjectured that the cyclotomic
quotients of quiver Hecke algebras give a categorification of
integrable highest weight modules over quantum Kac-Moody algebras.
This conjecture was proved by Kang and Kashiwara \cite{KK11}. (See
\cite{W10} for another proof of this conjecture.)

Naturally, our next goal is to find a super-version of
Khovanov-Lauda-Rouquier categorification theorem and Kang-Kashiwara
cyclotomic categorification theorem. In \cite{KKT11}, Kang,
Kashiwara and Tsuchioka introduced the notion of {\em quiver Hecke
superalgebras} and {\em quiver Hecke-Clifford superalgebras} which
are $\Z$-graded algebras over a commutative graded ring $\k =
\oplus_{n\ge 0} \k_{n}$ with $\k_{0}$ a field. They showed that
these superalgebras are weakly Morita superequivalent and that,
after some completion, the quiver Hecke-Clifford superalgebras are
isomorphic to the affine Hecke-Clifford superalgebras. It folws
that the same statements hold for the cyclotomic quotients of these
superalgebras,

Based on the results of \cite{KKT11}, Kang, Kashiwara and Oh proved
that the quiver Hecke superalgebras and the cyclotomic quiver Hecke
superalgebras provide a supercategorification of quantum Kac-Moody
algebras and their integrable highest weight modules \cite{KKO12}.
Here, a supercategorification of an algebraic structure means a
construction of a 1-supercategory or a 2-supercategory whose
Grothendieck group is isomorphic to the given algebraic structure.
To describe the main results of \cite{KKO12} in more detail, we need
to fix some notations and conventions.

Let $\cor$ be a commutative ring in which 2 is invertible. A {\em
supercategory} is a $\cor$-linear category $\shc$ with an
endofunctor $\Pi$ and a natural isomorphism $\xi: \Pi^2
\overset{\sim} \longrightarrow \id$ such that $\xi \cdot \Pi = \Pi
\cdot \xi$. A {\em 1-supercategory} is a $\cor$-linear category
$\shc$ such that

\ \ (i) $\Hom_{\shc}(X,Y)$ is endowed with a $\cor$-supermodule
structure for all $X,Y \in \shc$,

\ \ (ii) the composition map $$\Hom_{\shc}(Y,Z) \times
\Hom_{\shc}(X,Y) \rightarrow \Hom_{\shc}(X,Z)$$ is
$\cor$-superbilinear.

The notion of supercategories and that of 1-supercategories are
almost equivalent. One can also define the notion of {\em
2-supercategories}. The basic properties of supercategories,
1-supercategories and 2-supercategories are explained in Section
\ref{Sec:supers}.

Let $A = A_0 \oplus A_1$ be a $\cor$-superalgebra with an involution
$\phi_{A}$ defined by
$$\phi_{A}(a)=(-1)^{\epsilon} a \quad (a \in A_{\epsilon}, \
\epsilon = 0, 1).$$ We denote by $\Mod(A)$ be the category of left
$A$-modules. Then $\Mod(A)$ is endowed with a supercategory
structure induced by $\phi_{A}$. On the other hand, let
$\Mod_{\super}(A)$ denote the category of left $A$-supermodules with
$\Z_{2}$-degree preserving homomorphisms. Then $\Mod_{\super}(A)$
has a structure of supercategory induced by the parity shift functor
$\Pi$.

For $\beta \in \rtl^{+}$, let $R(\beta)$ and $R^{\Lambda}(\beta)$ be
the quiver Hecke superalgebra and the cyclotomic quiver Hecke
superalgebra at $\beta$, respectively. In \cite{KKO12}, we dealt
with the supercategory $\Mod(R(\beta))$ and
$\Mod(R^{\Lambda}(\beta))$, not $\Mod_{\super}(R(\beta))$ and
$\Mod_{\super}(R^{\Lambda}(\beta))$. More precisely, let $(\car, \P,
\Pi, \Pi^{\vee})$ be a {\em Cartan superdatum}. That is, the index
set $I$ has a decomposition $I =\Iev \sqcup \Iod$ and $a_{ij} \in 2
\Z$ for $i \in \Iod$, $j \in I$. We denote by
$\text{Proj}(R(\beta))$ the supercategory of finitely generated
projective $\Z$-graded $R(\beta)$-modules and $\text{Rep}(R(\beta))$
the supercategory of $\Z$-graded $R(\beta)$-modules that are
finite-dimensional over $\k_0$. We define the supercategories
$\Mod(R^{\Lambda}(\beta))$, $\text{Proj}(R^{\Lambda}(\beta))$ and
$\text{Rep}(R^{\Lambda}(\beta))$ in a similar way. Consider the
supercategories
$$
\begin{aligned}
& \Rep(R^{\Lambda}) = \bigoplus_{\beta \in \rtl^{+}}
\Rep(R^{\Lambda}(\beta)), \quad \Proj(R^{\Lambda})= \bigoplus_{\beta
\in \rtl^{+}} \Proj(R^{\Lambda}(\beta)), \\
& \Rep(R) = \bigoplus_{\beta \in \rtl^{+}} \Rep(R(\beta)), \quad
\Proj(R)= \bigoplus_{\beta \in \rtl^{+}} \Proj(R(\beta)).
\end{aligned}
$$
In \cite{KKO12}, we proved that
$$
\begin{aligned}
& V_{\A}(\Lambda)^{\vee} \overset{\sim} \longrightarrow
[\Rep(R^{\Lambda})], \qquad V_{\A}(\Lambda) \overset{\sim}
\longrightarrow [\Proj(R^{\Lambda})], \\
& U_{\A}^{-}(\g)^{\vee} \overset{\sim} \longrightarrow [\Rep(R)],
\qquad U_{\A}^{-}(\g) \overset{\sim} \longrightarrow [\Proj(R)],
\end{aligned}
$$
where $[ \ \ \ ]$ denotes the Grothendieck group and
$V_{\A}(\Lambda)^{\vee}$ (resp. $U_{\A}^{-}(\g)^{\vee}$) is the dual
of $V_{\A}(\Lambda)$ (resp. $U_{\A}^{-}(\g)$).

The main theme of this paper is to investigate the structure of
supercategories
$$
\begin{aligned}
& \MOD(R^{\Lambda}) = \bigoplus_{\beta \in \rtl^{+}}
\MOD(R^{\Lambda}(\beta)), \quad \MOD(R)= \bigoplus_{\beta
\in \rtl^{+}} \MOD(R(\beta)), \\
& \REP(R^{\Lambda}) = \bigoplus_{\beta \in \Q^{+}}
\REP(R^{\Lambda}(\beta)), \quad  \REP(R)= \bigoplus_{\beta
\in \rtl^{+}} \REP(R(\beta)), \\
& \PROJ(R^{\Lambda}) = \bigoplus_{\beta \in \rtl^{+}}
\PROJ(R^{\Lambda}(\beta)), \quad \PROJ(R)= \bigoplus_{\beta \in
\rtl^{+}} \PROJ(R(\beta)).
\end{aligned}
$$
Here,  we denote by $\MOD(R(\beta))$ the supercategory of
$\Z$-graded $R(\beta)$-supermodules, by $\PROJ(R(\beta))$ the
supercategory of finitely generated projective
$R(\beta)$-supermodules and by $\REP(R(\beta))$ the supercategory of
$R(\beta)$-supermodules finite-dimensional over $\k_0$. We define
the supercategories $\MOD(R^{\Lambda}(\beta))$,
$\PROJ(R^{\Lambda}(\beta))$ and $\REP(R^{\Lambda}(\beta))$ in a
similar manner. The parity functor $\Pi$ induces involutions $\pi$
on the Grothendieck groups of these supercategories and we have
isomorphisms
$$
\begin{aligned}
& [\Rep(R^{\Lambda})] \overset{\sim} \longrightarrow
[\REP(R^{\Lambda})] \big/ (\pi -1) [\REP(R^{\Lambda})], \\
& [\Proj(R^{\Lambda})] \overset{\sim} \longrightarrow
[\PROJ(R^{\Lambda})] \big/ (\pi - 1) [\PROJ(R^{\Lambda})],\\
& [\Rep(R)] \overset{\sim} \longrightarrow
[\REP(R)] \big/ (\pi -1) [\REP(R)], \\
& [\Proj(R)] \overset{\sim} \longrightarrow [\PROJ(R)] \big/ (\pi -
1) [\PROJ(R)].
\end{aligned}
$$

Our goal is to prove that quiver Hecke superalgebras and cyclotomic
quiver Hecke superalgebras provide a supercategorification of a
certain family of quantum superalgebras and their integrable highest
weight modules. We will also show that, by taking a specialization,
we obtain a supercategorification of quantum Kac-Moody superalgebras
and their integrable highest weight modules. However, it is quite
delicate and needs some special care to present a precise statement
of our main theorem.

We first define the algebras $\FS$ and $\HS$ which are
generalizations of quantum Kac-Moody (super)algebras. Let
$\theta\seteq\{\theta_{ij}\}_{i,j\in I}$ and $\pa\seteq
(\{p_{ij}\}_{i,j\in I},\{p_i\}_{i\in I})$ be families of invertible
elements in $\cor$ such that $p_i^{n}-1$ is invertible for all $i\in
I$ and $n \in \Z_{>0}$. Assume that $\theta$ and $\pa$ satisfy the
condition \eqref{cond:pt}. We define $\FS$ to be the $\cor$-algebra
generated by $e_i$, $f_i$, $K_i^{\pm 1}$ with the defining relations
\eqref{def:genU1} and \eqref{eq:Serre}. We denote by
$\Mod^{\P}(\FS)$ the category of $\P$-weighted $\FS$-modules and
$\Oint^{\P}(\FS)$ the subcategory consisting of $\P$-weighted
integrable $\FS$-modules.

For each $i\in I$, choose a function $\chi_{i}: \P \rightarrow
\cor^{\times}$ satisfying \eqref{cond:chi}. The {\em Verma module}
$\MBe(\Lambda)$ is defined to be the $\FS$-module generated by a
vector $u_{\Lambda}$ with defining relations
$$K_i u_{\Lambda} = \chi_{i}(\Lambda) u_{\Lambda}, \quad e_{i}
u_{\Lambda}=0 \ \ \text{for all} \ i \in I.$$ We define
$\VBe(\Lambda) = \MBe(\Lambda) / \NBe(\Lambda)$, where
$\NBe(\Lambda)$ is the unique maximal $\FS$-submodule of
$\MBe(\Lambda)$ such that $\NBe(\Lambda) \cap \cor u_{\Lambda} =0$.
If $\Lambda \in \P^{+}$, then $\VBe(\Lambda)$ belongs to
$\Oint^{\P}(\FS)$ and $ f_i^{\la h_i, \Lambda \ra +1} v_{\Lambda} =0
\quad \text{for all} \ i \in I,$ where $v_{\Lambda}$ is the image of
$u_{\Lambda}$ in $\VBe(\Lambda)$. We conjecture that the category
$\Oint^{\P}(\FS)$ is semisimple and every simple object is
isomorphic to $\VBe(\Lambda)$ for some $\Lambda \in \P^{+}$. (See
Conjecture \ref{conj}.)

On the other hand, let $\tilde{\theta} = \{\theta_{ij}\}_{i,j \in
I}$ and $\tilde{\pa} = \{\tilde{p}_{i} \}_{i\in I}$ be families of
invertible elements in $\cor$ such that $1 - \tilde{p}_{i}^{n}$ is
invertible for all $i \in I$, $n \in \Z_{>0}$. Assume that
$\tilde{\theta}$ and $\tilde{\pa}$ satisfy the condition
\eqref{cond:ttp}. We define $\HS$ to be the $\cor$-algebra generated
by $e_i$, $f_i$, $\tilde{K}_{i}^{\pm 1}$ with defining relations
\eqref{def:genU2} and \eqref{eq:Serre2}. Assume that $\theta$,
$\pa$, $\tilde{\theta}$ and $\tilde{\pa}$ satisfy the relation
\eqref{rel:ptt}. Then we have the following equivalences of
categories (Proposition \ref{Prop: equivalent 1}):
$$\Mod^{\P}(\FS) \overset{\sim} \longrightarrow \Mod^{\P}(\HS),
\quad \Oint^{\P}(\FS) \overset{\sim}\longrightarrow
\Oint^{\P}(\HS).$$ Moreover these categories only depend on
$\{p_i^2\}_{i\in I}$.

The algebras $\FSm$ and $\HSm$ have a structure of $\Bg$-module,
where $\FSm$ (resp. $\HSm$) is the subalgebra of $\FS$ (resp. $\HS$)
generated by $f_i$'s $(i \in I)$ and $\Bg$ is the {\em quantum boson
algebra} (see Definition \ref{def:BqBg}).

For a Cartan superdatum $(A, \P, \Pi, \Pi^{\vee})$, we define the
{\em parity function} ${\rm p}: I \rightarrow \{0, 1\}$ by ${\rm
p}(i)=0$ if $i$ is even, \ ${\rm p}(i)=1$ if $i$ is odd. We denote
by $\Pev = \{ \lambda \in \P \mid \la h_i, \lambda \ra \in 2 \Z \ \
\text{for} \ i \in \Iod \}$ and set $\Pev^{+} = \P^{+} \cap \Pev$.

Let $\pi$ (resp. $\sqrt{\pi}$) be an indeterminate such that $\pi^2
=1$ (resp. $(\sqrt{\pi})^2 = \pi$). For any ring $R$, we define
$$R^{\pi} = R \otimes \Z[\pi], \qquad R^{\sqrt{\pi}} = R \otimes
\Z[\sqrt{\pi}].$$

Set $\cor=\Q(q)^{\sqrt{\pi}}$ and choose $\theta$ and $\pa$
satisfying \eqref{eq:KMsuper}:
$$p_i = q_i \sqrt{\pi_{i}}, \ \ p_{ij} = q_i^{a_{ij}}, \ \
\theta_{ij}\theta_{ji} = 1, \ \ \theta_{ii}=\pi_{i}.$$ Let ${\rm
U}_{\theta}^q(\g) = \FS$ and $V_{\theta}^q(\Lambda)=\VBe(\Lambda)$
for this choice of $\theta$ and $\pa$. The algebra ${\rm
U}_{\theta}^q(\g)$ is the {\em quantum Kac-Moody superalgebra}
introduced by \cite{KT91, BKM98}. It was shown in \cite{BKM98} that
the category $\Oint^{\Pev}(\C(q) \otimes_{\Q(q)} {\rm
U}_{\theta}^{q}(\g))$ is semisimple and every simple object is
isomorphic to $V_{\theta}^{q}(\Lambda) \big/ (\sqrt{\pi} -c)$ for
some $\Lambda \in \Pev$ and $c \in \C$ with $c^{4}=1$.

The parameter $\pi$ was first introduced by Hill and Wang
\cite{HW12}. Using this, they defined the notion of {\it covering
Kac-Moody algebras} which specialize to Kac-Moody algebras when
$\pi=1$ and to Kac-Moody superalgebras when $\pi=-1$. The discovery
of $\pi$ is a simple but an important observation because it
explains the subtle behavior of the parity functor  $\Pi$. In this
sense,  $\Pi$ gives a categorification of $\pi$. 

Now we take another choice of $\theta$ and $\pa$ given in
\eqref{eq:boldU}:
$$p_i = q_i \sqrt{\pi_{i}}, \quad p_{ij} = p_i^{a_{ij}}, \quad
\theta_{ij} = \begin{cases}  \sqrt{\pi_{j}}^{a_{ji}} \ \ &\text{if}
\ i\neq j, \\ 1 \ \ & \text{if} \ i=j.
\end{cases}$$
We denote by  $\Uqsug = U_{\theta, \pa}(\g)$ and
$\mathbf{V}(\Lambda)=V_{\theta, \pa}(\Lambda)$ for this choice. We
prove in Corollary \ref{cor: tri dec Super} and in Theorem \ref{thm:
ch V(Lambda)} that

\ \ (i) We have the equivalence of categories
$$\Mod^{\P} \Uqsug \overset{\sim} \longrightarrow \Mod^{\P} U_{\theta}^q
(\g).$$

\ \ (ii) The category $\Oint^{\P}(\Uqsug)$ is semisimple and every
simple object is isomorphic to ${\mathbf V}(\Lambda)$ for some
$\Lambda \in \P^{+}$.

The key ingredient of the proof is the {\em quantum Casimir
operator} for the quantum superalgebra $\Uqsug$ (See Section
\ref{sec: Rep Uqsug}).

We finally define $\Uqsg$ to be the $\cor$-algebra $\HS$ with
$\tilde{\theta}$, $\tilde{\pa}$ and $\cor$ given in \eqref{eq:Uqsg}:
$$\tilde{p}_i = q_i^2 \pi_{i}, \quad \tilde{\theta}_{ij} =
\tilde{\theta}_{ji} = \pi^{{\rm p}(i) {\rm p}(j)}
q_i^{-a_{ij}},\quad \cor = \Q(q)^{\pi}.$$ For $\Lambda \in \P^{+}$,
let $\Vs(\Lambda)$ be the $\P$-weighted $\Uqsg$-module generated by
$v_{\Lambda}$ with defining relations \eqref{def:V}:
$$\tilde{K}_{i} v_{\Lambda} = (q_i^2 \pi_{i})^{\la h_i, \Lambda \ra}
v_{\Lambda} = \tilde{p}_i^{\la h_i, \Lambda \ra} v_{\Lambda}, \quad
e_i v_{\Lambda}=0, \quad f_i^{\la h_i, \Lambda \ra +1} v_{\Lambda}
=0.$$ Then we prove in Theorem \ref{thm:Oint} that

\ \ (i) We have the equivalences of categories
$$
\begin{aligned}
& \Mod^{\P}(\Q(q)^{\sqrt{\pi}} \otimes_{\Q(q)^{\pi}} \Uqsg)
\overset{\sim} \longrightarrow \Mod^{\P}(\Uqsug), \\
& \Oint^{\P}(\Q(q)^{\sqrt{\pi}} \otimes_{\Q(q)^{\pi}} \Uqsg)
\overset{\sim} \longrightarrow \Oint^{\P}(\Uqsug).
\end{aligned}
$$

\ \ (ii) The category $\Oint^{\P}(\Uqsg)$ is semisimple and every
simple object is isomorphic to $\Vs(\Lambda) \big/ (\pi -
\varepsilon) \Vs(\Lambda)$ for some $\Lambda \in \P^{+}$ and
$\varepsilon = \pm 1$.

The algebra $\Uqsg$ and the $\Uqsg$-module $\Vs(\Lambda)$ are
directly related to the supercategorification via quiver Hecke
superalgebras and cyclotomic quiver Hecke superalgebras. We denote
by $\Us_{\A^{\pi}}(\g)$ and $\Vs_{\A^{\pi}}(\Lambda)$ the
$\A^{\pi}$-forms of $\Uqsg$ and $\Vs(\Lambda)$, respectively, where
$\A^{\pi}=\Z[q, q^{-1}]^{\pi} \subset \Q(q)^{\pi}$. Also, we denote
by $B_{\A^{\pi}}^{\rm up}(\g)$ and $B_{\A^{\pi}}^{\rm low}(\g)$ the
upper and lower $\A^{\pi}$-forms of the quantum boson algebra
$B_{\tilde{\theta}, \tilde{\pa}}(\g)$.

Now we can state our supercategorification  theorems (Theorem
\ref{th:main1} and Corollary \ref{cor:main2}):

\ \ (a) There exist isomorphisms of $\Us_{\A^{\pi}}(\g)$-modules
$$\Vs_{\A^{\pi}}(\Lambda)^{\vee} \overset{\sim} \longrightarrow
[\REP(R^{\Lambda})], \qquad \Vs_{\A^{\pi}}(\Lambda)
\overset{\sim}\longrightarrow [\PROJ(R^{\Lambda})].$$

\ \ (b) There exist isomorphisms
$$\Us_{\A^{\pi}}^{-}(\g)^{\vee} \overset{\sim}\longrightarrow [\REP(R)],
\qquad \Us_{\A^{\pi}}^{-}(\g) \overset{\sim}\longrightarrow
[\PROJ(R)]$$ as $B_{\A^{\pi}}^{\rm up}(\g)$-modules and
$B_{\A^{\pi}}^{\rm low}(\g)$-modules, respectively.

To prove our main theorems, for each $i\in I$ and $\beta \in
\rtl^{+}$, we define the superfunctors
\begin{align*}
& E_i^{\Lambda}\cl \MOD(R^{\Lambda}(\beta+\alpha_i)) \to \MOD(R^{\Lambda}(\beta)), \\
& F_i^{\Lambda}\cl \MOD(R^{\Lambda}(\beta)) \to
\MOD(R^{\Lambda}(\beta+\alpha_i))
\end{align*}
by
\begin{align*}
& E_i^{\Lambda}(N)=e(\beta,i)N
=e(\beta,i)R^{\Lambda}(\beta+\alpha_i)
\otimes_{R^{\Lambda}(\beta+\alpha_i)}N,\\
& F_i^{\Lambda}(M)=R^{\Lambda}(\beta+\alpha_i)e(\beta,i)
\otimes_{R^{\Lambda}(\beta)}M
\end{align*}
for $M \in \MOD(R^{\Lambda}(\beta))$ and $N\in
\MOD(R^{\Lambda}(\beta+\alpha_i))$. By the same argument as in
\cite{KKO12}, one can verify that $E_{i}^{\Lambda}$ and
$F_{i}^{\Lambda}$ are well-defined exact functors on
$\REP(R^{\Lambda})$ and $\PROJ(R^{\Lambda})$. Similarly, one can
show that there exist natural isomorphisms of endofunctors on
$\MOD(R^{\Lambda}(\beta))$ given below:
\begin{equation*}
\begin{aligned}
& E^{\Lambda}_iF^{\Lambda}_j \overset{\sim}{\to}
 q^{-(\alpha_i|\alpha_j)}\Pi^{\pa(i)\pa(j)}F^{\Lambda}_j E^{\Lambda}_i \quad
  \text{ if $i \neq j$}, \\
& \Pi_iq_i^{-2}F^{\Lambda}_iE^{\Lambda}_i \oplus \bigoplus^{\langle
h_i,\Lambda-\beta \rangle-1}_{k=0}\Pi_i^k q_i^{2k}
\overset{\sim}{\to} E^{\Lambda}_iF^{\Lambda}_i \quad
\text{ if $\langle h_i,\Lambda-\beta \rangle \ge 0$}, \\
& \Pi_iq_i^{-2}F^{\Lambda}_iE^{\Lambda}_i \overset{\sim}{\to}
E^{\Lambda}_iF^{\Lambda}_i \oplus \bigoplus^{-\langle
h_i,\Lambda-\beta \rangle-1}_{k=0}\Pi_i^{k+1}q_i^{-2k-2} \quad
\text{ if $\langle h_i,\Lambda-\beta \rangle < 0$}.
\end{aligned}
\end{equation*}
It follows that $[\REP(R^{\Lambda})]$ and $[\PROJ(R^{\Lambda})]$ are
endowed with $\Us_{\A^{\pi}}(\g)$-module structure. Moreover, using
the characterization theorem of $\Vs_{\A^{\pi}}(\Lambda)^{\vee}$ in
terms of strong perfect bases (Theorem \ref{Thm:recognition thm}),
we conclude that
$$\Vs_{\A^{\pi}}(\Lambda)^{\vee} \overset{\sim}\longrightarrow
[\REP(R^{\Lambda})].$$ The rest of our statements follow by duality
and by taking inductive limit.

When the Cartan superdatum satisfies the (C6) condition proposed by
\cite{HW12}: ${\rm d}_i$ is odd if and only of $i \in \Iod$, we have
$$\Mod^{\P}(\Uqsg) \overset{\sim}\longrightarrow \Mod^{\P}(\Uqg),$$
where $\Uqg$ is the usual quantum Kac-Moody algebra with $v = q
\sqrt{\pi}$. Hence the results in \cite{HW12} follow as a special
case of our supercategorification theorems.

\vskip 5mm

{\bf Acknowledgements.} We would like to express our gratitude to
Sabin Cautis for fruitful correspondences.

\vskip 5mm

\section{Preliminaries}\label{sec:cartan}

Let $I$ be an index set.
 An integral matrix
$\car=(a_{ij})_{i,j \in I}$ is called a {\em Cartan matrix}  if it
satisfies: (i) $a_{ii}=2$, \  (ii) $a_{ij} \le 0$ for $i \neq j$, \
(iii) $a_{ij} =0 $ if $a_{ji}=0$. We say that  $\car$ is {\em
symmetrizable}
 if there is a diagonal matrix  ${\rm D}={\rm diag}(\dg_i \in \Z_{>0} \ | \ i \in I)$
such that ${\rm D} \car$ is symmetric.

\begin{definition} \label{dfn:cartan datum}
A {\em Cartan datum} is a quadruple $(\car,\P,\Pi,\Pi^\vee)$
consisting of \bnum
\item a symmetrizable Cartan matrix $\car$,
\item a free abelian group $\P$, called the {\em weight lattice},
\item $\Pi=\{ \alpha_i \in \P \ | \ i \in I \}$, called the set of
{\em simple roots},
\item $\Pi^\vee=\{ h_i \ | \ i \in I \} \subset \P^\vee \seteq \Hom(\P,\Z)$,
called the set of {\em simple coroots},
\end{enumerate}
satisfying the following properties:
\bna
\item $\la h_i,\alpha_j \ra=a_{ij}$ for all $i,j \in I$, \label{eqn:root pairing}
\item $\Pi$ is linearly independent.
\end{enumerate}
\end{definition}

The weight lattice $\P$ has a symmetric bilinear form $( \ | \ )$ satisfying
$$ (\alpha_i | \lambda ) = \dg_i \la h_i , \lambda \ra \quad \text{ for all } \lambda \in \P.$$
In particular, we have $(\alpha_i | \alpha_j)=\dg_i a_{ij}$. Let
$\P^+ \seteq \{ \lambda \in \P \ | \ \la h_i,\lambda\ra\in\Z_{\ge 0}
\text{ for all } i \in I \}$ be the set of {\em dominant integral
weights}. The free abelian group $\rtl \seteq \oplus_{i \in
I}\Z\alpha_i$ is called the {\em root lattice}. Set $\rtl^+=\sum_{i
\in I} \Z_{\ge 0} \alpha_i$ and $\rtl^- = -\rtl^+$. For $\beta =
\sum k_i \alpha_i \in \rtl$, the {\em height} of $\beta$ is defined
to be $|\beta| = \sum |k_i|$.  For each $i \in I$, let $s_i \in
\mathrm{GL}(\P)$ be the {\em simple reflection} on $\P$ defined by
$s_i(\lambda)=\lambda - \la h_i,\lambda \ra \alpha_i$ for $\lambda
\in \P$. The subgroup $W$ of $\mathrm{GL}(\P)$ generated by $s_i$ is
called the {\em Weyl group} associated with the Cartan datum
$(\car,\P,\Pi,\Pi^\vee)$.

\begin{definition}[\cite{Kac90}] \label{dfn:KM}
The Kac-Moody Lie algebra $\g$ associated with the Cartan datum
$(\car,\P,\Pi,\Pi^\vee)$ is the Lie algebra over $\Q$ generated by
$\gt\seteq\Q\otimes P^\vee$ and $e_i$, $f_i$ $(i\in I)$ satisfying
the following defining relations: \bnum
\item $\gt$ is abelian,
\item $[h,e_i]=\lan h,\al_i\ran e_i$, $[h,f_i]=-\lan h,\al_i\ran f_i$,
\item $[e_i,f_j]=\delta_{i,j}h_i$,
\item $\ad(e_i)^{1-a_{ij}}e_j=0$, \ $\ad(f_i)^{1-a_{ij}}f_j=0$
for any $i\not=j\in I$.
\end{enumerate}
\end{definition}
Then $\g$ has the root space decomposition:
$\g=\soplus_{\beta\in Q}\g_\beta$, where
$$\g_\beta=\set{a\in\g}{\text{$[h,a]=\lan h,\beta\ran a$ for any
$h\in\gt$}}.$$ We denote by
\begin{itemize}
\item[(i)]
$\Delta\seteq\set{\beta\in \rtl\setminus\{0\}}{\g_\beta\not=0}$,
the set of {\em roots} of $\g$,
\item[(ii)] $\Delta^\pm\seteq\Delta\cap\rtl^\pm$, the set of {\em positive roots}
(resp.\ {\em negative roots}) of $\g$,
\item[(iii)] $\mult(\beta) \seteq \dim \g_{\beta}$, the {\em multiplicity}
of the root $\beta$.
\end{itemize}

Let $\cor$ be a commutative ring which will play the role of base
ring. In this paper, we will deal with several associative
$\cor$-algebras $\mathcal{A}$ generated by $e_i$, $f_i$ , $K_i^{\pm
1}$ ($i \in I$) satisfying the relations
$$K_i e_j K_i^{-1} = g_{i}^{a_{ij}}e_j, \quad  K_i f_j K_i^{-1} =
g_i^{-a_{ij}}f_j $$ for some invertible elements $g_{i}$ in $\cor$.

We say that $\mathcal{A}$ has a {\em weight space decomposition} if
it is endowed with a decomposition
$$\mathcal{A} = \bigoplus_{\alpha \in \rtl} \mathcal{A}_{\alpha}$$
such that
$e_i\mathcal{A}_{\alpha}+\mathcal{A}_\alpha e_i\subset \mathcal{A}_{\alpha+\al_i}$,
$f_i\mathcal{A}_{\alpha}+\mathcal{A}_\alpha f_i\subset \mathcal{A}_{\alpha-\al_i}$
and $K_iaK_i^{-1}=g_{i}^{\lan h_i,\al\ran}a$
for any $\al\in\rtl$ and $a\in\mathcal{A}_\al$.

Let $G$ be a subset of $\P$ such that $G+\rtl\subset \P$. An
$\mathcal{A}$-module $V$ is called a {\em $G$-weighted module} if it
is endowed with a {\em $G$-weight space decomposition}
$$ V= \bigoplus_{\mu \in G}V_\mu $$
such that $\mathcal{A}_\alpha V_\mu\subset V_{\mu+\alpha}$, and
$K_i\vert_{V_\mu}=g_i^{\lan h_i,\mu\ran}\id_{V_\mu}$ for any $\al\in
\rtl$ and $\mu\in G$. A vector $v \in V_{\mu}$ is called a {\em
weight vector} of weight $\mu$. We denote the set of weights of $V$
by $\wt(V) \seteq \set{ \mu \in G}{\ V_{\mu} \neq 0 }$.

We call an $\mathcal{A}$-module  $M$ a {\em highest weight module}
with highest weight $\Lambda$ if $M$ is $(\Lambda+\rtl)$-weighted
module and there exists a vector $v_\Lambda \in M_{\Lambda}$ (called
a {\em highest weight vector}) such that
\begin{equation} \label{dfn: highest weight module}
\  M = \mathcal{A} \ v_\Lambda, \quad e_iv_\Lambda =0 \ \text{ for
all } i\in I.
\end{equation}
An $\mathcal{A}$-module $M(\Lambda)$ with highest weight $\Lambda
\in \P$ is called an $\mathcal{A}$-{\em Verma} module if every
$\mathcal{A}$-module with highest weight $\Lambda$ is a quotient of
$M(\Lambda)$.

\vskip 3mm

For later use, we fix some notations.

\bnum
\item We denote by $\Mod^{G}(\mathcal{A})$ the abelian category of
$G$-weighted $\mathcal{A}$-modules $V$.
\item We denote by $\Ocat[]^{G}(\mathcal{A})$
the full subcategory of $\Mod^{G}(\mathcal{A})$
consisting of $G$-weighted $\mathcal{A}$-modules $V$ satisfying the
following conditions:
\bna
\item $\dim V_\lam<\infty$ for any $\lam\in G$,
\item there are finitely many $\lambda_1, \ldots, \lambda_s \in
G$ such that $\wt(V) \subset \bigcup^{s}_{i=1}(\lambda_i -
\rtl^{+}).$ \ee
\item We denote by $\Oint^{G}(\mathcal{A})$
the full subcategory of $\Ocat[]^{G}(\mathcal{A})$
 consisting of the modules $V$ satisfying the additional condition:
\be
\item[{\rm(c)}] For any $i \in I$, the actions of $e_i$ and $f_i$ on $V$
are locally nilpotent.
\ee
\ee

\begin{definition} \label{Def: O, Oint w.r.t G} \hfill
\bna
\item
 We say that an $\mathcal{A}$-module is \emph {integrable}
if it belongs to the category $\Ocat^\P(\mathcal{A})$.
\item For $V\in \Ocat[]^\P(\mathcal{A})$, we define its \emph{character} by
$$\ch(V)=\sum_{\lam\in \P}(\dim V_\lam)\ex{\lam}.$$
\ee
\end{definition}

Let $R$ be a ring and let $\{ X_{j}^{\pm 1} \mid j \in J \}$ be a
family of commuting variables. Set $$R[X_{j}^{\pm 1} \mid j \in J ]
= R \otimes_{\Z} Z [X_{j}^{\pm 1} \mid j \in J ].$$ Then the
following lemma is obvious.

\Lemma  \ {\rm (a)} Let $\{\varphi_{j} \mid j \in J \}$ be a family
of commuting automorphisms of $R$. Then $R[X_{j}^{\pm 1} \mid j \in
J]$ has a ring structure given by
$$X_{j}^{\pm 1} \, a = \varphi_{j}^{\pm 1}(a) \,
 X_{j}^{\pm 1} \quad (a
\in R, \ j \in J).$$

{\rm (b)} If $J' \subset J$ and $\varphi_{j}^2 = \id$ for all $j \in
J'$, then we may assume that $X_{j}^2 =1$ for all $j \in J'$.
\enlemma

In this case, we say that $R[X_{j}^{\pm 1} \mid j \in J ]$ is
obtained from $R$ by adding the mutually commuting operators
satisfying
$$X_{j} \, a \, X_{j}^{-1} = \varphi_{j}(a) \quad (a \in R, j \in
J).$$

For $a$, $b\in\cor$
and $n \in \Z_{\ge 0}$ , we define
\begin{equation}
 \begin{aligned}
 \ &[n]_{a,b} =\frac{ a^n - b^{n} }{ a -b },
 \ &{}[n]_{a,b}! = \prod^{n}_{k=1} [k]_{a,b} ,
 \ & {\binm{m}{n}}_{a,b}= \frac{ [m]_{a,b}! }{[m-n]_{a,b}!\; [n]_{a,b}! }.
 \end{aligned}\label{def:qint}
\end{equation}
Note that they are polynomials of $a$ and $b$. Moreover, we have \eq
&&[n]_{ac,bc} =c^{n-1}[a]_{a,b}, \quad[n]_{ac,bc}!=c^{n(n-1)/2}
[n]_{a,b}!,
 \quad{\binm{m}{n}}_{ac,bc}= c^{n(m-n)}{\binm{m}{n}}_{a,b},\\
&& \sum_{k=0}^{n}{\binm{n}{k}}_{a,b}(ab)^{\frac{k(k-1)}{2}}z^k
=\prod_{k=0}^{n-1}(1+a^{n-1-k}b^kz). \label{eq:bino}
\eneq

\vskip 5mm

\section{The algebra $\FS$}
\label{sec:FS}

Let $\theta\seteq\{\theta_{ij}\}_{i,j\in I}$ and $\pa\seteq
(\{p_{ij}\}_{i,j\in I},\{p_i\}_{i\in I})$ be families of invertible
elements of a commutative ring $\cor$ such that $p_i^n-1$ is
invertible for any $i$ and $n\in\Z_{>0}$. Define $\F$ to be the
$\cor$-algebra generated by $e_i$, $f_i$, $K_i^{\pm1}$ ($i\in I$)
with the defining relations \eq &&\ba{l}
K_iK_j=K_jK_i, \quad  K_i e_j K_i^{-1}=p_{ij}e_j,\quad K_i f_j K_i^{-1}=p_{ij}^{-1}f_j,\\[1ex]
e_if_j-\theta_{ji}f_je_i=\delta_{i,j}\dfrac{K_i-K_i^{-1}}{p_i-p_i^{-1}}.
\ea\label{def:genU1} \eneq Then there exists an anti-isomorphism \eq
&&\F\isoto\F[{\trp{\theta},p}] \label{eq:antiiso} \eneq given by
$$e_i\mapsto f_i, \quad f_i\mapsto e_i,\quad  K_i\mapsto K_i \quad (i \in I),$$
where $(\trp{\theta})_{ij}=\theta_{ji}$.

Let us denote by $\Fm$ be the subalgebra of $\F$ generated by the
$f_i$'s $(i \in I)$. Then $\Fm$ is a free $\cor$-algebra with $\{f_i
\mid i \in I \}$ as generators. Similarly, let $\Fp$ be the
subalgebra generated by the $e_i$'s $(i \in I)$ and set
$\mathcal{F}^{0}=\cor[K_{i}^{\pm 1} \mid i \in I]$. Then we have a
triangular decomposition
\eq&&\Fm\tens\cor[K_i^{\pm1}\mid i\in
I]\tens\Fp\isoto\F.\eneq

We will investigate the role of $\theta$ and $\pa$ in characterizing
the algebra $\F$. Let $\theta'$ and $\pa'$ be another choice of such
families and consider the algebra ${\mathcal F}(\theta', \pa')$. We
take a set of invertible elements $x_{ij}$,$y_{ij}$,
$\varepsilon_{ij}$, $c_i$ in $\cor$ and let ${\mathcal F}(\theta,
\pa)[P,Q,R]$ (resp. ${\mathcal F}(\theta', \pa')[P, Q, R]$) be the
algebra obtained from $\F$ (resp. ${\mathcal F}(\theta', \pa')$) by
adding mutually commuting operators $P=(P_{i}^{\pm 1})$,
$Q=(Q_{i}^{\pm 1})$, $R=(R_{ij}^{\pm 1})$ satisfying
\eq
&&\ba{l}
P_ie_jP_i^{-1}=x_{ij}e_j, \ P_if_jP_i^{-1}=x_{ij}^{-1}f_j,
\ P_iK_jP_i^{-1}=K_j,\\[1ex]
Q_ie_jQ_i^{-1}=y_{ij}e_j, \ Q_if_jQ_i^{-1}=y_{ij}^{-1}f_j,
\ Q_iK_jQ_i^{-1}=K_j,\\[1ex]
R_ie_jR_i^{-1}=\eps_{ij}e_j, \ R_if_jR_i^{-1}=\eps_{ij}^{-1}f_j, \
R_iK_jR_i^{-1}=K_j, \\[1ex]
 x_{ij} y_{ij} = \eps_{ij}, \ \eps_{ij}^2 =1, \ P_i Q_i = c_i R_i,
\ R_i^2 =1.
\ea\label{eq:PQR}
\eneq

\Prop\label{prop:guage}
Assume that \eq
&&\theta'_{ij}=\eps_{ij}x_{ji}x_{ij}^{-1}\theta_{ij}=
\eps_{ji}y_{ij}y_{ji}^{-1}\theta_{ij},\quad p'_{ij}=\eps_{ij}p_{ij},
\quad c_i=x_{ii}\dfrac{p'_i-p'_i{}^{-1}}{p_i-p_i^{-1}}.\label{eq:xyc}
\eneq Then
there exists a $\cor$-algebra isomorphism
$$\kappa\cl\F{[P,Q,R]}\isoto\F[\theta',\pa']{[P,Q,R]}$$
given by
\eq
&&e_i\mapsto e_iP_i,\ f_i\mapsto f_iQ_i,\ K_i\mapsto K_iR_i.
\label{def:kappa}
\eneq
\enprop \Proof We have \eqn \kappa(e_if_j-\theta_{ji}f_je_i)
&=&e_iP_if_jQ_j-\theta_{ji}f_jQ_je_iP_i\\
&=&\bl x_{ij}^{-1}e_if_j-\theta_{ji}y_{ji}f_je_i\br P_iQ_j
\eneqn
Since $\theta'_{ji}=x_{ij}y_{ji}\theta_{ji}$,
it is equal to
\eqn
x_{ij}^{-1}\bl e_if_j-\theta'_{ji}f_je_i\br P_iQ_j
&=&\delta_{i,j}x_{ii}^{-1}\dfrac{K_i-K_i^{-1}}{p'_i-p_i'{}^{-1}}P_iQ_i\\
&=&\delta_{i,j}x_{ii}^{-1}c_i\dfrac{
K_iR_i-(K_iR_i)^{-1}}{p'_i-p'_i{}^{-1}}
=\kappa\Bigl(\delta_{i,j}\dfrac{K_i-K_i^{-1}}{p_i-p_i^{-1}}\Bigr).
\eneqn The other relations can be easily checked. \QED Hence we
obtain the following corollary.

\Cor\label{cor:Fequ} Suppose we have
\eq {p'_{ij}}^2=p_{ij}^2,\quad
(p'_{ij}p'_{ji})/(\theta'_{ij}\theta'_{ji})=
(p_{ij}p_{ji})/(\theta_{ij}\theta_{ji}), \quad
p'_{ii}/\theta'_{ii}=p_{ii}/\theta_{ii}.\label{eq:twop}
\eneq Then there exists a
$\cor$-algebra isomorphism
$$\kappa:\F{[P,Q,R]}\isoto\F[\theta',\pa']{[P,Q,R]}$$
for some choice of $x_{ij},\;y_{ij},\;\eps_{ij},\;c_i \ (i, j \in
I)$. \encor

Now let us investigate the conditions under which the Serre type
relations \eqn &&\sum_{k=0}^{n}x_k f_i^\px{n-k}f_jf_i^\px{k}=0
\eneqn can be added to the defining relations \eqref{def:genU1}.
Here,
\begin{equation}
\begin{aligned}
& \qintp{n}=[n]_{p_i,p_i^{-1}}, \quad
\factp{n}={[n]}_{p_i,p_i^{-1}}!, \quad & \binmp{n}{m} =
\frac{\factp{n}}{\factp{m} \, \factp{n-m}}, \\
& e_i^\px{n}=e_i^n/{\factp{n}},\quad f_i^\px{n}=f_i^n/{\factp{n}}.
\end{aligned}
\end{equation}

Assume for a while that \eq &&\text{$\theta_{ii}=1$ and
$p_{{ii}}=p_i^2$.} \eneq For $i,j\in I$ with $i\not=j$, let \eqn
&&S_{ij}\seteq\sum_{m=0}^{n_{ij}}x_{ij,m}
f_i^\px{n_{ij}-m}f_jf_i^\px{m} \eneqn for some $n_{ij}\in\Z_{>0}$
and $x_{ij,m}\in\cor$. We shall investigate the conditions under
which $S_{ij}$ satisfies: $e_kS_{ij}\in\F e_k$ for any $k\in I.$ It
is obvious that $e_kS_{ij}\in \F e_k$ for any $k$ such that
$k\not=i,j$. Set
$$\{x\}_i^\pa=(x-x^{-1})/(p_i-p_i^{-1}).$$
 Then we have
\eqn e_if_i^\px{n}=f_i^\px{n}e_i+f_i^\px{n-1}\{p_i^{1-n}K_i\}_i^\pa.
\eneqn It follows that \eqn & & e_iS_{ij}=
\sum_{m=0}^{n_{ij}}x_{ij,m}\bl
f_i^\px{n_{ij}-m}e_i+f_i^\px{n_{ij}-m-1}
\{p_i^{1-n_{ij}+m}K_i\}^\pa_i\br
f_jf_i^\px{m}\\
&=&\sum_{m=0}^{n_{ij}}x_{ij,m}\theta_{ji}f_i^\px{n_{ij}-m}f_j
\bl f_i^\px{m}e_i+f_i^\px{m-1}\{p_i^{1-m}K_i\}^\pa_i\br\\
&&\hs{15ex}+\sum_{m=0}^{n_{ij}}x_{ij,m}
f_i^\px{n_{ij}-m-1}f_jf_i^\px{m}
\{p_i^{1-n_{ij}+m}p_{ij}^{-1}p_i^{-2m}K_i\}^\pa_i\\
&=&\theta_{ji}S_{ij}e_i
+\sum_{m=0}^{n_{ij}-1}f_i^\px{n_{ij}-m-1}f_jf_i^\px{m}
\Bigl(x_{ij,m+1}\theta_{ji}\{p_i^{-m}K_i\}^\pa_i+
x_{ij,m}\{p_i^{1-n_{ij}+m}p_{ij}^{-1}p_i^{-2m}K_i\}^\pa_i\Bigr).
\eneqn Comparing the coefficients of $K_i^{\pm1}$, we see that
$e_iS_{ij}\in\F e_i$ if and only if
$$x_{ij,m+1}\theta_{ji}(p_i^{-m})^{\pm1}
+x_{ij,m}(p_i^{1-n_{ij}-m}p_{ij}^{-1})^{\pm1}=0 \quad\text{for $0\le
m\le n_{ij}$.}$$ Hence we obtain \eqn x_{ij,m+1}
&=&-\theta_{ji}^{-1}(p_i^{1-n_{ij}}p_{ij}^{-1})^{\pm1}x_{ij,m}.\eneqn
Set  $x_{ij,0}=1$. Then we have
$$(p_i^{1-n_{ij}}p_{ij}^{-1})^2=1, \quad
p_{ij}=c_{ij}p_i^{1-n_{ij}},\quad
x_{ij,m}=(-c_{ij}\theta_{ji}^{-1})^m,\quad c_{ij}^2=1, $$ which
yields
$$S_{ij}=\sum_{m=0}^{n_{ij}}(-c_{ij}\theta_{ji}^{-1})^mf_i^\px{n_{ij}-m}f_jf_i^ \px{m}.$$
Thus we have \eqn e_jS_{ij}&=&
\sum_{m=0}^{n_{ij}}(-c_{ij}\theta_{ji}^{-1})^m
\theta_{ij}^{n_{ij}-m}f_i^\px{n_{ij}-m}(f_je_j+\{K_j\}^\pa_j)f_i^\px{m}\\
&=&\theta_{ij}^{n_{ij}}S_{ij}+
f_i^\px{n_{ij}}\Bigl(\sum_{m=0}^{n_{ij}}(-c_{ij}\theta_{ji}^{-1})^m
\binmp{n_{ij}}{m}
\theta_{ij}^{n_{ij}-m}\{p_{ji}^{-m}K_j\}^\pa_j\Bigr).
\eneqn
Hence the following quantity vanishes for $\eps=\pm1$:
\eqn
\sum_{m=0}^{n_{ij}}(-c_{ij}\theta_{ji}^{-1})^m\binmp{n_{ij}}{m}
\theta_{ij}^{n_{ij}-m}p_{ji}{}^{\eps\, m}
&=&\theta_{ij}^{n_{ij}}\sum_{m=0}^{n_{ij}}\binmp{n_{ij}}{m}
\bl -c_{ij}\theta_{ji}^{-1}\theta_{ij}^{-1}p_{ji}{}^{\eps}\br^m\\
&=&\prod_{k=0}^{n_{ij}-1} \bl
1-p_i^{1-n_{ij}+2k}c_{ij}\theta_{ji}^{-1}\theta_{ij}^{-1}p_{ji}{}^{\eps}\br.
\eneqn Here, the last equality follows from \eqref{eq:bino}.

Therefore  there exist $\ell_\eps$ with $|\ell_\eps|<n_{ij}$
satisfying
$$\ell_\eps\equiv n_{ij}-1\mod2, \quad
p_{ji}=\bl\theta_{ji}\theta_{ij}c_{ij}\br ^\eps p_{i}^{\ell_\eps}.$$
Hence $(p_{ji})^2=p_i^{\ell_++\ell_-}$ which implies
$p_{ji}=d_{ij}p_i^{\ell_{ij}}$, where
$\ell_{ij}=(\ell_++\ell_-)/2\in\Z$ and $d_{ij}^2=1$. Then we have
$\theta_{ij}\theta_{ji}=c_{ij}d_{ij}p_i^{\ell'_{ij}}$ for some
$\ell'_{ij}$. Since
$p_i^{\ell_{ij}}=p_i^{\eps\,\ell'_{ij}}p_i^{\ell_\eps}$, we have
$\ell_\eps=\ell_{ij}-\eps\ell_{ij}'$. Thus we obtain \eqn
\parbox{70ex}
{$p_{ji}=d_{ij}p_i^{\ell_{ij}}$,
$\theta_{ij}\theta_{ji}=c_{ij}d_{ij}p_i^{\ell'_{ij}}$
with $d_{ij}^2=1$, $|l_{ij}|+|l_{ij}'|\le n_{ij}-1$, $\ell_{ij}+\ell_{ij}'\equiv n_{ij}-1\mod2$}
\eneqn

As its solution, we take \eq &&n_{ij}=1-a_{ij},\quad
\ell_{ij}=n_{ij},\quad \ell'_{ij}=0. \eneq With this choice, we have
$$p_{ij}=c_{ij}p_i^{a_{ij}}=d_{ji}p_j^{a_{ji}}, \quad
\theta_{ij}\theta_{ji}=c_{ij}d_{ji}.$$ Hence,  together with
$\theta_{ii}=1$, we obtain
\begin{equation*}
p_{ij}^2=p_i^{2a_{ij}},\quad
(p_{ij}p_{ji})/(\theta_{ij}\theta_{ji})=p_i^{2a_{ij}},\quad
p_{ii}/\theta_{ii}=p_i^2.
\end{equation*}

\Prop\label{prop:serre} Assume that families
$\theta\seteq\{\theta_{ij}\}_{i,j\in I}$ and $\pa\seteq
(\{p_{ij}\}_{i,j\in I},\{p_i\}_{i\in I})$ of invertible elements of
$\cor$ satisfy the following conditions: \eq &&\ba{l}
p_{ij}^2=p_i^{2a_{ij}},\quad
(p_{ij}p_{ji})/(\theta_{ij}\theta_{ji})=p_i^{2a_{ij}},\quad
p_{ii}/\theta_{ii}=p_i^2\qtext{and}\\[1ex]
\text{$1-p_i^n$ is an invertible element of $\cor$
for any $i\in I$ and $n\in\Z_{>0}$.}
\ea. \label{cond:pt} \eneq

Set $p_{ij}=c_{ij}p_i^{a_{ij}}$. Then we have
\begin{align*}
&e_\ell\Bigl(\sum_{k=0}^{1-a_{ij}}(-c_{ij}\theta_{ji}^{-1})^k
f_i^\px{1-a_{ij}-k}f_jf_i^\px{k}\Bigr)\\
&\hs{15ex}
=\theta_{i\ell}^{1-a_{ij}}\theta_{j\ell }
\Bigl(\sum_{k=0}^{1-a_{ij}}(-c_{ij}\theta_{ji}^{-1})^k
f_i^\px{1-a_{ij}-k}f_jf_i^\px{k}\Bigr)e_\ell, \\
&f_\ell\Bigl(\sum_{k=0}^{1-a_{ij}}(-c_{ij}\theta_{ij})^ke_i^\px{1-a_{ij}-k}e_j
e_i^\px{k}\Bigr)\\
&\hs{15ex}=
\theta_{\ell i}^{-1+a_{ij}}\theta_{\ell j}^{-1}
\Bigl(\sum_{k=0}^{1-a_{ij}}(-c_{ij}\theta_{ij})^k
e_i^\px{1-a_{ij}-k}e_je_i^\px{k}\Bigr)f_\ell
\end{align*}
for all $\ell$ and $i\not=j$ in $I$. \ro Note that $c_{ij}^2=1$.\rf
\enprop

\Proof Set $\theta'=\{\theta'_{ij}\}$, $\pa'=(\{p'_{ij}\}_{i,j\in
I},\{p_i\}_{i\in I})$ with $\theta'_{ij}=\theta_{ij}/\theta_{jj}$
and $p'_{ij}=p_{ij}/\theta_{ii}$. Then
$p'_{ij}=(\theta_{ii}c_{ij})p_i^{a_{ij}}$ and as shown in
Proposition~\ref{prop:guage}, there exists an isomorphism
$\kappa\cl\F[{\pa,\theta}][P,Q,R]\isoto \F[{\pa',\theta'}][P,Q,R]$
with $x_{ij}=\theta_{ii}$, $y_{ij}=1$, $Q_i=1$, and
$\eps_{ij}=\theta_{ii}$. Set \eqn
S_{ij}&=&\sum_{k=0}^{1-a_{ij}}(-(\theta_{ii}c_{ij})\theta'_{ji}{}^{-1})^k
f_i^\px{1-a_{ij}-k}f_jf_i^\px{k}\\
&=&\sum_{k=0}^{1-a_{ij}}(-c_{ij}\theta_{ji}^{-1})^k
f_i^\px{1-a_{ij}-k}f_jf_i^\px{k}
\in\F[{\pa,\theta}]\quad\text{for $i\not=j$.}
\eneqn
Then we have $e_\ell \kappa(S_{ij})=\theta'_{i\ell}{}^{1-a_{ij}}\theta'_{j\ell}
\kappa(S_{ij})e_{\ell}$. On the other hand, we have
\begin{align*}
\kappa^{-1}(e_\ell)f_i^\px{1-a_{ij}-k}f_jf_i^\px{k}&= e_\ell
P_\ell^{-1}f_i^\px{1-a_{ij}-k}f_jf_i^\px{k} =e_\ell
\theta_{\ell\ell}^{2-a_{ij}}f_i^\px{1-a_{ij}-k}f_jf_i^\px{k}P_\ell^{-1}.
\end{align*}
Hence we obtain the first equality.

The other equality follows from this equality by applying
the anti-automorphism \eqref{eq:antiiso}.
\QED

The condition  \eqref{cond:pt} implies \eq&&
(\theta_{ij}\theta_{ji})^2=1,\quad \theta_{ii}^2=1, \quad
\theta_{ij}\theta_{ji}=p_{ij}p_{ji}^{-1},
 \quad p_i^{2a_{ij}}=p_j^{2a_{ji}}. \label{cond:tip} \eneq
Conversely, for any family $\{p_i\}_{i\in I}$ of elements in
$\cor^\times$ satisfying \eqref{cond:tip}, we can find
$\theta=\{\theta_{ij}\}_{i,j\in I}$ and $\pa=(\{p_{ij}\}_{i,j\in
I},\{p_i\}_{i\in I})$ satisfying \eqref{cond:pt}. Indeed, it is
enough to take $$p_{ij}=p_i^{a_{ij}}, \quad \theta_{ii}=1, \quad
\theta_{ij}\theta_{ji}= p_i^{a_{ij}}p_j^{-a_{ji}} \ \  (i\not=j).$$
Note that under the condition \eqref{cond:pt}, we have \eq
&&e_if_i^{\px{n}}=\theta_{ii}^nf_i^{\px{n}}e_i+\theta_{ii}^{n-1}f_i^{\px{n-1}}
\{p_i^{1-n}K_i\}_i^\pa. \eneq

\begin{definition} \label{Dfn:U(theta,p)}
Assume that $\theta=\{\theta_{ij}\}_{i,j\in I}$ and
$\pa=(\{p_{ij}\}_{i,j\in I},\{p_i\}_{i\in I})$ satisfy the condition
\eqref{cond:pt}. We define the \emph{quantum algebra}
 $\FS$ to be
the quotient of $\F$ by imposing the Serre relations:
\begin{equation}\label{eq:Serre}
\begin{aligned}
&
\sum\limits_{k=0}^{1-a_{ij}}(-c_{ij}\theta_{ji}^{-1})^kf_i^\px{1-a_{ij}-k}f_j
f_i^\px{k}=0 \quad (i \neq j),\\
&
\sum\limits_{k=0}^{1-a_{ij}}(-c_{ij}\theta_{ij})^ke_i^\px{1-a_{ij}-k}e_j
e_i^\px{k}=0 \quad (i \neq j).
\end{aligned}
\end{equation}
\end{definition}

Note that
$$c_{ij}\theta_{ji}^{-1}=\theta_{ij}p_{ji}^{-1}p_i^{a_{ij}}, \quad
c_{ij}\theta_{ij}=\theta_{ji}^{-1}p_{ij}p_i^{-a_{ij}}.$$ Hence there
exists an automorphism $\psi\cl\FS \to \FS$ given by
\begin{equation} \label{eq: auto psi}
e_i \mapsto f_i K_i^{-1}, \quad f_i \mapsto K_ie_i,
\quad K_i \mapsto  K_i^{-1}\theta_{ii},
\end{equation}

It is easy to see that the  algebra $\FS$ has a $\rtl$-weight space
decomposition
$$\FS=\soplus_{\al \in \rtl}\FS_\al$$
with $K_i^{\pm1}\in\FS_0$, $e_i\in\FS_{\al_i}$,
$f_i\in\FS_{-\al_i}$. Let $\FSp$ (resp.\ $\FSm$) be the
$\cor$-subalgebra of $\FS$ generated by  $f_i$'s (resp.\  $e_i$'s)
$(i \in I)$ and set $\FSz=\cor[K_i^{\pm1}\mid i\in I]$. By a
standard argument, we obtain a triangular decomposition of $\FS$:

\Prop\label{prop:tri} The multiplication on $\FS$ induces an
isomorphism
$$\FSm\tens\FSz\tens\FSp\isoto\FS.$$
\enprop

Let $G$ be a subset of $\P$ such that $G+\rtl\subset \P$. For each
$i\in I$, let us take a function $\chi_i\cl G\to \cor^\times$ such
that \eq \chi_i(\lam)^2=p_i^{2\la h_i,\lam\ra}, \quad
\chi_i(\lam+\al_j)=p_{ij}\,\chi_i(\lam) \ \ \text{for all} \ \
\lam\in G, \ j\in I. \label{cond:chi} \eneq Such a $\chi_i$ always
exists as seen in Lemma~\ref{lem:equiv} below. We say that a
$\FS$-module $V$ is a {\em $G$-weighted module} if it is endowed
with a {\em $G$-weight-space decomposition}
$$ V= \bigoplus_{\lam \in G}V_\lam $$
such that $\FS_\alpha V_\lam\subset V_{\lam+\alpha}$ for any $\al\in
\rtl$, $\lam\in G$ and $K_i\vert_{V_\lam}=\chi_i(\lam)\id_{V_\lam}$
for any $\lam\in\P$, $i\in I$.

We define $\Mod^G(\FS)$, $\Ocat[]^G(\FS)$ and $\Ocat^G(\FS)$ in the
same way as in Section \ref{sec:cartan}. The category $\Mod^G(\FS)$
does not depend on the choice of $\{ \chi_i \}_{i \in I}$ in the
following sense.

\begin{lemma}\label{lem:equiv}Let $G$ be a subset of $\P$ such that $G+\rtl\subset \P$.
\bnum
\item There exists $\{\chi_i\}_{i\in I}$ satisfying the condition
\eqref{cond:chi}.
\item
For  another choice of $\{\chi'_i\}_{i\in I}$ satisfying \eqref{cond:chi}, let
 $\Mod^G(\FS)'$ be the  category of $G$-weighted
$\FS$-modules with respect to $\{\chi'_i\}_{i\in I}$.
Then there is an equivalence of categories
$$\Phi\cl\Mod^G(\FS)\isoto\Mod^G(\FS)'.$$
\ee
\end{lemma}

\begin{proof}

\noi
(i)\quad We may assume that $G=\lam_0+\rtl$ for some $\lam_0$.
Then it is enough to take
$\chi_i(\lam_0+\sum_{j\in I}m_j\al_j)=p_{i}^{\lan h_i,\lam_0\ran}\prod_{j\in I}p_{ij}^{m_j}$.

\noi (ii)\quad Set $\xi_i(\lam)=\chi'_i(\lam)\chi_i(\lam)^{-1}$.
Then we have $\xi_i(\lam+\al_j)=\xi_i(\lam)$ and $\xi_i(\lam)^2=1$.
For $M\in \Mod^G(\FS)$, we define $\Phi(M)=\set{\vphi(u)}{u\in M}$
with the actions
$$K_i\vphi(u)=\vphi(\xi_i(\lam)K_iu), \quad
e_i\vphi(u)=\vphi(\xi_i(\lam)e_iu), \quad f_i\vphi(u)=\vphi(f_iu) \
\ \text{for} \ u\in M_\lam.$$ We can easily see that $\Phi(M)$
belongs to $\Mod^G(\FS)'$, and hence $\Phi$ gives a desired
equivalence.
\end{proof}

The following proposition is a consequence of
Proposition~\ref{prop:guage}.

\Prop \label{Prop: depend on
p_{ii}} 
Under the condition \eqref{cond:pt}, the category
$\Mod^G(\FS)$ depends only on $\left\{ p_{i}^2 \right\}_{i\in I}$.
\enprop
\Proof Assume that $\pa= (\{p_{ij}\}_{i,j\in
I},\{p_i\}_{i\in I})$ and $\pa'= (\{p'_{ij}\}_{i,j\in
I},\{p'_i\}_{i\in I})$ satisfy \eqref{cond:pt} and also
$p_i^2=p'_i{}^2$. Then the condition \eqref{eq:twop} is satisfied.
Therefore, there exist $x_{i,j}$, $y_{i,j}$, $\eps_{i,j}$ and $c_i$
in $\cor^\times$ such that $\eps_{i,j}^2=1$ and \eqref{eq:xyc}
holds. Hence, Proposition~\ref{prop:guage} implies that there exists
an isomorphism $\kappa\cl\F{[P,Q,R]}\isoto\F[\theta',\pa']{[P,Q,R]}$
satisfying \eqref{def:kappa}. Now we can check easily that $\kappa$
sends the Serre relation in $\F$ to the Serre relation in
$\F[\theta',\pa']$, which implies that $\kappa$ induces an
isomorphism
$$\kappa'\cl\FS{[P,Q,R]}\isoto\FS[\theta',\pa']{[P.Q,R]}.$$
Now we shall show
$\Mod^G(\FS)$ and $\Mod^G(\FS[\theta',\pa'])$ are equivalent.
We may assume that $G=\lam_0+\rtl$ for some $\lam_0\in \P$
without loss of generality.
Then for $M\in\Mod^G(\FS)$, we define the action of
$P_i$, $Q_i$, $R_i$ by
\eqn
&&P_iu=(\prod_{j\in I}x_{i,j}^{m_{j}})u,\
Q_iu=c_i(\prod_{j\in I}y_{i,j}^{m_{j}})u,\
R_iu=(\prod_{j\in I}\eps_{i,j}^{m_{j}})u
\eneqn
for $u\in M_\lam$ with $\lam=\lam_0+\sum_{j\in I}m_j\al_j$.
Then it is obvious that $P=(P_i)_{i\in I}$, $Q=(Q_i)_{i\in I}$ and $R=(R_i)_{i\in I}$
satisfy the relations
\eqref{eq:PQR}.
Hence $M$ has a structure of $\FS{[P,Q,R]}$.
Then the isomorphism $\kappa'$ induces
a $\FS[\theta',\pa']{[P.Q,R]}$-module structure on $M$.
Thus we obtain a functor
$\Mod^G(\FS)\to\Mod^G(\FS[\theta',\pa'])$.
It is obvious that it is an equivalence of categories.
\QED

Recall that $p_{ii}\theta_{ii}^{-1}=p_i^2$ and that if
$$\text{$(p_i^2)^{a_{ij}}=(p_j^2)^{a_{ji}}$ for any $i,j\in I$,}$$
then we can find $\theta$ and $\pa$ satisfying \eqref{cond:pt}.

Let us take $\chi_i\cl \P\to \cor^\times$  satisfying the condition
\eqref{cond:chi}.

For $\Lambda \in \P$, the {\em Verma module} $\MBe(\Lambda)$
is the $\P$-weighted $\FS$-module
generated by a vector
$u_\Lambda$ of weight $\Lambda$ with the defining relations:
\begin{equation} \label{eq:Verma}
K_iu_\Lambda=\chi_i(\La) u_\Lambda, \quad e_iu_\Lambda=0
\quad\text{for all $i\in I$.}
\end{equation}

Then $\FSm\to \MBe$ ($a\mapsto au_\Lambda$) is a $\FSm$-linear
isomorphism.

There exists a unique maximal submodule $\NBe(\Lambda)$ of
$\MBe(\Lambda)$ such that $\NBe(\Lambda)\cap\cor u_\Lambda=0$. Let
\begin{equation}\label{eq:V(lambda)}
\VBe(\Lambda) \seteq\MBe(\Lambda)/\NBe(\Lambda).
\end{equation} Then
$\VBe$ is generated by $v_{\Lambda}$ which is the image of $u_\Lam$.
If $\Lam\in\P^+$, then $\VBe(\Lambda)$ belongs to $\Ocat^{P}(\FS)$
and we have $f_i^{\la h_i,\Lam\ra+1}v_\Lam=0$ for any $i\in I$.

\Conj \label{conj} When $\cor$ is a field, the representation theory
of $\FS$ is similar to that of quantum group.

More precisely,
we conjecture that
\bnum
\item
$\ch (\FSm)\seteq\sum_{\mu\in Q}\bl\dim_\cor\FSm_\mu \br \ex{\mu} =
\prod_{\alpha \in \Delta^+} (1-\ex{-\alpha})^{-\mult(\alpha)}$,

\item
the category $\Ocat^\P(\FS)$ is semisimple,

\item for any $\Lambda\in\P^+$,
the $\FS$-module $\VB(\Lambda)$ is a simple object in
$\Ocat^\P(\FS)$ and is isomorphic to $$\FS/\sum\limits_{i\in
I}\bl\FS(K_i-\chi_i(\La))+\FS e_i+\FS f_i^{\la h_i,\Lam\ra+1}\br.$$
That is, $\VBe(\Lambda)$ is generated by $v_{\Lambda}$ with defining
relations
\begin{equation*}
K_{i}v_{\Lambda} = \chi_{i}(\Lambda) v_{\Lambda}, \quad e_{i}
v_{\Lambda} =0, \quad f_{i}^{\langle h_{i}, \Lambda \rangle +1 }
v_{\Lambda} = 0 \ \ \text{for all} \ i \in I.
\end{equation*}

\item every simple module in $\Ocat^\P(\FS)$ is isomorphic to
$\VBe(\Lambda)$ for some $\Lambda\in\P^+$,
\item for any $\Lambda\in\P^+$, we have
$$\ch (\VBe(\Lambda)) \seteq
\sum_{\mu\in \P}\bl\dim \VBe(\Lambda)_\mu \br \ex{\mu}
= \dfrac{\sum_{ w \in W} \epsilon(w)\ex{w(\Lambda+\rho)-\rho}}%
{\prod_{\alpha \in \Delta_+} (1-\ex{-\alpha})^{\mult(\alpha)}},$$
where $\rho$ is an element of $\P$ such that $\la h_i,\rho \ra =1$
for all $i \in I$.
\ee
\enconj

Note that we have assumed that any $p_{i}$ is not a root of unity.

The notion of quantum Kac-Moody superalgebras introduced in
\cite{KT91, BKM98} is a special case of $\FS$. We will show that our
conjecture holds for such algebras (Theorem~\ref{thm: ch
V(Lambda)}). Our proof depends on their results
(Corollary~\ref{Cor:BKM e quasi commute}).

Now we will prove the $\FS$-version of \cite[Proposition
B.1]{KMPY96} under the condition \eqref{cond:pt}. We assume that the
base ring $\cor$ is a field and that any of $p_i$ is not a root of
unity. We say that an $\F$-module $M$ is {\em integrable} if
\bnum
\item $M$ has a weight decomposition
$$M=\soplus_{\lam\in \P}M_\lam$$
such that $\F_\al M_\lam\subset M_{\lam+\al}$ and
$K_i^2\vert_{M_\lam}=p_i^{2\la h_i,\lam\ra}\id_{M_\lam}$,
\item the action of $K_i$ on $M$ is semisimple for any $i$,
\item the actions of $e_i$ and $f_i$ on $M$ are locally nilpotent
for all $i \in I$. \ee

\begin{proposition} \label{prop: KMPY96}
Let $M$ be an integrable  $\F$-module. Then $M$ is a $\FS$-module.
That is, the actions of $e_i$ and $f_i$ on $M$ satisfy the Serre
relations in \eqref{eq:Serre}.
\end{proposition}

We begin with the following lemma.

\Lemma Let $M$ be an integrable $\F$-module. Fix $i\in I$ and let
$\vphi$ be a $\cor$-linear endomorphism of $M$. Suppose that $\vphi$
satisfies the following conditions:
\bna
\item  $\vphi$ is of weight $\mu$; i.e., $\vphi(M_\lam)\subset M_{\lam+\mu}$
for any $\lam\in \P$,
\item  $e_i\vphi=c\vphi e_i$ for some $c\in\cor^\times$.
\ee

Then $\la h_i,\mu\ra<0$ implies $\vphi=0$. \enlemma \Proof By
Proposition~\ref{prop:guage}, we may assume that $\theta_{ii}=1$.
Let $S_i$ be the operator defined by $S_i\vert_{M_\lam}=p_i^{-\la
h_i,\lam\ra}K_i\id_{M_\lam}$. Then $S_i^2=1$ and the algebra
generated by $e_i,f_iS_i,K_iS_i$ is isomorphic to $U_q(\Sl_2)$.
Hence we can reduce our statement to the one for integrable
$U_q(\Sl_2)$-modules.

Recall that any integrable $U_q(\Sl_2)$-module is semisimple and
generated by the vectors killed by $f_i$.

Hence it is enough to show that
$\vphi(U_q(\Sl_2)v)=0$ for any
$v\in M_\lam$ with $f_iv=0$.
Set $m=-\la h_i,\lam\ra\in\Z_{\ge0}$. Then $e_i^{m+1}v=0$ and hence
$e_i^{m+1}\vphi(v)=c^{m+1}\vphi(e_i^{m+1}v)=0$.

On the other hand, setting $n=-\la h_i,\mu\ra>0$, the map
$e_i^{m+n}\cl M_{\lam+\mu}\to M_{s_i(\lam+\mu)}$ is bijective.
Hence
$e_i^{m+n}\vphi(v)=0$ implies $\vphi(v)=0$. Therefore we obtain
$\vphi(e_i^kv)=0$ for any $k$. \QED

\begin{proof}[Proof of Proposition~\ref{prop: KMPY96}]
Let us denote by $S_{ij}$  the multiplication operator on $M$ by
$\sum_{k=0}^{1-a_{ij}}(-c_{ij}\theta_{ji}^{-1})^k
f_i^\px{1-a_{ij}-k}f_jf_i^\px{k}$. Then $S_{ij}$ has weight
$\mu=-(1-a_{ij})\al_i-\al_j$. Moreover, $e_iS_{ij}=cS_{ij}e_i$ for
some $c\in\cor^\times$ by Proposition~\ref{prop:serre}. Since $\la
h_i,\mu\ra=-2(1-a_{ij})-a_{ij}=-2+a_{ij}<0$, we have $S_{ij}=0$. .
\end{proof}

\vskip 5mm

\section{The algebra $\HS$}\label{sec: HS}

In this section, we introduce another generalization of quantum
groups. Let $\tth\seteq\{\tth_{ij}\}_{i,j\in I}$ and $\tpa\seteq
\{\tip_i\}_{i\in I}$ be families of invertible elements in the base
ring $\cor$ such that $1-\tip_i^n$ is invertible for any
$n\in\Z_{>0}$. We define $\HF$ to be the $\cor$-algebra generated by
$e_i$, $f_i$, $\dK_i^{\pm1}$ with the defining relations \eq
&&\ba{l} \dK_i \dK_j = \dK_j \dK_i, \quad \dK_i e_j
\dK_i^{-1}=\tip_{i}^{a_{ij}}e_j,\quad \dK_i f_j
\dK_i^{-1}=\tip_{i}^{-a_{ij}}f_j,
\\[1ex]
e_if_j-\tth_{ji}f_je_i=\delta_{i,j}\dfrac{1-\dK_i}{1-\tip_i}.
\ea\label{def:genU2}
\eneq
Then there exists an anti-isomorphism
\eq
&&\HF\isoto\HF[{\trp{\tth},{\tip}}]
\label{eq:antiiso2}
\eneq
given by
$$e_i\mapsto f_i, \quad f_i\mapsto e_i, \quad \dK_i\mapsto \dK_i, $$
where $(\trp{\tth})_{ij}=\tth_{ji}$.

We embed $\cor[\dK_i^{\pm1}\mid i\in I]$ into $\cor[K_i^{\pm1}\mid
i\in I]$ by $\dK_i=K_i^2$. If $p_{ij}^2=\tip_{i}^{a_{ij}}$ for
$i,j\in I$, then $\HF\tens_{\cor[\dK_i^{\pm1}\mid\; i\in
I]}\cor[K_i^{\pm1}\mid i\in I]$ has a ring structure given by
$$K_i e_j K_i^{-1}=p_{ij}e_j,\quad K_i f_j K_i^{-1}=p_{ij}^{-1}f_j.$$

\Prop \label{prop:FHequiv} Let $\theta\seteq\{\theta_{ij}\}_{i,j\in
I}$ and $\pa\seteq (\{p_{ij}\}_{i,j\in I},\{p_i\}_{i\in I})$ be
families of invertible elements in $\cor$ such that \eq&&
\tth_{ij}=\theta_{ij}p_{ji}^{-1},\quad \tip_i^{a_{ij}}=p_{ij}^2,
\quad \tip_i=p_i^2. \label{rel:ptt} \eneq

Then we have a $\cor$-algebra isomorphism
$$\phi:\HF\tens_{\cor[\dK_i^{\pm1}\mid i\in I]}\cor[K_i^{\pm1}\mid i\in I]
\isoto\F
$$
given by $$e_i\mapsto p_i^{-1}p_{ii}\,e_iK_i,\quad f_i\mapsto f_i,
\quad K_i\mapsto K_i \ \ (i \in I).$$
\enprop
\Proof we have
\begin{align*}
\phi(e_if_j & - \tth_{ji}f_je_i)
=p_i^{-1}p_{ii}(e_iK_if_j-\theta_{ji}p_{ij}^{-1}f_je_iK_i)\\
&=p_i^{-1}p_{ii}p_{ij}^{-1}(e_if_j-\theta_{ji}f_je_i)K_i\\
&=\delta_{i,j}p_i^{-1}\dfrac{K_i-K_i^{-1}}{p_i-p_i^{-1}}K_i
=\delta_{i,j}\dfrac{K_i^2-1}{p_i^2-1}
=\phi\Bigl(\delta_{i,j}\dfrac{1-\dK_i}{1-\tip_i}\Bigr),
\end{align*}
which proves our claim.
\QED

If \eqref{cond:pt} and \eqref{rel:ptt} are satisfied, then we have
\eq && \tth_{ij}\tth_{ji}=\tip_i^{-a_{ij}} \quad \text{and} \quad
\tth_{ii}=\tip_i^{-1}, \label{cond:ttp} \eneq which implies
\eq&&\tip_i{}^{a_{ij}}=\tip_{j}{}^{a_{ji}}\label{eq:tp}\eneq
Conversely, if the family $\{\tip_i\}_{i\in I}$ satisfies
\eqref{eq:tp}, then we can find $\{\tth_{ij}\}_{i,j\in I}$
satisfying \eqref{cond:ttp}.

\smallskip
Let $\tth=\{ \tth_{ij} \}$ and $\tip=\{ \tip_{i} \}$ be families of
elements in $\cor^\times$ satisfying \eqref{cond:ttp}. Set
\eq
&&\qintt{n}\seteq\dfrac{1-\tip_i^n}{1-\tip_i},\quad
\factt{n}\seteq\prod_{k=1}^n\qintt{k},\quad e_i^\pix{n}\seteq
e_i^n/\factt{n},\quad f_i^\pix{n}=f_i^n/\factt{n}.\label{def:factt}
\eneq Then under the condition \eqref{rel:ptt}, we have
$$\qintt{n}=p_i^{n-1}\qintp{n} \quad \text{and} \quad
\factt{n}=p_i^{n(n-1)/2}\factp{n}.$$ Hence we have
$$f_i^\px{n}=p_i^{n(n-1)/2}f_i^\pix{n}.$$

Take $p_i\in \cor^\times$ such that $p_i^2=\tip_i$ and set \eq
&&\text{$p_{ij}=p_i^{a_{ij}}$ and
$\theta_{ij}=\tth_{ij}p_j^{a_{ji}}$.} \label{eq:expl} \eneq Then
\eqref{cond:pt} and \eqref{rel:ptt} hold.
Since we have \eqn
&&f_i^\px{1-a_{ij}-k}f_jf_i^\px{k}
=p_i^{-(1-a_{ij})a_{ij}/2-k(1-a_{ij}-k)}f_i^\pix{1-a_{ij}-k}f_jf_i^\pix{k}
\eneqn and \eqn
&&(-c_{ij}\theta_{ji}^{-1})^kp_i^{-(1-a_{ij})a_{ij}/2-k(1-a_{ij}-k)}\\
&&=(-c_{ij}\tth_{ji}^{-1}p_{ij}^{-1})^kp_i^{-(1-a_{ij})a_{ij}/2-k(1-a_{ij}-k)}\\
&&=(-\tth_{ji}^{-1}p_i^{-a_{ij}})^kp_i^{-(1-a_{ij})a_{ij}/2-k(1-a_{ij}-k)}\\
&&=(-\tth_{ji})^{-k}\,\tip_i{}^{k(k-1)/2}
p_i^{-(1-a_{ij})a_{ij}/2},
\eneqn
Proposition~\ref{prop:serre} implies that
$$S'_{ij}\seteq\ssum_{k=0}^{1-a_{ij}}(-\tth_{ji})^{-k}\tip_i{}^{k(k-1)/2}
f_i^\pix{1-a_{ij}-k}f_jf_i^\pix{k}$$ quasi-commutes with $e_k$'s for
all $k$ (i.e., $e_kS'_{ij}\in\cor^\times S'_{ij}e_k$). Hence by
applying the anti-involution  \eqref{eq:antiiso2}, we see that
$\ssum_{k=0}^{1-a_{ij}}(-\tth_{ij})^{-k}\tip_i{}^{k(k-1)/2}
e_i^\pix{k}e_je_i^\pix{1-a_{ij}-k}$ quasi-commutes with all the
$e_\ell$'s.

\begin{definition}
Assume that $\tilde{\theta}$ and $\tilde{\pa}$ satisfy the
condition \eqref{cond:ttp}. We define the \emph{quantum algebra}
$\HS$ to be the quotient of $\HF$ by imposing the Serre relation:
\begin{equation} \label{eq:Serre2}
\begin{aligned}
& \ssum_{k=0}^{1-a_{ij}}(-\tth_{ji})^{-k}\tip_i{}^{k(k-1)/2}
f_i^\pix{1-a_{ij}-k}f_jf_i^\pix{k}=0 \quad (i \neq j), \\
& \ssum_{k=0}^{1-a_{ij}}(-\tth_{ij})^{-k}\tip_i{}^{k(k-1)/2}
e_i^\pix{k}e_je_i^\pix{1-a_{ij}-k}=0  \quad (i \neq j).
\end{aligned}
\end{equation}
\end{definition}
We can see that the algebra $\HS$ has a $\rtl$-weight space
decomposition
$$\HS=\soplus_{\al\in\rtl}\HS_\al.$$
Let $\HSp$ (resp.\ $\HSm$) be the $\cor$-subalgebra of $\HS$
generated by the $f_i$'s (resp.\ the $e_i$'s) $(i \in I)$ and set
$\HSz=\cor[\dK_i^{\pm1}\mid i\in I]$. By a standard argument, we
have: \Lemma \label{Lem: triangluar for HS} The multiplication on
$\HS$ induces an isomorphism
$$\HSm\tens\HSz\tens\HSp\isoto\HS.$$
\enlemma

Note that we have an algebra isomorphism: \eq&& \HSm\simeq\FSm.
\eneq

For a subset $G$ of $\P$ such that $G+\rtl\subset \P$, a
$\HS$-module $V$ is called a {\em $G$-weight module} if it is
endowed with a {\em $G$-weight space decomposition}
$$ V= \bigoplus_{\mu \in G}V_\mu $$
such that $\HS_\alpha V_\mu\subset V_{\mu+\alpha}$ and
$\dK_i\vert_{V_\mu}=\tip_i^{\lan h_i\mu\ran}\id_{V_\mu}$ for any
$\al\in \rtl$ and $\mu\in G$. We define the categories
$\Mod^G(\HS)$,  $\Ocat[]^G(\HS)$ and $\Ocat^G(\HS)$ in the same
manner as in Section \ref{sec:cartan}. The following proposition is
an immediate consequence of Proposition~\ref{Prop: depend on p_{ii}}
and Proposition~\ref{prop:FHequiv}.

\Prop \label{Prop: equivalent 1} Assume that
$\tth\seteq\{\tth_{ij}\}_{i,j\in I}$, $\tpa\seteq \{\tip_i\}_{i\in
I}$, $\theta\seteq\{\theta_{ij}\}_{i,j\in I}$ and $\pa\seteq
(\{p_{ij}\}_{i,j\in I},\{p_i\}_{i\in I})$ satisfy \eqref{rel:ptt}
and \eqref{cond:ttp}. Then the following statements hold.

\bnum
\item The relation \eqref{cond:pt} is satisfied.
\item
There exist equivalences of categories
$$\Mod^G(\HS)\simeq \Mod^G(\FS)
\quad\text{and}\quad\Ocat^G(\HS)\simeq \Ocat^G(\FS).$$
\item
The category $\Mod^G(\HS)$ depends only on the parameters
$\{\tip_i\}_{i\in I}$ satisfying $\tip_i^{a_{ij}}=\tip_j^{a_{ji}}$.
\ee \enprop

Let $\FS{[T_i\mid i\in I]}$ be the algebra obtained from $\FS$ by
adding the mutually commuting operators $T_i$ $(i \in I)$ with the
multiplication given by \eq&& T_ie_j T_i^{-1} = \theta_{ji}e_j
,\quad T_i f_j T_i^{-1} = \theta_{ji}^{-1}f_j,\quad T_i K_j T_i^{-1}
=K_j \quad\text{for any $j\in I$.}\label{def:Ti} \eneq We will
introduce another kind of algebra that acts on $\FSm$ and $\HSm$. We
first prove:

\begin{lemma} \label{Lem: boson relation BqBg}
For any $P \in \FSm$, there exist unique $Q$, $R \in \FSm$  such
that \eq && e_i P - (T_i^{-1} P T_i)e_i =
\dfrac{(T_i^{-1}QT_i)K_i-K_i^{-1} R}{p_i-p_i^{-1}}. \label{eq:der}
\eneq
\end{lemma}

\begin{proof}
The uniqueness follows from Proposition \ref{prop:tri}. Using
induction on the height of $P$, it is enough to show \eqref{eq:der}
for $f_jP$ assuming \eqref{eq:der} for $P$. If \eqref{eq:der} holds
for $P$, then we have
\begin{equation} \label{eqn: boson com rel}
\begin{aligned}
& \ e_i f_j P - (T_i^{-1} f_j P T_i)e_i \\
& = \left(e_if_j-(T_i^{-1}f_jT_i)e_i\right)P + (T_i^{-1} f_j T_i) \left( e_i P - (T_i^{-1} P T_i) e_i \right ) \\
& = \delta_{i,j} \dfrac{K_i-K_i^{-1}}{p_i-p_i^{-1}}P + \theta_{ji}f_j\dfrac{(T_i^{-1}QT_i)K_i-K_i^{-1}R}{p_i-p_i^{-1}}\\
& = \dfrac{(T_i^{-1}f_jQT_i+\delta_{i,j}K_iPK_i^{-1})K_i-
K_i^{-1}(\theta_{ji}p_{ij}^{-1}f_jR+\delta_{i,j}P)}{p_i-p_i^{-1}}.
\end{aligned}
\end{equation}
\end{proof}

We define the endomorphisms $e_i'$ and $e_i^*$ of $\FSm$ by
$$e_i'(P)=R, \quad e_i^*(P)=Q.$$
Assume that $\tth$ and $\tip$ satisfy \eqref{rel:ptt}. Then by
Proposition~\ref{prop:FHequiv}, we have \eq \FSm\simeq\HSm \eneq and
hence we may also regard $e_i'$ and $e_i^*$ as endomorphisms of
$\HSm$. Note that $f_i$ can be regarded as an operator on $\FSm$
given by left multiplication. Thus we have the following relations
in $\End(\FSm)\simeq\End(\HSm)$ as is shown by \eqref{eqn: boson com
rel}:
\begin{equation} \label{eq: e_i' f_j 1}
e_i'f_j = \theta_{ji}p_{ij}^{-1}f_je_i' +\delta_{i,j},\quad
 e_i^*f_j = f_je_i^* +\delta_{i,j}\Ad(T_iK_i).
\end{equation}
More generally, we have \Lemma For $a$, $b\in \FSm$, we have \eqn
e'_i(ab)=(e'_ia)b+\bl\Ad(T_i^{-1}K_i)a\br e'_ib,\\
e^*_i(ab)=(e^*_ia)\bl\Ad(T_iK_i)b\br+a e_i^*b. \eneqn \enlemma

\Proof We have \eqn \bl e_iab-T_i^{-1}abT_ie_i\br
&=&\bl e_ia-T_i^{-1}aT_ie_i\br b+T_i^{-1}aT_i\bl e_ib-T_i^{-1}bT_ie_i\br\\
&=&\dfrac{T_i^{-1}(e^*_ia)T_iK_i-K_i^{-1}e'_ia}{p_i-p_i^{-1}}b
-T_i^{-1}aT_i\dfrac{T_i^{-1}(e^*_ib)T_iK_i-K_i^{-1}e'_ib}{p_i-p_i^{-1}}\\
&=&\dfrac{T_i^{-1}(e^*_ia)(T_iK_ibT_i^{-1}K_i^{-1})K_i-K_i^{-1}(e'_ia)b}{p_i-p_i^{-1}}\\
&&\hs{10ex}-\dfrac{T_i^{-1}a(e^*_ib)T_iK_i-K_i^{-1}(K_iT_i^{-1}aT_iK_i^{-1})e'_ib}{p_i-p_i^{-1}},
\eneqn
which proves our assertion.  \QED

Recalling
$\tth_{ij}=\theta_{ij}p_{ji}^{-1}=\theta_{ji}^{-1}p_{ij}^{-1}$, we
obtain
$$e_i'f_j = \tth_{ji}f_je_i' +\delta_{i,j}.$$
Using induction on $n$, we obtain
\begin{equation} \label{eq: e_i' f_j 2}
e_i'^{n}f_j = \tth_{ji}^nf_je_i'^{n} +
\delta_{i,j}p_i^{1-n}\qintp{n} e_i'^{n-1}.
\end{equation}

\begin{definition} \label{def:BqBg}
We define the \emph{quantum boson algebra} $\Bg$ to be the
$\cor$-algebra generated by $e_i'$, $f_i \ (i \in I)$ satisfying the
following defining relations :
\begin{equation} \label{eq:qboson}
\begin{aligned}
& e_i'f_j = \tth_{ji}f_je_i' + \delta_{i,j}, \\
& \sum\limits_{k=0}^{1-a_{ij}}  (-\tth_{ij}p_i^{a_{ij}})^k
\binmp{1-a_{ij}}{k} e_i'{}^{1-a_{ij}-k}e_j'e_i'{}^{k}=0 \ \ (i\neq
j), \\
& \sum\limits_{k=0}^{1-a_{ij}} (-\tth_{ij}p_i^{a_{ij}})^k
\binmp{1-a_{ij}}{k}f_i^{1-a_{ij}-k}f_jf_i^{k} =0 \ \ (i\neq j).
\end{aligned}
\end{equation}
\end{definition}

Note that $p_i^{ka_{ij}}\binmp{1-a_{ij}}{k}\in\Z[p_i^2,p_i^{-2}]$.
There is an anti-isomorphism $\Bg \leftrightarrow \Bgt$ given by
\begin{equation}\label{eqn: anti-auto of BqBg}
e_i' \leftrightarrow f_i, \quad f_i \leftrightarrow e_i', \quad
\text{where} \ \ ({}^t\tth)_{ij}=\tth_{ji}.
\end{equation}

\begin{proposition} \label{prop: BqBg-structure on UqBg^-}
The algebras $\FSm$ and $\HSm$ have a structure of left
$\Bg$-modules and they are isomorphic as $\Bg$-modules.
\end{proposition}

\begin{proof}
We have only to verify the second relation in Definition
\ref{def:BqBg}. For $i \ne j$ and $b \seteq 1-a_{ij}$, let
$$S=\sum_{n=0}^{b} x_n e_i'{}^{b-n} e_j' e_i'{}^n,$$
where $x_n=(-\tth_{ij}p_i^{1-b})^n\binmp{b}{n}
=(-\tth_{ji}^{-1}p_i^{b-1})^{-n}\binmp{b}{n}$. It is enough to show
that $S$ quasi-commutes with all the $f_k$'s as an operator on
$\FSm$. We have \eqn e_i'{}^{b-n} e_j' e'_i{}^nf_k&=& e_i'{}^{b-n}
e_j'
(\tth_{ki}^nf_ke'_i{}^n+\delta_{k,i}p_i^{1-n}\qintp{n}e_i'{}^{n-1})\\
&=&\tth_{ki}^n e_i'{}^{b-n}(\tth_{kj}f_k e_j' +\delta_{k,j})e'_i{}^n
+\delta_{k,i}p_i^{1-n}\qintp{n}e_i'{}^{b-n} e_j'e_i'{}^{n-1}\\
&=&\tth_{ki}^n \tth_{kj}(\tth_{ki}^{b-n}
f_ke_i'{}^{b-n}+\delta_{k,i}p_i^{1-b+n}\qintp{b-n}e_i'{}^{b-n-1})e'_j
e'_i{}^n\\
&&\hs{15ex}+\delta_{k,j}\tth_{ki}^n e_i'{}^{b}+\delta_{k,i}p_i^{1-n}\qintp{n}e_i'{}^{b-n} e_j'e_i'{}^{n-1}\\
&=&\tth_{ki}^b\tth_{kj}f_ke_i'{}^{b-n}e'_je'_i{}^n
+\delta_{k,j}\tth_{ki}^n e_i'{}^{b}\\
&&\hs{5ex}+\delta_{k,i}\Bigl(\tth_{ki}^n \tth_{kj}
p_i^{1-b+n}\qintp{b-n}e_i'{}^{b-n-1}e'_je'_i{}^n
+p_i^{1-n}\qintp{n}e_i'{}^{b-n} e_j'e_i'{}^{n-1}\Bigr).
\eneqn

Using $\tth_{ii}=p_i^{-2}$, we have
\begin{equation*}
\begin{aligned}
Sf_k
& = \tth_{ki}^b\tth_{kj}f_kS +\delta_{k,j}\Bigl(\ssum_{n=0}^bx_n\tth_{ji}^n\Bigr)
e_i'{}^{b}\\
& \quad + \delta_{k,i} \Bigl(\ssum_{n=0}^b x_n\tth_{ij}
p_i^{1-b-n}\qintp{b-n}e_i'{}^{b-n-1}e'_je'_i{}^n +\ssum_{n=0}^bx_n
p_i^{1-n}\qintp{n}e_i'{}^{b-n} e_j'e_i'{}^{n-1}\Bigr).
\end{aligned}
\end{equation*}
The second term vanishes since \eqn \ssum_{n=0}^bx_n\tth_{ji}^n
&=&\sum_{n=0}^b(-p_{i}^{b-1})^n\binmp{b}{n}=0. \eneqn

Since $\binmp{b}{n}\qintp{b-n}=
\binmp{b}{n+1}\qintp{n+1}$,
the coefficient of $e_i'{}^{b-n-1}e'_je'_i{}^n$ in the third term is equal to
\eqn
&&x_n\tth_{ij}p_i^{1-b-n}\qintp{b-n}+x_{n+1} p_i^{-n}\qintp{n+1}\\
&&=(-\tth_{ij}p_i^{1-b})^n\binmp{b}{n}\tth_{ij}p_i^{1-b-n}\qintp{b-n}
+ (-\tth_{ij}p_i^{1-b})^{n+1}\binmp{b}{n+1}p_i^{-n}\qintp{n+1}=0
\eneqn as desired.
\end{proof}

The following lemma will be used when we prove that, if the base
ring is a field, then $\FSm$ is a simple $\Bg$-module in the case of
quantum Kac-Moody superalgebras.

\begin{lemma} \label{lem: e_i' e_j* 1}
For $i,j \in I$, we have
$$ e_i'e_j^* = e_j^*e_i'.$$
\end{lemma}

\begin{proof}
Set $S = e_i'e_j^*-e_j^*e_i'$. It is enough to show that $S$
quasi-commutes with $f_k$ for any $k \in I$. The relation \eqref{eq:
e_i' f_j 2} yields
\begin{align*}
e'_ie_j^*f_k & = e_i'(f_ke_j^*+\delta_{j,k}\Ad(T_jK_j)) \\
&= (\tth_{ki}f_ke_i'+\delta_{i,k})e_j^*+\delta_{j,k}e_i'\Ad(T_jK_j)\\
&= \tth_{ki}f_ke_i'e_j^*+\delta_{i,k}e_j^*+\delta_{j,k}e_i'\Ad(T_jK_j).
\end{align*}
Similarly, we have
\begin{align*}
e_j^*e'_if_k
 & = e_j^*(\tth_{ki}f_ke_i'+\delta_{i,k}) \\
&= \tth_{ki}\bl f_ke_j^*+\delta_{j,k}\Ad(T_jK_j)\br e_i'+\delta_{i,k}e_j^* \\
&= \tth_{ki}f_ke_j^*e_i'+\delta_{j,k}\tth_{ji}\Ad(T_jK_j)e_i'
+\delta_{i,k}e_j''.
\end{align*}
Since we have
$\Ad(T_jK_j)e_i'=\theta_{ij}p_{ji}e_i'=\tth_{ji}^{-1}e'_i$, we
obtain
$$Sf_k = \tth_{jk}^{-1}\tth_{ki}f_k{S}.$$
\end{proof}

\begin{proposition} \label{prop:simple BqBg-module UqBg^-}
Suppose that the following condition holds:
\begin{equation}
\begin{aligned}
& \text{If $P \in \FSm$ satisfies $e_i P \in \FSm e_i$ for all $i
\in I$,} \\
& \text{then $P$ is a constant multiple of $1$.}
\end{aligned}
\end{equation}

Then any $\rtl$-weighted $\Bg$-submodule $N$ of $\FSm$ vanishes if
$N\cap \cor =0$.
\end{proposition}

\begin{proof}
Suppose $N \cap \cor =0$.  It is obvious that any non-zero
$\Bg$-submodule $N$ of $\FSm$ should have a non-zero highest weight
vector with respect to the action of $e_i'$ for all $i \in I$. Hence
it is enough to show that a highest weight vector $u$ of weight
$\al\not=0$ vanishes. We will show this by induction on the height
$|\al|$ of $\al$. If $\al=-\al_i$, then $u=f_i$ up to a constant
multiple, and it is not a highest weight vector. Assume that
$|\al|\ge2$. Then $e^*_iu$ is a highest weight vector by the
preceding lemma. By induction hypothesis, we have $e_i^*u=0$ which
implies $e_iu\in\FS e_i$. Then by our  assumption, $u$ must be a
constant multiple of $1$, which is a contradiction.
\end{proof}

\vskip 5mm

\section{Quantum Kac-Moody superalgebras}
\label{sec: KM super}

In this section, we show that quantum Kac-Moody superalgebras arise
as a special case of the algebras $\FS$ and we study their structure
and representation theory. We first recall the definition and their
properties following \cite{BKM98}.

\subsection{Quantum Kac-Moody superalgebras} \label{subsec: qGK super}

A {\em Cartan superdatum} is a Cartan datum $(\car,\P,\Pi,\Pi^\vee)$
endowed with a decomposition $I=\Iev \sqcup \Iod$ of $I$ such that
\eq&&\text{ $a_{ij} \in 2 \Z $ for all $ i \in \Iod$ and $j \in I$.}
\eneq For a Cartan superdatum $(\car,\P,\Pi,\Pi^\vee)$, we define
the {\em parity function} $\pa \cl I \to \{ 0, 1 \}$ by
$$\pa(i)=1  \quad \text{ if } i \in \Iod \quad \text{ and } \quad \pa(i)=0  \quad \text{ if } i \in \Iev.$$
We extend the parity function on $I^n$ and $\rtl^+$ as follows:
$$ \pa(\nu) \seteq \sum^n_{k=1} \pa(\nu_k), \ \ \pa(\beta) \seteq  \sum_{k=1}^r\pa(i_k)\quad \text{for all $\nu \in I^n$ and $\beta= \sum^r_{k=1} \alpha_{i_k} \in \rtl^+$.}
$$

\noindent
We denote by
$\Pev \seteq \set{ \lambda \in \P }{ \la h_i, \lambda \ra \in 2\Z
\hs{1.5ex} \text{for all $i \in \Iod$}}$ and
$\Pevp \seteq \P^+ \cap \Pev$.

\medskip
Let $\pi$ be an indeterminate with the defining relation
$\pi^2=1$. Then we have $\Z[\pi]=\Z\oplus\Z\pi$.
Let $\sqrt{\pi}$ be an indeterminate such that $(\sqrt{\pi})^2=\pi$.
Hence $\Z[\sqrt{\pi}]=\Z\oplus\Z\sqrt{\pi}\oplus
\Z{\pi}\oplus\Z(\sqrt{\pi})^{-1}$.
For a ring $R$, we define the rings $R^\pi$ and $R^{\sqrt{\pi}}$ by
\begin{equation} \label{dfn: pi ring}
R^\pi\seteq R\tens \Z[\pi]\qtext{and}
R^{\sqrt{\pi}}\seteq R\tens \Z[\sqrt{\pi}].
\end{equation}

For each $i\in I$, set $\pi_{i}:=\pi^{\pa(i)}$ and  choose
$\sqrt{\pi_i}\in \Z[\sqrt{\pi}]$ such that $(\sqrt{\pi_i})^2=\pi_i$.
Note that we have four choices of $\sqrt{\pi_i}$. The element
$\sqrt{\pi_i}$ may not be contained in $\Z[\pi]$ but
$\sqrt{\pi_i}^{a_{ij}} \in \Z[\pi]^\times$ because
$\sqrt{\pi_i}=\pm1$ or $\pm\pi$ for $i\in \Iev$ and $a_{ij}\in2\Z$
for $i\in \Iod$. Throughout this section, we fix a choice of
$\sqrt{\pi_i}$.

Let $q$ be an indeterminate, and set \eq&&\A= \Z[q,q^{-1}], \quad
q_i=q^{\dg_i}, \quad [ n ]^\pi_i=[n]_{\pi_i q_i,\,q^{-1}_i}=
\dfrac{(\pi_i q_i)^n- q_i^{-n}}{\pi_i q_i-q_i^{-1}} \quad \text{for
 $n \in \Z_{\ge 0}$.}\label{def:pifact}
\eneq We define  $[ n ]^\pi_i!$ and ${\binm{n}{m}}_i^\pi$ in a
natural way. Recall that $\dg_i\in\Z_{>0}$ satisfies
$\dg_ia_{ij}=\dg_ja_{ji}$. Hence we have
$q_i^{a_{ij}}=q_j^{a_{ji}}$.

\bigskip
Let $\cor = \Q(q)^{\sqrt{\pi}}$. The {\em quantum Kac-Moody
superalgebra} $\UqBg$ is the $\cor$-algebra $\FS$ with
\begin{equation} \label{eq:KMsuper}
\begin{aligned}
p_i = q_i\sqrt{\pi_i}, \quad  p_{ij}=q_i^{a_{ij}}, \quad
\theta_{ij}\theta_{ji}=1, \quad \theta_{ii}=\pi_i.
\end{aligned}
\end{equation}
Note that $\theta\seteq\{\theta_{ij}\}_{i,j\in I}$ and $\pa\seteq
(\{p_{ij}\}_{i,j\in I},\{p_i\}_{i\in I})$ satisfy the condition
\eqref{cond:pt}. We have $\sqrt{\pi_i}^{2a_{ij}}=1$ and hence
$p_i^{2a_{ij}}=q_i^{2a_{ij}}$. Hence, by multiplying $e_i$ by a
constant, the explicit description of the algebra $\UqBg$ can be
given as follows:


\begin{definition}[{\cite[Definition 2.7]{BKM98}}] \label{dfn: QBKM}
The \emph{quantum Kac-Moody superalgebra} $\UqBg$ associated with a
Cartan superdatum $(\car,\P,\Pi,\Pi^\vee)$ and $\theta$ is the
algebra over $\cor=\Q(q)^{\sqrt{\pi}}$ generated by $e_i$, $f_i$ and
$K_i^{\pm 1}$ $(i \in I)$ subject to the following defining
relations:
\begin{align*}
& K_iK_j=K_jK_i, \quad
K_i e_j K_i^{-1} = q^{a_{ij}}_i e_j, \quad K_i f_j K_i^{-1}= q^{-a_{ij}}_i f_j, \\
& e_if_j- \theta_{ji}f_je_i =\delta_{i,j} \dfrac{K_i-K^{-1}_i}{q_i\pi_i-q_i^{-1}}, \\
& \displaystyle \sum_{k=0}^{1-a_{ij}} (-\theta_{ij})^k \pi_i^{\frac{k(k-1)}{2}}
 f^{\{ 1-a_{ij}-k \}}_i f_j f_i^{ \{k \}} =0\quad (i \neq j),\\
& \sum_{k=0}^{1-a_{ij}} (-\theta_{ij})^k \pi_i^{\frac{k(k-1)}{2}}
e^{\{ 1-a_{ij}-k\} }_i e_j e_i^{\{ k\} } = 0 \quad (i \neq j),
\end{align*}
where $f_i^{\{n\}}=f_i^n/[n]^\pi_i!$ and
$e_i^{\{n\}}=e_i^n/[n]^\pi_i!$.
\end{definition}

We recall some of the basic properties of highest weight
$\BKM$-modules proved in \cite{BKM98}. We denote by ${\rm
 V}_{\theta}^{q}(\Lambda)=\VBe(\Lambda)$ the $\UqBg$-module defined in
\eqref{eq:V(lambda)}. Choose $\chi_i$ such that
$\chi_i(\lam)=p_i^{\la h_i,\lam\ra}$ for $\lam\in\Pev$. Then,  we
have
$$K_iu=p_i^{\la h_i,\lambda\ra}u=c_iq_i^{\la h_i,\lambda\ra}
u\quad\text{for all $\lam\in \Pev$ and $u\in V_\lam$,}
$$
where $c_i\seteq\sqrt{\pi_i}^{\la h_i,\lam\ra}$ satisfies $c_i^2=1$.
Hence the notion of weight space in this paper is the same as the
one in \cite{BKM98} for $\Pev$-weighted $\UqBg$-modules (after
applying the automorphism $K_i\mapsto c_iK_i$, $e_i\mapsto c_ie_i$).
However, {\em the notion of weight spaces in \cite{BKM98} is
different from ours when the weights are not in $\Pev$.} (See also
Section \ref{sec:QKM}.)

\begin{theorem}[{\cite[Theorem 4.15]{BKM98}}] \label{thm: main thm of BKM} \
\bna
\item For $\Lambda \in \Pev^+$, the $\UqBg$-module ${\rm V}_{\theta}^{q}(\Lambda)$ is generated by a highest weight vector $v_\Lambda$
with the defining relations:
\begin{equation} \label{eqn: defining relation of V^BKM(Lambda)}
 K_i v_{\Lambda} = p_i^{\langle h_i, \Lambda \rangle} v_{\Lambda}, \ \
 e_i v_{\Lambda} = 0, \ \
 f_i^{\langle h_i, \Lambda \rangle + 1} v_{\Lambda} =0 \ \
 \text{for all} \ i \in I.
\end{equation}
\item We have $\set{u\in {\rm V}_{\theta}^{q}(\Lambda)}{\text{$e_iu=0$ for any $i\in I$}}
=\cor v_\Lambda$.
\item The category $\Oint^{\Pev}(\C(q)\tens_{\Q(q)}\UqBg)$ is semisimple and every simple object
is isomorphic to ${\rm V}_{\theta}^{q}(\Lambda)
\big/(\sqrt{\pi}-c){\rm V}_{\theta}^{q}(\Lambda)$ for some $\Lambda
\in \Pev^+$ and $c\in\C$ such that $c^4=1$.
 \item  For $\Lambda \in \Pev^+$,
the weight spaces of \/$\UqmBg$ and ${\rm V}_{\theta}^{q}(\Lambda)$
are free $\cor$-modules, and their ranks are given by
\begin{align*}
 &\ch (\UqmBg)\seteq\sum_{\mu\in Q}\bl\rk_{\Q(q)^{\sqrt{\pi}}} \UqBg_\mu \br \ex{\mu}
= \prod_{\alpha \in \Delta^+} (1-\ex{-\alpha})^{-\mult(\alpha)},  \\
 &\ch (\VB(\Lambda)) \seteq
\sum_{\mu\in\P}\bl\rk_{\Q(q)^{\sqrt{\pi}}} {\rm
V}_{\theta}^{q}(\Lambda)_\mu \br \ex{\mu} = \dfrac{\sum_{ w \in W}
\epsilon(w)\ex{w(\Lambda+\rho)-\rho}} {\prod_{\alpha \in \Delta_+}
(1-\ex{-\alpha})^{\mult(\alpha)}},
\end{align*}
where $\rho$ is an element of $\P$ such that $\la h_i,\rho \ra =1$ for all $i \in I$,
\ee
\end{theorem}

The following corollary will play a crucial role in studying the
representation theory of $\UqBg$.

\Cor\label{Cor:BKM e quasi commute}
We have
$$\set{a\in\UqmBg}{\text{$e_ia\in \UqBg e_i$ for any $i\in I$}}
=\cor.$$
\encor
\Proof
We may assume that $a$ is a weight vector of weight different from $0$.
Then, we have $a v_\Lambda=0$ for any $\Lambda\in\Pevp$
by Theorem~\ref{thm: main thm of BKM} (b).
Hence, $a$ belongs to
$\sum_{i\in I}\FSm f_i^{1+\lan h_i,\Lambda\ran}$ for any $\Lambda\in\Pevp$,
which implies that $a=0$.
\QED

\medskip

\subsection{The algebra $\Uqsug$}\label{sec: Uqsug}

Now we will take another choice of $\theta$ and $\pa$ satisfying
\eqref{cond:pt}:
\begin{align} \label{eq:boldU}
p_{i}= q_i \sqrt{\pi_i}, \quad p_{ij} = p_i^{a_{ij}}, \quad
\theta_{ij}=\begin{cases}
\sqrt{\pi_j}^{a_{ji}}&\text{if $i\not=j$,}\\
1&\text{if $i=j$}
\end{cases}
\end{align}
Note that $\theta_{ij}\in\Z[\pi]$ and $\theta_{ij}^2=1$.

We denote by $\Uqsug$ the $\cor$-algebra $\FS$ for this choice. The
explicit description of the algebra $\Uqsug$ is given as follows.

\begin{definition}
The algebra $\Uqsug$ associated with a Cartan superdatum
$(\car,\P,\Pi,\Pi^\vee)$ is the algebra over
$\cor=\Q(q)^{\sqrt{\pi}}$ generated by $e_i$, $f_i$ and $K_i^{\pm
1}$ $(i \in I)$ subject to the following defining relations:
\begin{equation} \label{eq:Uqsug}
\begin{aligned}
& K_iK_j=K_jK_i, \quad
K_i e_j K_i ^{-1}= p^{a_{ij}}_i e_j, \quad K_i f_j K_i^{-1}= p^{-a_{ij}}_i f_j,  \\
& e_if_j- \theta_{ji}f_je_i = \delta_{i,j}\dfrac{K_i-K^{-1}_i}{p_i-p_i^{-1}} \quad (i, j \in I), \\
& \displaystyle \sum_{k=0}^{1-a_{ij}} (-\theta_{ji})^k
f^{(1-a_{ij}-k)}_i f_j f_i^{(k)} =0 \quad (i \neq j), \\
& \displaystyle\sum_{k=0}^{1-a_{ij}} (-\theta_{ij})^k
e^{(1-a_{ij}-k)}_i e_j e_i^{(k)} =0 \quad (i \neq j),
\end{aligned}
\end{equation}
where $f_i^{(k)}= f_i^k / [k]^\pa_i!$ and $e_i^{ (k) } = e_i^k /
[k]^\pa_i!$.
\end{definition}

Let $\Uqmsug$ (resp.\ $\Uqpsug$) be the $\cor$-subalgebra of
$\Uqsug$ generated by the $f_i$'s (resp.\ the $e_i$'s)  and let
$\Uqzsug$ be the $\cor$-subalgebra generated by the $K_i^{\pm 1}$'s
$(i \in I)$. We choose $\chi_i(\lam)=p_i^{\la h_i,\lam\ra}$ to
define $\Mod^P(\Uqsug)$. By Corollary \ref{cor:Fequ} and Proposition
\ref{prop:serre}, we have

\begin{align} \label{eq:Iso1}
\Uqsug[P,Q,R] \simeq U^\theta_q(\g)[P,Q,R].
\end{align}
Hence the triangular decomposition of $U^\theta_q(\g)$ and Theorem
\ref{thm: main thm of BKM} imply the following corollary.

\begin{corollary} \label{cor: tri dec Super} \hfill
\bnum
\item The algebra $\Uqsug$ has a triangular decomposition
$$\Uqsug \simeq  \Uqmsug \ot \Uqzsug \ot \Uqpsug.$$
\item  $\ch(\Uqmsug) = \displaystyle
 \prod_{\alpha \in \Delta^+}
 (1-\ex{-\alpha})^{-\mult(\alpha)}$.
\item We have $\set{a\in\Uqmsug}{\text{$e_ia\in \Uqsug e_i$ for any $i\in I$}}
=\cor$.
\item There is an equivalence of categories $\Mod^{\P}(\UqBg) \simeq \Mod^{\P}(\Uqsug)$.

\ee
\end{corollary}

Let $\mathbf{B}(\g)$ be the algebra $\Bg$ given in Definition
\ref{def:BqBg} with \eq
\tth_{ij}=\pi_i^{\delta_{i,j}}q^{-(\alpha_i|\alpha_j)},\quad
\tip_{ij}=q_i^{2a_{ij}},\quad\tip_i=q_i^2\pi_i. \eneq The explicit
description of $\mathbf{B}(\g)$ is given as follows.

\begin{definition} \label{dfn:quantum  boson algebra}
The {\em quantum  boson algebra} $\mathbf{B}(\g)$ is the
associative algebra over $\cor$ generated by $e_i', f_i \ (i \in I)$
satisfying the following defining relations:
\begin{equation}\label{def:boson2}
\begin{aligned}
& e_i'f_j = \pi_i^{\delta_{i,j}}q^{-(\alpha_i|\alpha_j)}f_je_i' +
\delta_{i,j}, \\
& \sum\limits_{k=0}^{1-a_{ij}} (-\theta_{ij})^k
\left[\begin{matrix}1-a_{ij}
\\ k \end{matrix} \right]^\pa_i
               e_i'^{1-a_{ij}-k} e_j'e_i'^{k}=0 \quad (i \neq j), \\
& \sum\limits_{k=0}^{1-a_{ij}} (-\theta_{ij})^k
\left[\begin{matrix}1-a_{ij}
\\ k \end{matrix} \right]^\pa_i
               f_i^{1-a_{ij}-k}f_jf_i^{k} =0 \quad (i \neq j).
\end{aligned}
\end{equation}
\end{definition}

Note that $\mathbf{B}(\g)$ has an anti-automorphism given by $e_i'
\mapsto f_i$, $f_i \mapsto e_i'$ $(i \in I)$. By Proposition
\ref{prop: BqBg-structure on UqBg^-}, Proposition \ref{prop:simple
BqBg-module UqBg^-} and Corollary~\ref{cor: tri dec Super}, we have
the following proposition.

\begin{proposition} \label{prop:simple Bqsug-module Uqsug^-}
Suppose $N$ is a $\rtl$-weighted $\mathbf{B}(\g)$-submodule of
$\Uqmsug$ such that $N \cap \cor = 0$. Then $N=0$.
\end{proposition}

Let $E_i' \seteq (p_i-p_i^{-1})^{-1}e_i'$ and $E_i^* \seteq
(p_i-p_i^{-1})^{-1}e_i^*$ $(i \in I)$. Then we have \eq && e_i P -
(T_i^{-1} P T_i)e_i = \bl T_i^{-1}E_i^*(P)T_i\br K_i-K_i^{-1}
E'_i(P). \label{eq:EE*} \eneq The same argument as in \cite[Lemma
3.4.3, Proposition 3.4.4]{Kash91} shows that there exists a unique
non-degenerate symmetric bilinear form on $\Uqmsug$ satisfying
\eq&&\label{eqn: E_i', E_i^*} (1,1)=1, \ (E_i'P,Q) = (P,f_iQ), \
(E_i^*P,Q)=(P,Qf_i) \ \text{for} \ i\in I, P,Q \in \Uqmsug. \eneq

\vskip 5mm

\subsection{Representation theory of $\Uqsug$} \label{sec: Rep Uqsug}

In this subsection, we show that the category $\Oint^\P(\Uqsug)$ of
integrable $\Uqsug$-modules is semisimple. We first construct the
{\em quantum Casimir operator} which is the key ingredient of our
proof. The main argument follows those of  \cite[Chapter
9,10]{Kac90} and \cite[Chapter 1]{Lus93}. Note that, in the present
case, we take $\cor = \Q(q)^{\sqrt{\pi}}$. Moreover, we have
$\theta_{ij}^2=\theta_{ii}=1$ and hence the automorphism $\psi$ of
$\Uqsug$ introduced in \eqref{eq: auto psi} is given by \eq
e_i\mapsto f_iK_i^{-1},\quad  f_i\mapsto K_ie_i,\quad K_i\mapsto
K_i^{-1}. \eneq Recall that the operators $T_i$ introduced in
\eqref{def:Ti} become \eqn T_ie_jT_i^{-1}=\theta_{ji}e_j,\quad
T_if_jT_i^{-1}=\theta_{ji}^{-1}f_j,\quad T_iK_jT_i^{-1}=K_j. \eneqn
In this case, we have $T_i^2=1$.

\begin{lemma} \label{Lem: define inductive}
Let $a_i\cl \lambda-\rtl^+ \to \cor^\times$ $(i \in I)$ be a family
of maps such that
\begin{equation} \label{eqn: a_i, a_j}
\dfrac{a_i(\mu-\alpha_j)}{a_i(\mu)} =
\dfrac{a_j(\mu-\alpha_i)}{a_j(\mu)}
\end{equation}
for all $\mu \in \lam-\rtl^+$ and $i,j \in I$. Then there exists a
unique map $\Psi\cl\lambda - \rtl^+ \to \cor^\times$ such that
$$\Psi(\lambda)=1, \qquad \Psi(\mu-\alpha_i)=a_i(\mu)^{-1}\Psi(\mu).$$
\end{lemma}

\begin{proof}
We shall define $\Psi(\lambda-\beta)$ for $\beta \in \rtl^+$ by induction on $|\beta|$ such that
$$\Psi(\lambda-\beta) =a_i(\lambda-\beta+\alpha_i)^{-1}\Psi(\lambda-\beta+\alpha_i)\qtext{whenever $\beta-\alpha_i \in \rtl^+$.}$$
It is enough to show that the right hand does not depend on $i$. Assume that
$i \neq j$ and $\beta-\alpha_i, \beta-\alpha_j \in \rtl^+$. Then $\beta-\alpha_i-\alpha_j \in \rtl^+$.
By the induction hypothesis, we have
\eqn
&&a_i(\lambda-\beta+\alpha_i)^{-1}\Psi(\lambda-\beta+\alpha_i)\\
&&\hs{10ex} =a_i(\lambda-\beta+\alpha_i)^{-1}a_j(\lambda-\beta+\alpha_i+\alpha_j)^{-1}
\Psi(\lambda-\beta+\alpha_i+\alpha_j),
\eneqn
and
\eqn
&&a_j(\lambda-\beta+\alpha_j)^{-1}\Psi(\lambda-\beta+\alpha_j)\\
&&\hs{10ex}
=a_j(\lambda-\beta+\alpha_j)^{-1}a_i(\lambda-\beta+\alpha_i+\alpha_j)^{-1}\Psi(\lambda-\beta+\alpha_i+\alpha_j).
\eneqn By our assumption \eqref{eqn: a_i, a_j}, the above two
quantities coincide.
\end{proof}

For $i \in I$, define $a_i\cl \rtl^- \to \cor $ as follows:
$$a_i(\beta) \seteq \prod \theta_{ji}^{-m_j}p_i^{-\la h_i,\beta \ra} \text{ for } \beta = \sum m_j\alpha_j.$$
Then we have
$$\dfrac{a_j(\beta-\alpha_i)}{a_j(\beta)}=\theta_{ji}p_i^{a_{ij}}=q_i^{a_{ij}}=q_j^{a_{ji}}
=\theta_{ij}p_j^{a_{ji}}= \dfrac{a_i(\beta-\alpha_j)}{a_i(\beta)}.$$
By Lemma \ref{Lem: define inductive}, we have a map $\Psi\cl \rtl^- \to
 \cor$
satisfying
\begin{equation} \label{eqn: dfn of Psi}
\Psi(0)=1 \ \text{ and } \ \Psi(\beta-\alpha_i)=a_i(\beta)^{-1}\Psi(\beta).
\end{equation}

We take a $\rtl$-homogeneous basis $\{ A_{\nu} \}$ of $\Uqmsug$ and its dual basis
$\{ A'_{\nu} \}$ with respect to the non-degenerate pairing in \eqref{eqn: E_i', E_i^*}.
Then we have
\begin{equation} \label{eqn: sum ot 1}
\begin{aligned}
& {\rm (i)} \ \sum_{\nu} A'_\nu \ot f_i A_\nu \hs{-0.3ex} = \hs{-0.3ex} \sum_{\nu} E_i'A'_\nu \ot A_\nu, \ \ \
\sum_{\nu} A'_\nu \ot A_\nu f_i \hs{-0.3ex} = \hs{-0.3ex} \sum_{\nu} E_i^*A'_\nu \ot A_\nu, \\
& \hs{-1ex} {\rm (ii)} \ \sum_{\nu} A'_\nu f_i \ot  A_\nu = \sum_{\nu} A'_\nu \ot E_i^* A_\nu, \ \
   \sum_{\nu} f_i A'_\nu \ot A_\nu  = \sum_{\nu} A'_\nu \ot E_i'A_\nu.
\end{aligned}
\end{equation}

\begin{proposition} \label{Prop: Phi-operator} Let $M \in \O^\P(\Uqsug)$
and set  $\Phi= \sum_{\nu} \Psi(\wt(A_\nu))A_\nu'\psi(A_\nu)$ as a
$\Uqsug$-module  endomorphism of $M$, where $\psi$ is the
automorphism in \eqref{eq: auto psi}. Then we have
$$ e_i \Phi = \Phi K_i^2 e_i, \quad \Phi f_i = f_i \Phi K_i^2 \quad \text{ for all } i \in I.$$
\end{proposition}

\begin{proof}
{}From \eqref{eq:EE*},
\eqref{eqn: E_i', E_i^*} and
\eqref{eqn: sum ot 1} (i), we obtain
\begin{align*}
\sum_{\nu} (e_iA'_{\nu} & -  (T_i^{-1}A'_{\nu}T_i)e_i) \ot A_\nu
 = \sum_{\nu} (E_i^*(T_i^{-1}A'_{\nu}T_i)K_i-K_i^{-1}E_i'(A_\nu)) \ot A_\nu \\
& = \sum_{\nu} (T_i^{-1}A'_{\nu}T_i)K_i \ot A_\nu f_i - K_i^{-1}A_\nu \ot f_i A_\nu.
\end{align*}
Thus
\eq \label{eqn: Casimir 1} &&
 \sum_{\nu} \left( e_i A'_{\nu} \ot A_{\nu}\hs{-0.5ex}-\hs{-0.5ex}(T_i^{-1} A'_{\nu} T_i) K_i \ot A_{\nu}f_i
 \right) \hs{-0.5ex}=\hs{-0.5ex}\sum_{\nu}\left( (T_i^{-1}A'_{\nu}T_i)e_i \ot
A_{\nu}\hs{-0.5ex}-\hs{-0.5ex}K_i^{-1} A_{\nu} \ot f_i A_{\nu}
\right)\hs{-0.5ex}. \eneq We define a map $ \varrho_1\cl\Uqmsug \ot
\Uqmsug \longrightarrow \Uqsug$ given by
$$ a  \ot b \longmapsto \Psi(\beta)a \psi(b), \ \ \text{ where } b \in \Uqmsug_\beta.$$

\noindent Applying $\varrho_1$, the right-hand-side of \eqref{eqn:
Casimir 1} vanishes by \eqref{eqn: dfn of Psi} as can be seen below:
\begin{align*}
& \quad \sum_{\nu}\Psi(\wt(A_\nu))(T_i^{-1}A_\nu'T_i)e_i\psi(A_\nu)
 \hs{-0.5ex}-\hs{-0.5ex}\sum_{\nu} \Psi(\wt(A_\nu)-\hs{-0.5ex}\alpha_i)K_i^{-1}A_\nu'K_ie_i\psi(A_\nu) \\
& \hs{-0.5ex}=\hs{-0.5ex} \sum_{\nu}\Psi(\wt(A_\nu))
 \prod \theta_{ji}^{-m_j^\nu} A_\nu' e_i\psi(A_\nu)
  \hs{-0.5ex}-\hs{-0.5ex}
 \sum_{\nu}\Psi(\wt(A_\nu)\hs{-0.5ex}-\hs{-0.5ex}\alpha_i)
 p_{i}^{\la h_i,\wt(A_\nu) \ra} A_\nu' e_i\psi(A_\nu) \hs{-0.5ex}=\hs{-0.5ex}0,
\end{align*}
where $\wt(A_\nu)=\sum m_j^\nu \alpha_j$.

The first term of  the left-hand-side of \eqref{eqn: Casimir 1} is
equal to $e_i\Phi$ and the second term is equal to \eq \label{eqn:
Casimir 2} &&\ba{l}
\sum_{\nu} \Psi(\wt(A_\nu)-\alpha_i)(T_i^{-1}A_\nu'T_i)K_i\psi(A_\nu)K_ie_i \\[1ex]
\hs{10ex}=\sum_{\nu} \left(\Psi(\wt(A_\nu)-\alpha_i)\prod
\theta_{ji}^{-m^\nu_j}p_i^{\la h_i,-\wt(A_\nu) \ra}\right)
    A_\nu'\psi(A_\nu)K^2_i e_i \\[1ex]
\hs{10ex} = \left(\sum_{\nu} \Psi(\wt(A_\nu))
A_\nu'\psi(A_\nu)\right)K^2_i e_i = \Phi K_i^2 e_i. \ea \eneq
Hence
we obtain $e_i\Phi=\Phi K_i^2 e_i$.

As in the
case of $e_i$'s with \eqref{eqn: sum ot 1}(ii), we have
\begin{equation} \label{eqn: sum ot 2}
\begin{aligned}
& \ \sum_{\nu} A_\nu' \ot \left(e_i A_\nu - (T_i^{-1} A_\nu T_i)e_i
\right)
 = \sum_{\nu} A_\nu'\tens \left( E_i^*(T_i^{-1} A_\nu T_i)K_i-K_i^{-1}E_i'(A_\nu) \right) \\
&\hs{30ex}
 = \sum_{\nu} A_\nu'f_i \ot (T_i^{-1} A_\nu T_i)K_i - f_i A_\nu' \ot K_i^{-1}A_\nu.
\end{aligned}
\end{equation}
By applying $x\tens y\mapsto x\tens T_i^{-1}\psi(y)T_iK_i$,
\eqref{eqn: sum ot 2} becomes
\begin{align*}
& \sum_{\nu} A_\nu' \ot T_i^{-1}f_i K_i^{-1} \psi(A_\nu)T_iK_i - A_\nu' \ot  \psi(A_\nu) f_i \\
& \qquad\qquad\qquad\qquad\qquad
= \sum_{\nu} A_\nu'f_i \ot \psi(A_\nu) -f_i  A_\nu' \ot T_i^{-1} K_i \psi(A_\nu) T_i K_i.
\end{align*}
Thus we have
\begin{equation} \label{eqn: Casimir f_i 3}
\begin{aligned}
&  \sum_{\nu} A_\nu'f_i \ot \psi(A_\nu) - A_\nu' \ot T_i^{-1}f_i K_i^{-1} \psi(A_\nu)T_iK_i\\
& \qquad\qquad\qquad = \sum_{\nu} f_i  A_\nu' \ot T_i^{-1} K_i
\psi(A_\nu) T_i K_i - A_\nu' \ot  \psi(A_\nu) f_i.
\end{aligned}
\end{equation}

Define a map $\varrho_2\cl\Uqmsug \ot \Uqmsug \longrightarrow
\Uqsug$ by
$$ a  \ot b  \longmapsto \Psi(\beta)ab, \ \ \text{ where } a \in \Uqsug^{-}_\beta.$$
The left-hand-side of \eqref{eqn: Casimir f_i 3} vanishes after
applying $\varrho_2$:
$$ \sum_{\nu} \Psi(\wt(A_\nu)-\alpha_i)A_\nu'f_i\psi(A_\nu) -
\sum_{\nu} \Psi(\wt(A_\nu)) \prod_j \theta_{ji}^{m^\nu_j}p_i^{\la
h_i,\wt(A_\nu) \ra} A_\nu' f_i \psi(A_\nu)=0,$$ and the
right-hand-side of \eqref{eqn: Casimir f_i 3} becomes \eqn&& f_i
\left(\sum_\nu \Psi(\wt(A_\nu)-\alpha_i)\prod_j
\theta_{ij}^{-m^\nu_j}p_i^{\la h_i,-\wt(A_\nu) \ra} A_\nu'
\psi(A_\nu) \right) K_i^2- \Phi f_i = f_i \Phi K_i^2- \Phi f_i,
\eneqn which completes the proof.
\end{proof}

Define an operator $\Xi$ on $M \in \mathcal{O}^\P(\Uqsug)$ such that
$$ \Xi|_{M_\lambda}=t(\lambda) q^{(\lambda+\rho|\lambda+\rho)-(\rho|\rho)}\id_{M_\lam}$$
where $t\cl \P \to \{ 1,\pi \}$ is a function satisfying
$$\dfrac{t(\lambda)}{t(\lambda-\alpha_i)}= \pi_i^{\langle h_i,\lambda \rangle}.$$
By Lemma~\ref{Lem: define inductive}, such a function $t$ uniquely
exists up to a constant multiple on a $\rtl$-orbit in $\P$. We
define the {\em quantum Casimir operator} of $\Uqsug$ by:
$$ \Omega\seteq \Phi \Xi.$$

\begin{theorem} \label{thm: quantum casimir op of Uqsug}
For any $M \in \mathcal{O}^{\P}(\Uqsug)$ and $i \in I$, we have
$$ \Omega e_i = e_i\Omega \ \text{ and } \ \Omega f_i = f_i\Omega$$
as $\Uqsug$-module endomorphisms in $M$.
\end{theorem}

\begin{proof}
For $u \in M_\lambda$,
$$ K_i^2 e_i \Xi u = t(\lambda)q^{(\lambda+\rho|\lambda+\rho)-(\rho|\rho)}
(q_i^2\pi_i)^{\la h_i,\lambda+\alpha_i \ra}e_i u.$$
On the other hand,
\begin{align*}
 \Xi e_i u &= t(\lambda+\alpha_i)q^{(\lambda+\alpha_i+\rho|\lambda+\alpha_i+\rho)-(\rho|\rho)}e_iu
= t(\lambda) \pi_i^{\la h_i,\lambda+\alpha_i\ra}
q^{(\lambda+\alpha_i+\rho|\lambda+\alpha_i+\rho)-(\rho|\rho)}e_iu\\
&= t(\lambda) \pi_i^{\la h_i,\lambda+\al_i\ra}q^{(\lambda+\alpha_i+\rho|\lambda+\alpha_i+\rho)-(\rho|\rho)}e_iu.
\end{align*}
Since
$$(\lambda+\alpha_i+\rho|\lambda+\alpha_i+\rho)-(\rho|\rho) =
(\lambda+\rho|\lambda+\rho)-(\rho|\rho)+2\dg_i\la
h_i,\lambda+\alpha_i \ra,$$ we have $K_i^2 e_i \Xi=\Xi e_i$, which
implies $e_i(\Phi \Xi) = \Phi K_i^2 e_i \Xi = (\Phi \Xi)e_i $.

The assertion for $f_i$ can be obtained in a similar way.
\end{proof}

\begin{definition} \label{dfn: primitive weight}
Let $V$ be a $\Uqsug$-module in $\mathcal{O}^{\P}(\Uqsug)$. A
vector $v \in V_\mu$ is called {\em primitive} if there exists a
$\Uqsug$-submodule $U$ in $V$ such that
$$ v \not\in U \text{ and } \Uqpsug v \in U.$$
In this case, $\mu$ is called a {\em primitive weight}.
\end{definition}

The following corollary immediately follows from
Theorem \ref{thm: quantum casimir op of Uqsug}.
\begin{corollary} \label{cor:Casimir behavior} \
\bnum
\item If $V$ is a highest weight $\Uqsug$-module with highest weight $\Lambda$, then
$$\Omega= t(\Lambda)  q^{(\Lambda+\rho|\Lambda+\rho)- (\rho|\rho)} \id_V.$$
\item
If $V$ is a $\Uqsug$-module in $\mathcal{O}^{\P}(\Uqsug)$ and $v$ is a
primitive vector with weight $\Lambda$,
 then there exists a submodule $U \subset V$ such that $v \not\in U$ and
$$\Omega(v)\equiv  t(\Lambda) q^{(\Lambda+\rho|\Lambda+\rho)- (\rho|\rho)}v \mod U.$$
\end{enumerate}
\end{corollary}

Let us take a ring homomorphism $\Z^{\sqrt{\pi}}\to\C$ and change
the base ring from $\Q(q)^{\sqrt{\pi}}$ to $\C(q)$. We then consider
$\Uqsug$ as an algebra over the field $\C(q)$. For the choice of
$\theta$ and $\pa$ given in \eqref{eq:boldU}, we denote by
$\Msu(\Lambda)=\MBe(\Lambda)$ the Verma module and
$\Vsu(\Lambda)=\VBe(\Lambda)$ the simple head of $\Msu(\Lambda)$
over $\Uqsug$, respectively.

\begin{lemma}[{\rm cf.\ \cite[Lemma 9.5, Lemma 9.6]{Kac90}}]
 \label{lem: Kac 9.5, 9.6}
Let $V$ be a non-zero $\Uqsug$-module in the category
$\mathcal{O}^\P(\Uqsug)$. \bna
\item If $\mu \ge \eta$ implies $\mu = \eta$ for any primitive
weights $\mu$ and $\eta$ of $V$, then $V$ is completely reducible.
\item For any $\lambda \in \P$, there exist a filtration
$V=V_t \supset V_{t-1} \supset \cdots \supset V_1 \supset V_0=0 $ and a subset
$J \subset \{ 1,\ldots,t \}$ such that
\bnum
\item if $j \in J$, then $V_j/V_{j-1} \simeq \Vsu(\lambda_j)$ for some $\lambda_j \ge \lambda$,
\item if $j \not\in J$, then $(V_j/V_{j-1})_{\mu}=0$ for every $\mu \ge \lambda$.
\end{enumerate}
\ee
\end{lemma}

By Corollary \ref{cor: tri dec Super}(a), we have
\begin{equation} \label{eqn: ch of M^super(Lambda)}
\ch (\Msu(\Lambda)) = \ex{\Lambda} \prod_{\alpha \in \Delta_+}
(1-\ex{-\alpha})^{-\mult(\alpha)}.
\end{equation}

\begin{proposition} [{\rm cf.\ \cite[Proposition 9.8]{Kac90}}]
Let $V$ be a $\Uqsug$-module with highest weight $\Lambda$. Then
\begin{equation} \label{eqn: character h.w module V}
\ch(V) = \sum_{ \substack{ \lambda \le \Lambda, \\ (\lambda+\rho|\lambda+\rho) = (\Lambda+\rho|\Lambda+\rho) } }
t_\lambda \ch(\Msu(\lambda)), \ \ \text{ where } t_\lambda \in \Z, \ t_\Lambda=1.
\end{equation}
\end{proposition}

\begin{proposition} [{\rm cf.\ \cite[Proposition 9.9 b)]{Kac90}}]
Let $V$ be a $\Uqsug$-module in the category
$\mathcal{O}^\P(\Uqsug)$. Assume that for any two primitive weights
$\lambda$ and $\mu$ of $V$ such that $\lambda-\mu=\beta \in \rtl^+
\setminus \{ 0 \}$, we have
$2(\lambda+\rho|\beta)\neq(\beta|\beta)$. Then $V$ is completely
reducible.
\end{proposition}

\begin{proof}
We may assume that the $\Uqsug$-module $V$ is indecomposable. Since
$\Omega$ is locally finite on $V$, i.e., every $v \in V$ is
contained in a finite-dimensional $\Omega$-invariant subspace, there
exist $\eps\in\{0,1\}$ and $a \in \Z$ such that $\Omega-\pi^\eps
q^a{\rm Id}$ is locally nilpotent on $V$. Thus Corollary
\ref{cor:Casimir behavior} (b) implies
$(\lambda+\rho|\lambda+\rho)=(\mu+\rho|\mu+\rho)$. Our assertion
follows from Lemma \ref{lem: Kac 9.5, 9.6} (a).
\end{proof}

As in \cite[Chapter 3, 9]{Kac90}, one can prove that
$\ch(\Vsu(\Lambda))$ is $W$-invariant. Thus we have the following
theorem.

\begin{theorem} \label{thm: ch V(Lambda)}
Let $\Vsu(\Lambda)$ be an irreducible $\Uqsug$-module with highest
weight $\Lambda \in \P^+$. Then the following statements hold.
\bna
\item $\ch(\Vsu(\Lambda))=\dfrac{\sum_{w \in W} \epsilon(w)
\ex{w(\Lambda+\rho)-\rho}}{\prod_{\alpha \in \Delta_+}
(1-\ex{-\alpha})^{\mult(\alpha)}}$.
\item $\Vsu(\Lambda)$ is generated by a vector $v_\Lambda$
with the defining relations:
$$  K_i v_{\Lambda} = p_i^{\langle h_i, \Lambda \rangle} v_{\Lambda}, \ \
 e_i v_{\Lambda} = 0, \ \
 f_i^{\langle h_i, \Lambda \rangle + 1} v_{\Lambda} =0 \ \
 \text{for all} \ i \in I.$$
\item The category $\Oint^\P(\Uqsug)$ is semisimple and
every simple object is isomorphic to $\Vsu(\Lambda)$ for some
$\Lambda \in \P^+$. \ee
\end{theorem}

\begin{proof}
The proofs are similar to those of \cite[Theorem 10.4, Corollary 10.4, Theorem 9.9 b)]{Kac90}.
\end{proof}

As an immediate corollary, we obtain:

\begin{theorem} {\rm Conjecture \ref{conj}} is true
if the following conditions are satisfied.
\bna
\item $(\car,\P,\Pi,\Pi^\vee)$ is a Cartan superdatum,
\item the base field $\cor$ is of characteristic $0$,
\item $q$ is algebraically independent over $\Q$,
\item there exists $\eps=\pm1$ such that
$p_{ii}\theta_{ii}^{-1}=q^{(\alpha_i|\alpha_i)}\eps^{\pa(i)}$ for
any $i\in I$.
\ee
\end{theorem}

\vskip 5mm

\section{The algebra $\Uqsg$}\label{sec: Uqsg}

In this section, we introduce an algebra $\Uqsg$ corresponding to a
Cartan superdatum, which is directly to our supercategorification
theorems via quiver Hecke superalgebras and their cyclotomic
quotients. Throughout this section, we take $\cor = \Q(q)^{\pi}$.

The algebra $\Uqsg$ is the $\cor$-algebra $\HS$, where
$\tilde{\theta}$ and $\tilde{\pa}$ are given by
\begin{equation}\label{eq:Uqsg}
\begin{aligned}
\tip_i = q_i^2\pi_i,  \quad
\tth_{ij}=\tth_{ji}=\pi^{\pa(i)\pa(j)}q_i^{-a_{ij}}.
\end{aligned}
\end{equation}
The explicit description of the algebra $\Uqsg$ is given as follows.

\begin{definition} \label{dfn:Uqsg}
The algebra $\Uqsg$ associated with a Cartan superdatum
$(\car,\P,\Pi,\Pi^\vee)$ is defined to be the algebra over $\cor =
\Q(q)^{\pi}$ generated by $e_i$, $f_i$ and $\dK_i^{\pm 1}$ $(i \in
I)$ subject to the following defining relations:
\begin{equation}
\begin{aligned}
& \dK_i\dK_j=\dK_j\dK_i, \quad \dK_i e_j \dK_i^{-1} = q_i^{2a_{ij}}
e_j, \quad \dK_i f_j \dK_i^{-1}=q_i^{-2a_{ij}}f_j, \\
& e_if_j- \pi^{\pa(i)\pa(j)}q_i^{-a_{ij}}f_je_i =\delta_{i,j}
\dfrac{1-\dK_i}{1-q_i^2\pi_i} \quad (i, j \in I), \\
& \sum_{k=0}^{1-a_{ij}}
(-\pi^{\pa(i)\pa(j)})^k\pi_i^{\frac{k(k-1)}{2}} f^{\{1-a_{ij}-k\}}_i
f_j f_i^{\{k\}} = 0 \quad (i \neq j), \\
& \sum_{k=0}^{1-a_{ij}} (-\pi^{\pa(i)\pa(j)})^k
\pi_i^{\frac{k(k-1)}{2}} e^{\{1-a_{ij}-k\}}_i e_j e_i^{\{k \}} =0
\quad (i \neq j),
\end{aligned}
\end{equation}
where $f_i^{\{n\}}=f_i^n/[n]^\pi_i!$ and
$e_i^{\{n\}}=e_i^n/[n]^\pi_i!$.
\end{definition}
Note that $\Uqsg$ has an anti-automorphism given by
\begin{equation}\label{eqn: anti-auto of Uqsg}
e_i \mapsto f_i, \quad f_i \mapsto e_i, \quad \dK_i^{\pm 1} \mapsto
\dK_i^{\pm 1}.
\end{equation}
For $\Lambda\in \P^+$, let $\Vs(\Lambda)$ be the $\P$-weighted
$\Uqsg$-module generated by $v_\Lambda$ of weight $\Lambda$ with the
defining relations given by: \eq &&\tilde{K}_i v_\Lambda =
(q_i^2\pi_i)^{\la h_i,\Lambda \ra}v_\Lambda, \quad e_i v_\Lambda=0,
\quad f_i^{\la h_i,\Lambda \ra+1} v_\Lambda=0 \quad\text{for all $i
\in I$.} \label{def:V} \eneq We define the subalgebras $\Uqmsg$,
$\Uqzsg$ and $\Uqpsg$ in the same way as we did for $\FS$ in Section
\ref{sec:FS}.

Then, by Theorem \ref{thm: ch V(Lambda)}, we obtain the following
results.

\begin{theorem} \label{thm:Oint}\hfill
\bnum
\item The $\Q(q)^{\pi}$-algebra $\Uqsg$ has a triangular decomposition
$$\Uqsg \simeq  \Uqmsg \ot \Uqzsg \ot \Uqpsg.$$
\item  $\ch (\Uqmsg) = \displaystyle
 \prod_{\alpha \in \Delta^+} (1-\ex{-\alpha})^{-\mult(\alpha)}.$
\item
For $\Lam\in\P^+$, if a $\Uqsg$-submodule $N$ of $\Vs(\La)$
satisfies $N\cap\cor v_\La=0$, then $N=0$.
\item There exist equivalences of categories
$$\Mod^\P\bl\Q(q)^{\sqrt{\pi}}\tens_{\Q(q)^{\pi}}\Uqsg\br\simeq \Mod^\P(\Uqsug),
\quad\Ocat^\P\bl\Q(q)^{\sqrt{\pi}}\tens_{\Q(q)^{\pi}}\Uqsg\br\simeq
\Ocat^\P(\Uqsug).$$
\item
The category $\Ocat^\P(\Uqsg)$ is semisimple and every simple object
is isomorphic to $\Vs(\Lambda)/(\pi-\eps)\Vs(\Lambda)$ for some
$\Lambda \in \P^+$ and $\eps=\pm1$.
\ee
\end{theorem}

For $i \in I$, $c \in \Z$ and $n \in \Z_{\ge 1}$, we define
\begin{equation} \label{dfn: K_i c,n}
\binma{x}{n}\seteq
\prod_{r=1}^{n} \dfrac{1-x(q_i^2\pi_i)^{1-r}}{1-(q_i^2\pi_i)^r}.
\end{equation}

In particular, when $n=1$, we have
$$ \binma{\tilde{K}_i}{1}= \dfrac{1-\tilde{K}_i}{1-q_i^2\pi_i}
= e_if_i- q_i^{-2} \pi_i f_ie_i.$$ Define the {\em $\A^\pi$-form}
$\Us_{\A^\pi}(\g)$ of $\Uqsg$ to be the $\A^\pi$-subalgebra of
$\Uqsg$ generated by the elements $e_i^{\{n \}}$, $f_i^{\{n\}}$,
$\tilde{K}_{i}^{\pm 1}$ for $i \in I, \, n \in \Z_{>0}$. We denote
by $\Us_{\A^\pi}^+(\g)$ (resp.\ $\Us^-_{\A^\pi}(\g)$) the
$\A^\pi$-subalgebra of $\Us_{\A^\pi}(\g)$ with $1$ generated by
$e_i^{\{n\}}$ (resp. $f_i^{\{n\}}$) and by $\Us_{\A^\pi}^0(\g)$ the
$\A^\pi$-subalgebra of $\Us_{\A^\pi}(\g)$ with $1$ generated by
$\tilde{K}_i$ and $\binma{\tilde{K}_i}{n}$ for $i \in I$,
$n\in\Z_{>0}$.

By a direct computation, we have the following lemma:
\begin{lemma} \label{Lem: a-pi tri form of Uqs}
For $i\in I$ and $n$, $m\in \Z_{\ge 0}$, we have
$$e_i^{\{n\}}f_i^{\{m\}}=\sum_{0\le k\le n,m}
q_i^{-k(k-n-m+1)}(q_i^2\pi_i)^{k(k+1)/2-nm}f_i^{\{m-k\}}e_i^{\{ n-k\}}
\binma{(q_i^2\pi_i)^{n-m}\tilde{K}_i}{k}.
$$
\end{lemma}

As an immediate consequence of Lemma \ref{Lem: a-pi tri form of
Uqs}, we have a triangular decomposition of $\Us_{\A^\pi}(\g)$.
\Lemma The homomorphism
\begin{equation}\label{eqn: a-pi tri form of Uqs}
\Us^-_{\A^\pi}(\g) \ot_{\A^\pi} \Us^0_{\A^\pi}(\g) \ot_{\A_\pi}
\Us^+_{\A^\pi}(\g)\to\Us_{\A^\pi}(\g)
\end{equation}
induced by the multiplication on $\Us(\g)$ is surjective. By
tensoring with $\Q$, we obtain an isomorphism
\begin{equation*}
\Q\tens\bl\Us^-_{\A^\pi}(\g) \ot_{\A^\pi} \Us^0_{\A^\pi}(\g) \ot_{\A^\pi}
\Us^+_{\A^\pi}(\g)\br\isoto\Q\tens\Us_{\A^\pi}(\g).
\end{equation*}
\enlemma

We will see that $\Us^-_{\A^\pi}(\g)$ is a free $\A^\pi$-module
(Corollary~\ref{cor:main2}) and that \eqref{eqn: a-pi tri form of
Uqs} is an isomorphism.

\medskip
The following proposition easily follows from
Theorem~\ref{thm:Oint}. \Prop Let $\Lam\in\P^+$. Then there exists a
unique non-degenerate symmetric bilinear form $(\ ,\ )$ on
$\Vs(\Lambda)$ such that
$$(v_\La,v_\La)=1, \ (e_iu,v)=(u,f_iv) \ \text{for all} \ u,v\in
\Vs(\Lambda), i\in I.$$
\enprop

We introduce two $\A^{\pi}$-forms of $\Vs(\Lambda)$ by
\eq&&\Vs_{\A^\pi}(\Lambda)=\Us_{\A^\pi}(\g)v_\Lam \qtext{and}
\Vs_{\A^\pi}(\Lambda)^\vee=\set{u\in \Vs(\Lambda)}%
{(u,\Vs_{\A^\pi}(\Lambda))\subset\A^\pi}.\eneq

Note that we have an isomorphism
$$ \varphi|_{\UqBg^{-}}\cl\UqBg^{-}  \to \Uqmsg.$$
By Proposition~\ref{prop:simple BqBg-module UqBg^-} and
Corollary~\ref{Cor:BKM e quasi commute}, we have
\begin{proposition} \label{prop: constant multiple 2} \
If $P \in \Uqmsg$ satisfies $e_i'P=0$ for all $i \in I$, then
$P$ is a constant multiple of $1$.
\end{proposition}

Applying the arguments given in \cite[Lemma 3.4.3, Proposition
3.4.4]{Kash91}, we obtain the following proposition immediately.

\begin{proposition} \label{prop: non-deg form of Uqsg^-}
There is a unique non-degenerate symmetric bilinear form $( \ , \ )$ on $\Uqmsg$ such that
\begin{equation} \label{eqn: non-deg form on Uqsg-}
(1,1)=1, \quad (P,f_iQ) = (e_i'P,Q)  \ \ \text{for all} \ i \in I, \
P,Q \in \Uqmsg.
\end{equation}
\end{proposition}
We define the {\em dual $\A^{\pi}$-form} of $\Us^{-}(\g)$ to be
$$\Us_{\A^{\pi}}^{-}(\g)^{\vee}: = \{ u \in \Us^{-}(\g)  \mid (u, \,
\Us_{\A^{\pi}}(\g)) \subset \A^{\pi} \}.$$

\vskip 5mm

\section{Perfect bases} \label{sec: Perfect bases}
In this section, using the notion of strong perfect bases, we prove
a theorem that characterizes
$\Vs_{\A^\pi}(\Lambda)^\vee$.

\vskip 1em

Let $V= \bigoplus_{\lambda \in \P} V_\lambda $ be a $\P$-graded
$\Q(q)^{\pi}$-module. We assume that
\begin{itemize}
\item [(i)] there are finitely many $\lambda_1,\ldots,\lambda_s \in \P$ such that
$$\wt(V) \seteq \{ \mu \in \P \ | \ V_\mu \neq 0  \} \subset \bigcup_{i=1}^s (\lambda_i - \rtl^+),$$
\item [(ii)] for each $i\in I$, there is a linear operator $e_i : V \rightarrow V$ such that
$e_i V_\lambda \subset V_{\lambda+\alpha_i}$.
\end{itemize}
For any $v \in V$ and $i \in I$, we define
\begin{itemize}
\item[(a)] $\varepsilon_i(v) \seteq \begin{cases} \min \{ n \in \Z_{\ge 0}  \ | \ e_i^{n+1}v =0 \}
& \text{ if $v \neq 0$,} \\
\qquad  - \infty & \text{if $v = 0$,} \end{cases}$
\item[(b)] $V_i^{<k}\seteq \set{v \in V}{\varepsilon_i(v) < k } = {\rm Ker} \ e_i^k$\quad for $k\ge0$.
\end{itemize}

\begin{definition}[\cite{BerKaz07,KOP11a}] \label{Def: perfect, strongly perfect}
\bnum
\item A $\Q(q)^\pi$-basis $B$ of $V$ is called a {\em perfect basis} if
\bna
\item $B=\bigsqcup_{\mu\in \wt(V)} B_{\mu}, \ \text{ where } B_{\mu}\seteq B \cap V_{\mu}$,
\item for any $ b \in B$ and $i \in I$ with $e_i(b) \neq 0$,
there exists a unique element in $B$, denoted by $\tilde{\mathsf{e}}_i(b)$, satisfying the following formula:
\begin{align*}
    e_ib -c_i(b)\;\tilde{\mathsf{e}}_i(b)\in V^{< \varepsilon_i(b)-1}_i \mbox{ for some }
    c_i(b) \in (\Q(q)^\pi)^{\times},
\end{align*}
\item if $b,b'\in B$ and $i\in I$ satisfy $\eps_i(b)=\eps_i(b')>0$
and $\tilde{\mathsf{e}}_i(b) = \tilde{\mathsf{e}}_i(b')$, then $ b = b'$.
\end{enumerate}
\item We say that a perfect basis is {\em strong} if,
for any $i\in I$ and $b\in B$ such that $e_i(b)\not=0$,
there exist some $m\in\Z$ and $\eps=0,1$ such that
$$ c_i(b)=\pi^\eps q^m[\varepsilon_i(b)]^\pi_i.$$
\end{enumerate}
\end{definition}
\noindent
Note that $[n]^\pi_i=\sum_{k=0}^{n-1}q_i^{1-n+2k}\pi_i^k$
for $n \in \Z_{> 0}$.

For any sequence ${\bf i}=(i_1,\ldots,i_m) \in I^m$ $(m \ge 1)$, we define a binary relation
$\preceq_{{\bf i}}$ on $V \setminus \{0\} $ inductively as follows:
\begin{equation*}
 \begin{aligned}
 \mbox{ if } {\bf i}=(i),\  v \preceq_{{\bf i}} v' &\Leftrightarrow \varepsilon_i(v) \le \varepsilon_i(v'), \\
 \mbox{ if } {\bf i}=(i;{\bf i}'), \ v \preceq_{{\bf i}} v'
&\Leftrightarrow
\begin{cases} \varepsilon_i(v) < \varepsilon_i(v') \ \ \mbox{ or} \\
             \varepsilon_i(v)=\varepsilon_i(v'),\; e_i^{\varepsilon_i(v)}(v) \preceq_{{\bf i'}} e_i^{\varepsilon_i(v')}(v'). \end{cases}
\end{aligned}
\end{equation*}
We write: \  (i) $v \equiv_{{\bf i}} v'$ if  $v \preceq_{{\bf i}}
v'$ and $v' \preceq_{{\bf i}} v$,  \ (ii) $ v' \prec_{{\bf i}} v$ if
$ v' \preceq_{{\bf i}} v$ and $v \not \equiv_{{\bf i}} v'$.

One can easily verify the following lemma.

\begin{lemma} \hfill

{\rm (a)} If $v \not\equiv_{{\bf i}} v'$, then
$v+v' \equiv_{{\bf i}} \begin{cases} v &\text{if } v' \prec_{{\bf i}} v , \\
                                               v' &\text{if } v \prec_{{\bf i}} v'. \end{cases}$

{\rm (b)} For all $v \in V \setminus \{ 0 \}$,
 the set $V^{\prec_{{\bf i}} v }\seteq\{ 0 \} \bigsqcup
 \{ v' \in V \setminus \{ 0 \} \mid v' \prec_{{\bf i}} v \}$ forms a
$\mathbb{Q}^\pi(q)$-module of V.
\end{lemma}

For ${\bf i}=(i_1,\ldots,i_m) \in I^m$ and $v \in V \setminus \{ 0
\}$, we define $e^\tp_{{\bf i}}$ as follows:
$$ e_i^\tp(v):= \begin{cases} e_i^{\{\varepsilon_i(v) \}}(v)
\ \ & \text{if} \ {\bf i}=(i), \\
 e^\tp_{i} \circ e_{{\bf i'}}^\tp \ \ & \text{if} \ {\bf i}=(i,{\bf
 i'}).
\end{cases}
$$
One can see that if $B$ is a strong perfect basis, then $e_{{\bf i}}^\tp B \subset (\A^\pi)^\times \cdot B$.

Let $ V^{H}\seteq \set{ v \in V }{e_iv=0 \mbox{ for all } i \in I}$
be the space of highest weight vectors in $V$ and let $ B^{H} =
V^{H} \cap B$ be the set of highest weight vectors in $B$.
Then we have
\Lemma[{\cite[Claim 5.32]{BerKaz07}}]\label{lem:ht}
The subset $B^H$ is a $\Q(q)^\pi$-basis of $V^H$.
\enlemma

\begin{proof}
Indeed, \cite{BerKaz07} treated the case when the base ring is a
field. However, since $\Q(q)^\pi \simeq \Q(q)^{\oplus2}$, we can
reduce this lemma to their case.
\end{proof}

In
\cite{BerKaz07}, Berenstein and Kazhdan proved a uniqueness theorem
for perfect bases in the following sense:

\begin{theorem}[\cite{BerKaz07}] \label{Thm: perfect morphism}
Let $B$ and $B'$ be perfect bases of $V$ such that $B^H=(B')^H$.
Then there exist a bijection $\psi\cl B \isoto B'$ and a map $\xi\cl
B \to \mathbb{Q}(q)^\times$ such that
$$\psi(b)-\xi(b)b \in V^{\prec_{{\bf i}} b}$$
for any $b \in B$ and any ${\bf i}=(i_1,\dots,i_m)$ satisfying $e_{{\bf i}}^{\tp}(b) \in V^{H}$.
Moreover, such $\psi$
and $\xi$ are unique and $\psi$ commutes with $\tilde{\mathsf{e}}_i$
and $\eps_i$ $(i \in I)$.
\end{theorem}

\begin{lemma}\label{Lem:condition for recognition}
Let $B$ be a strong perfect basis of $V$.
\bnum
\item For any finite subset $S$ of $B$, there exists
a finite sequence $\mathbf{i}=(i_1,\ldots,i_m)$ of $I$ such that
$e_{\mathbf{i}}^\tp(b)\in  (\A^\pi)^\times \cdot B^H $
for any $b\in S$.
\item Let $b_0\in B^H$ and let $\mathbf{i}=(i_1,\ldots,i_m)$
be a finite sequence in $I$. Then the set  $$S\seteq\set{b\in
B}{e_{\mathbf{i}}^\tp(b) \in (\A^\pi)^\times \cdot b_0}$$ is
linearly ordered by $\preceq_{\mathbf{i}}$. \ee
\end{lemma}

\begin{proof}
The proof is similar to the one of \cite[Lemma 2.9]{KKO12}.
\end{proof}

Now we prove the main result of this section: a characterization
theorem for $\Vs_{\A^{\pi}}(\Lambda)^\vee$.

\begin{theorem} \label{Thm:recognition thm}
Let $M$ be a $\Us(\g)$-module in $\Oint^{\P}(\Us(\g))$ such that
$\wt(M) \subset \Lambda-\rtl^+$. Suppose  $M_{\A^\pi}$ is an
$\A^\pi$-submodule of $M$ satisfying the following conditions:
\bna
\item $e_i^{\{n \}}M_{\A^\pi} \subset M_{\A^\pi}$ for any
$i\in I$,
\item
$(M_{\A^\pi})_\Lambda=\A^\pi v_\Lambda$ for some
$v_{\Lambda} \in M_{\Lambda}$,
\item $M$ has a strong perfect basis $B \subset
M_{\A^\pi}$ such that $B^H=\{ v_\Lambda \}$.
\ee
\noindent Then we have
\bnum
\item $M_{\A^\pi} \simeq \Vs_{\A^\pi}(\Lambda)^\vee$,
\item $B$ is an $\A^\pi$-basis of $M_{\A^\pi}$,
\item $\Vs_{\A^\pi}(\Lambda)_\lambda \simeq
\Hom_{\A^\pi}(\Vs_{\A^\pi}(\Lambda)^\vee_\lambda,\A^\pi)$.
\ee
\end{theorem}

\begin{proof}
Since $M$ has a unique highest weight vector $v_\Lambda$,
the  $\Us(\g)$-module $M$ is isomorphic to
$\Vs(\Lambda)$. Since $(M_{\A^\pi})_\Lambda=\A^\pi
v_\Lambda$ and
\begin{align*}
\Vs_{\A^\pi}(\Lambda)^\vee_\lambda  =\left\{ u \in \Vs(\Lambda)_\lambda \ \big| \
\begin{matrix}
e^{\{a_1 \}}_{i_1} \cdots e^{\{ a_\ell \}}_{i_\ell} u \in \A^\pi v_\Lambda
\text{ for all } (i_1,\cdots,i_\ell)  \\
\text{ such that } \sum_{k=1}^\ell a_{k} \alpha_{i_k}+\lambda= \Lambda\\ \end{matrix}
 \right \},
\end{align*}
it is clear that $M_{\A^\pi}$ is contained in
$\Vs_{\A^\pi}(\Lambda)^\vee$. Thus, in order to see (i) and (ii), it
suffices to show that $\Vs_{\A^\pi}(\Lambda)^\vee
\subset\soplus_{b\in B}\A^{\pi}b$.

For any $u \in \Vs_{\A^\pi}(\Lambda)^\vee$, we write $u = \sum_{b
\in B}c_b b$ with $c_b \in \Q(q)^{\sqrt{\pi}}$. Set $B(u) \seteq \{
b \in B \mid c_b \neq 0 \}$. By Lemma \ref{Lem:condition for
recognition}(i), there exists a sequence ${\bf i}=(i_1,\ldots,i_m)$
such that $e_{{\bf i}}^\tp(b) \in (\A^\pi)^\times v_\Lambda$ for
every $b \in B(u)$.
 Then Lemma \ref{Lem:condition for
recognition}(ii) tells that $B(u)$ is linearly ordered with respect
to $\prec_{\mathbf{i}}$. Using the descending induction, we shall
show that $c_b \in \A^\pi$. For the maximal element $\mathbf{b}$,
$e_{\mathbf{i}}^{\tp}(c_{\mathbf{b}}
\mathbf{b})=e_{\mathbf{i}}^{\tp}(u)= a_{\mathbf{b}}
c_{\mathbf{b}}\mathbf{b}$ for some $a_{\mathbf{b}} \in
(\A^\pi)^\times$. Thus we can start an induction. Assume that
$c_{b'}\in\A^\pi$ for any $b'\in B$ such that
$b\prec_{\mathbf{i}}b'$. By setting $v_0=b$,
$\ell_k=\eps_{i_k}(v_{k-1})$ and $v_k=e_{i_k}^{\{\ell_k\}}v_{k-1}$
($1\le k\le m$), we have
$$e_{i_m}^{\{ \ell_m \}}\cdots e_{i_1}^{\{ \ell_1 \}}u\hs{-0.5ex}=\hs{-0.5ex}a_b c_b v_\Lambda
\hs{-0.5ex}+
\hs{-1.5ex}\sum_{ \substack{ b\prec_{\mathbf{i}}b'\\ a_{b'}\hs{-0.5ex} \in \hs{-0.5ex} (\A^\pi)^\times}}
\hs{-1.3ex}c_{b'}e_{i_m}^{\{ \ell_m \}}\cdots
e_{i_1}^{\{ \ell_1 \}}b' \in \Vs_{\A^\pi}(\Lambda)^\vee  \ \
\text{ for some } a_b \hs{-0.5ex} \in \hs{-0.5ex} (\A^\pi)^\times,$$
 which implies $c_b\in\A^\pi$.

\smallskip
\noi (iii) follows from (i), (ii) and the lemma below.
\end{proof}

\Lemma Assume that $\Vs(\Lambda)^\vee$ has a strong perfect basis
$B$ such that $v_\La\in B$ and $B\subset
\Vs_{\A^\pi}(\Lambda)^\vee$. Then the dual basis of $B$ is an
$\A^\pi$-basis of $\Vs_{\A^\pi}(\Lambda)$. \enlemma \Proof Let
$\{b^\vee\}_{b\in B}$ be the dual basis of $B$. By the definition of
strong perfect bases, for any $\ell\in\Z_{>0}$ and $b\in B$, we can
write
$$e_i^{\{\ell\}}b=c_{b,\ell}\binmpi{\eps_i(b)}{\ell}\tilde{\mathsf{e}}_i^\ell(b)
+\sum_{\eps_i(b')<\eps_i(b)-\ell}a_{b'}b'$$
for some $a_{b'}\in \A^\pi$ and $c_{b,\ell}\in(\A^\pi)^\times$.
Hence we have
\eq &&
f_i^{\{\ell\}}\bl\tilde{\mathsf{e}}_i^\ell(b)^\vee\br=
c_{b,\ell}\binmpi{\eps_i(b)}{\ell}b^\vee+\sum_{\eps_i(b')>\eps_i(b)}a'_{b'}
(b')^\vee\label{eq:fb}
\eneq
for some $a'_{b'}\in\A^\pi$.

Since $B$ is an $\A^\pi$-basis of $\Vs_{\A^\pi}(\Lambda)^\vee$, we
have
$$\Vs_{\A^\pi}(\Lambda)\subset\soplus_{b\in B} A^\pi b^\vee.$$
Hence it is enough to show that
\eq&&
b^\vee\in\Vs_{\A^\pi}(\Lambda)
\label{eq:bin}
\eneq
 for any $\beta\in\rtl^+$ and $b\in B_{\La-\beta}$.
We shall prove it by  induction on the height $|\beta|$.
If $\beta=0$, the assertion is trivial.
Let us assume $|\beta|>0$. Then we prove \eqref{eq:bin} for $i\in I$ and
$b\in B_{\La-\beta}$ such that $\eps_i(b)>0$ by the descending induction on $
\eps_i(b)$.
Taking $\ell=\eps_i(b)$, \eqref{eq:fb} implies
$$f_i^{\{\ell\}} \bl e_i^\tp(b)^\vee\br-b^\vee\in \soplus_{\eps_i(b')>\eps_i(b)}\A^\pi
(b')^\vee.$$
Since $f_i^{\{\ell\}}\bl e_i^\tp(b)\br^\vee$ and $(b')^\vee$
belong to  $\Vs_{\A^\pi}(\Lambda)$
by the induction hypothesis, we obtain
$b^\vee\in \Vs_{\A^\pi}(\Lambda)$.
\QED

In Theorem \ref{Thm: categorical strong} and Theorem \ref{th:main1},
we will show that $\Vs_{\A^\pi}(\Lambda)^\vee$ has a strong perfect
basis.
\vskip 5mm

\section{ Supercategories and $2$-supercategories} \label{Sec:supers}

In this section, we recall the notion of supercategories,
superfunctors, superbimodules and their basic properties (see
\cite[Section 2]{KKT11}). We also introduce the notion of
$2$-supercategories.

\subsection{ Supercategories }
\begin{definition} \
\bnum
\item
A {\em supercategory} is a category $\shc$ equipped with an
endofunctor $\Pi_\shc$ of $\shc$ and an isomorphism
$\xi_\shc\cl\Pi_\shc^2\overset{\sim}{\to}{\rm id}_{\shc}$ such that
$\xi_\shc\cdot\Pi_\shc=\Pi_\shc\cdot\xi_\shc \in
\Hom(\Pi_\shc^3,\Pi_\shc)$.
\item
For a pair of supercategories $\shc$ and
$\shc'$, a {\em superfunctor} from $\shc$ to
$\shc'$ is a functor $F\cl
\shc\to\shc'$ endowed with an isomorphism $\alm{F}\cl F\cdot \Pi_{\shc}\isoto
\Pi_{\shc'}\cdot F$ such that the following diagram commutes:
\eq&&\ba{c}\xymatrix@C=7ex{
F\cdot(\Pi_{\shc})^2\ar[r]^-{\alm{F}\cdot\Pi_{\shc}}
\ar[d]^-{F\cdot \xi_\shc}&\Pi_{\shc'}\cdot
F\cdot\Pi_{\shc}
\ar[r]^-{\Pi_{\shc'}\cdot\alm{F}}
&(\Pi_{\shc'})^2\cdot F\ar[d]_{\xi_{\shc'}\cdot F}\\
F\ar[rr]^{{\rm id}_F}&&F}\ea
\eneq
If $F$ is an equivalence of categories,
we say that $(F,\alm{F})$ is an {\em equivalence of supercategories}.

\item Let $(F,\alm{F})$ and $(F',\alm{F'})$ be superfunctors
from a supercategory  $\shc$ to $\shc'$.
A morphism from $(F,\alm{F})$ to $(F',\alm{F'})$ is a
morphism of functors $\varphi\cl F\to F'$ such that
$$\xymatrix@C=8ex@R=4ex{
F\cdot\Pi_{\shc}\ar[r]^{\varphi\cdot\Pi_{\shc}}\ar[d]_{\alm{F}}
&F'\cdot\Pi_{\shc}\ar[d]_{\alm{F'}}\\
\Pi_{\shc'}\cdot F\ar[r]^{\Pi_{\shc'}\cdot\varphi}&\Pi_{\shc'}\cdot F'
}$$
commutes.
\item
For  a pair of superfunctors $F\cl\shc\to\shc'$ and $F'\cl\shc'\to \shc''$, the
composition $F'\cdot F\cl \shc\to\shc''$ of superfunctors
is defined by
taking the composition
$$\xymatrix{F'\cdot F\cdot\Pi_{\shc}\ar[r]^{F'\cdot \alm{F}}&
F'\cdot\Pi_{\shc'}\cdot F\ar[r]^{\alm{F'}\cdot F}&\Pi_{\shc''}\cdot F'\cdot F}$$
as $\alm{F'\cdot F}$.
\end{enumerate}
\end{definition}

\noindent {\em In this paper, a supercategory is assumed to be a $\cor$-linear
additive category,
where $\cor$ is a commutative ring in which $2$ is invertible.}

The functors ${\rm id}_{\shc}$ and $\Pi$ are  superfunctors by
taking  $\alm{{\rm id}_{\shc}}={\rm id}_{\Pi}\cl{\rm id}_{\shc}\cdot
\Pi \to \Pi \cdot {\rm id}_{\shc}$ and $\alm{\Pi}= - {\rm
id}_{\Pi^2}\cl\Pi \cdot \Pi \to \Pi \cdot \Pi$. {\em Note the sign}.
This is one of the main reasons that the sign is involved in
calculation in supercategories. The morphism $\alm{F}\cl F \cdot \Pi
\to \Pi \cdot F$ is a morphism of superfunctors. Note that we have
\eq &&\alm{\Pi\cdot F}=-\Pi\cdot\alm{F}\in \Hom(\Pi\cdot F\cdot
\Pi, \Pi^2\cdot F). \eneq  For a supercategory $(\shc, \Pi,\xi)$,
its {\em sign-reversed supercategory} $\shc^{\rev}$ is the
supercategory $(\shc,\Pi,-\xi)$. If $\sqrt{-1}$ exists in $\cor$,
then $\shc^\rev$ is equivalent to $\shc$ as a supercategory.

The {\em Clifford twist} of a supercategory $(\shc,\Pi,\xi)$ is the
supercategory $(\shc^{\ct},\Pi^\ct,\xi^\ct)$,  where $\shc^\ct$ is
the category whose set of objects is the set of pairs $(X,\vphi)$ of
objects $X$ of $\shc$ and isomorphisms $\vphi\cl \Pi X\isoto X$ such
that \eq&&\ba[c]{ll} \xymatrix@R=3ex@C=7ex{
&\Pi X\ar[dr]^\vphi\\
\Pi^2 X\ar[ur]^{\Pi\vphi}\ar[rr]^{\xi_X}&&X}& \qquad
\parbox{15ex}{\vs{1ex}\quad \\commutes.}\ea \label{eq:phi2} \eneq
For objects $(X,\vphi)$ and $(X',\vphi')$ of $\shc^\ct$, we define
$\Hom_{\shc^\ct}\bigl((X,\vphi),(X',\vphi')\br)$ as the subset of
$\Hom_\shc(X,X')$ consisting of morphisms $f\cl X\to X'$ such that
the following diagram commutes:
$$\xymatrix@C=8ex@R=4ex{
\Pi X\ar[r]^{\Pi f}\ar[d]_\vphi&\Pi X'\ar[d]^{\vphi'}\\
X\ar[r]^f&X'.}
$$
We define $\Pi_{\shc^\ct}\cl \shc^\ct \to \shc^\ct$ and $\xi_{\shc^\ct}\cl
(\Pi_{\shc^\ct})^2\isoto\id_{\shc^\ct}$ by
$$
\begin{aligned}
& \Pi_{\shc^\ct}(X,\vphi)=(X,-\vphi), \\
& \xi_{\shc^\ct}{(X, \vphi)}= \id_{(X, \vphi)}:
(\Pi_{\shc^\ct})^2(X,\vphi)=(X,\vphi) \to (X,\vphi).
\end{aligned}
$$

We have morphisms of superfunctors
$$\text{$\shc^\rev\to\shc^\ct$ and $\shc^\ct\to\shc^\rev$.}$$
If $\shc$ is idempotent complete (i.e., any endomorphism $f$ of an object $X\in\shc$ such that $f^2=f$ has a kernel in $\shc$), then we have an equivalence of
supercategories \eq&&(\shc^\ct)^\ct\simeq\shc.\eneq

\subsection{Superbifunctors}

\Def {\rm Let $\shc$, $\shc'$ and $\shc''$ be supercategories. A
\emph{superbifunctor} $F\cl \shc\times\shc'\to\shc$ is a bifunctor
endowed with isomorphisms
$$\alm{F}(X,Y)\cl F(\Pi X,Y)\isoto \Pi F(X,Y)
\qtext{and} \beta_F(X,Y)\cl F(X,\Pi Y)\isoto \Pi F(X,Y)$$ which are
functorial in $X\in\shc$ and $Y\in \shc'$ such that the two diagrams
$$\xymatrix@C=10ex{
F(\Pi^2X,Y)\ar[r]^-{\alm{F}(\Pi X,Y)}\ar[dr]_-{\xi_{\shc}}&\Pi F(\Pi X,Y)
\ar[r]^-{\Pi\cdot\alm{F}(X,Y)}
&\Pi^2F(X,Y)\ar[dl]^-{\xi_{\shc''}}\\
&F(X,Y)
}
$$
and
$$\xymatrix@C=10ex{
F(X,\Pi^2Y)\ar[r]^-{\beta_F(X,\Pi Y)}\ar[dr]_-{\xi_{\shc'}}&\Pi F(X,\Pi Y)
\ar[r]^-{\Pi\cdot\beta_F(X,Y)}
&\Pi^2F(X,Y)\ar[dl]^-{\xi_{\shc''}}\\
&F(X,Y)
}
$$
commute, and the diagram \eq&&\ba{l} \xymatrix@C=15ex{ F(\Pi X,\Pi
Y)\ar[r]^{\beta_F(\Pi X,Y)}\ar[d]_-{\alm{F}(X,\Pi Y)}\ar@{}[dr]|-
{-}
&\Pi F(\Pi X,Y)\ar[d]^-{\Pi\cdot\alm{F}(X,Y)}\\
\Pi F( X,\Pi Y)\ar[r]_{\Pi\cdot\beta_F(X,Y)}&\Pi^2 F(X,Y)
}\ea\label{dia:antb}\eneq {\em anti-commutes}. } \edf

Let $F\cl\shc\times\shc'\to\shc''$ be a superbifunctor
of supercategories.
Then we can check that
$F$ induces superbifunctors
\eqn
&&\shc^\rev\times\shc'\,{}^\rev\to\shc''\,{}^\rev,\\
&&\shc^\ct\times\shc'\,{}^\rev\to\shc''\,{}^\ct.
\eneqn

Let $\shc$ and $\shc'$ be a pair of supercategories. We denote by
$\Fcts(\shc,\shc')$ the category of superfunctors from $\shc$ to
$\shc'$. This category is endowed with a structure of supercategory
by: \eqn &&\Pi(F,\alm{F})\seteq(\Pi_{\shc'},\alm{\Pi_{\shc'}})\cdot
(F,\alm{F})
=(\Pi_{\shc'}\cdot F,-\Pi_{\shc'}\cdot\alm{F}), \\
&&\xi(F,\alm{F})\seteq\xi_{\shc'}\cdot F\, :\,\Pi{}^2(F,\alm{F})
=((\Pi_{\shc'})^2\cdot F,\;(\Pi_{\shc'})^2\cdot \alm{F})\isoto (F,\alm{F}).
\eneqn
Note the sign in the definition of
$\Pi(F,\alm{F})$.

Let $\shc''$ be another supercategory.
Then we have the following proposition.
Since the proof is routine,
we just remark that
 the anti-commutativity of \eqref{dia:antb}
follows from $\alm{\Pi\cdot F}=-\Pi\cdot\alm{F}\in\Hom(\Pi F\Pi,\Pi^2F)$,
and we omit the details.

\Prop \hfill
\bnum\item The bifunctor
$\Fcts(\shc,\shc')\times\shc\to\shc'$, $(F,X)\mapsto F(X)$ is
endowed with a structure of  superbifunctor by: \eqn&&
\parbox{75ex}{$\al(F,X)\cl (\Pi\cdot F)(X)\isoto \Pi_{\shc'}(F(X))$
is the canonical isomorphism, \\
$\beta(F,X)\cl F(\Pi_\shc X)\to  \Pi_{\shc'}(F(X))$
is $\alm{F}(X)$.}
\eneqn
\item
The bifunctor
$\Fcts(\shc',\shc'')\times\Fcts(\shc.\shc')
\to \Fcts(\shc,\shc'')$ ,  $(G,F)\mapsto G\cdot F$,
is endowed with a structure of superbifunctor by:
\eqn&&
\parbox{75ex}{$\al(G,F)\cl (\Pi\cdot G)\cdot F\isoto \Pi\cdot (G\cdot F)$
is the canonical isomorphism, \\
$\beta(G,F)\cl G\cdot (\Pi\cdot F)\isoto \Pi\cdot (G\cdot F)$
is $\alm{G}\cdot F$.}
\eneqn
\ee
\enprop

The following proposition is also obvious.
\Prop Let $\shc$, $\shc'$ and $\shc''$ be supercategories.
A superbifunctor $\shc\times\shc'\to\shc''$
induces superfunctors
$$\shc\to\Fcts(\shc',\shc'')\qtext{and}\shc'\to\Fcts(\shc,\shc'').$$
\item
Conversely, a superfunctor $\shc\to\Fcts(\shc',\shc'')$ induces a
superbifunctor $\shc\times\shc'\to\shc''$. \enprop

Note that we have equivalences of supercategories:
$$
\begin{aligned}
& \Fcts(\shc^\ct,{\shc'}^\ct)\simeq \Fcts(\shc,\shc')^\rev, \\
& \Fcts(\shc^\rev,\shc'{}^\rev)\simeq\Fcts(\shc,\shc')^\rev.
\end{aligned}
$$

\subsection{Even and odd morphisms}
Let $(\shc,\Pi,\xi)$ be a supercategory. Let us denote by
$\sd{\shc}$ the category defined by $\Ob(\sd{\shc})=\Ob(\shc)$ and
$\Hom_{\sd{\shc}}(X,Y)=\Hom_{\shc}(X,Y)\oplus \Hom_{\shc}(X,\Pi Y)$.
The composition of $f\in \Hom_{\shc}(Y,\Pi^\eps
Z)\subset\Hom_{\sd{\shc}}(Y,Z)$ and $g\in \Hom_{\shc}(X,\Pi^{\eps'}
Y)\subset\Hom_{\sd{\shc}}(X,Y)$ ($\eps,\eps'=0,1$) is defined by
$X\To[g]\Pi^{\eps'} Y\To[{\Pi^{\eps'}f}]\Pi^{\eps+\eps'} Z$
(composed with $\Pi^2Z\isoto[\xi]Z$ when $\eps=\eps'=1$). Hence
$\Hom_{\sd{\shc}}(X,Y)$ has a structure of superspace, where
$\Hom_{\shc}(X,Y)$ is the even part and $\Hom_{\shc}(X,\Pi Y)$ is
the odd part. A morphism $X\to \Pi Y$ in $\shc$ is sometimes called
an {\em odd morphism} (in $\sd{\shc}$) from $X$ to $Y$.

The category $\sd{\shc}$ has a structure of supercategory. The
functor $\Pi_{\sd{\shc}}$ is defined as follows. For $X\in\shc$,
define $\Pi_{\sd{\shc}}(X)=X$. For $X,Y\in\shc$, the map
$\Pi_{\sd{\shc}}\cl\Hom_{\sd{\shc}}(X,Y) \to \Hom_{\sd{\shc}}\bl
\Pi_{\sd{\shc}}(X),\Pi_{\sd{\shc}}(Y)\br=\Hom_{\sd{\shc}}(X,Y)$ is
defined by
$$\Pi_{\sd{\shc}}\vert_{\Hom_{\shc}(X,\Pi^\eps Y)}
=(-1)^\eps\,\id_{\Hom_{\shc}(X,\Pi^\eps Y)} \qtext{for
$\eps=0,1$.}$$ The morphism $\xi_X\cl(\Pi_{\sd{\shc}})^2X\to X$ is
defined to be $\id_X$. Note that $\sd{\shc}$ is not idempotent
complete in general eve if $\shc$ is abelian.

There exists a canonical functor $\shc\to\sd{\shc}$ that
we denote by $X\mapsto \sd{X}$.
It has a structure of superfunctor by the isomorphism
$\alm{\sd{}}\cl\sd{}\cdot\Pi_{\shc}\isoto \Pi_{\sd{\shc}}\cdot\sd{}$ defined by
$(\alm{\sd{}})(X)=\id_{\Pi X}$, where
$(\alm{\sd{}})(X)\cl \sd{(\Pi X)}\isoto \Pi_{\sd{\shc}}(\sd{X})=\sd{X}$.

We can easily verify the following lemma.

\Lemma \label{lem:Dequi}
Let $\shc$ and $\shc'$ be supercategories.
\bnum\item There exists a canonical equivalence of supercategories
$$\Fcts(\shc,\shc')\isoto\Fcts(\sd{\shc},\sd{\shc'}).$$
\ro We denote it by $F\mapsto \sd{F}$.\rf
\item
We have $\sd{(\Pi_{\shc})}\simeq\Pi_{\sd{\shc}}$
as a superfunctor from $\shc^D$ to $\shc^D$.
\ee
\enlemma

\Lemma  Let $\shc$ and $\shc'$ be supercategories,
and let $\vphi\cl \sd{F}\to \sd{G}$ be a morphism in
$\sd{\Fcts(\shc,\shc')}_\eps$
and $f\cl X\to Y$be a morphism in
$\Hom_{\sd{\shc}}(X,Y)_{\eps'}$ $(\eps,\eps'=0,1)$.
Then the following diagram supercommutes:
$$\xymatrix@C=10ex@R=5ex{
\sd{F}(X)\ar[r]^{\sd{F}(f)}\ar[d]_{\vphi(X)}\ar@{}[dr]|{(-1)^{\eps\eps'}}
&\sd{F}(Y)\ar[d]^{\vphi(Y)}\\
\sd{G}(X)\ar[r]_{\sd{G}(f)}&\sd{G}(Y),
}
$$
i.e.\ $\vphi(Y)\circ \sd{F}(f)=(-1)^{\eps\eps'}\sd{G}(f)\circ\vphi(X)$.
\enlemma
\Proof
We denote by the same letters $\ol\vphi$ and $\ol f$ the morphisms
$\ol\vphi\cl F\to \Pi^\eps G$ and $\ol f\cl X\to \Pi^{\eps'}Y$
corresponding to $\vphi$ and $f$, respectively.
Then the result follows from the following commutative diagram in $\shc'$
$$\xymatrix@C=7ex@R=4ex{
F(X)\ar[r]^-{F(\ol{f})}\ar[d]_{\ol\vphi(X)}&F(\Pi^{\eps'}Y)
\ar[rr]^-{(\alm{F})^{\eps'}}
\ar[d]^{\ol\vphi(\Pi^{\eps'}Y)}&&
\Pi^{\eps'}F(Y)\ar[d]^{\ol\vphi(Y)}\\
\Pi^\eps G(X)\ar[r]_-{G(\ol f)}&\Pi^\eps G(\Pi^{\eps'}Y)
\ar[r]_-{(\alm{G})^{\eps'}}&
\Pi^\eps \Pi^{\eps'}G(Y)\ar[r]_-{(\alm{\Pi^\eps})^{\eps'}}
&\Pi^{\eps'} \Pi^{\eps}G(Y)
}
$$
and $(\al_{\Pi^\eps})^{\eps'}=(-1)^{\eps\eps'}\id_{\Pi^{\eps+\eps'}G(Y)}$.
\QED

\subsection{$2$-supercategories}\label{subsec:super2}

In this subsection, we give a definition of {\em
$2$-supercategories}. We only consider additive $2$-supercategories
over a base ring $\cor$ in which $2$ is invertible.

\Def {\rm A {\em $1$-supercategory} is a $\cor$-linear category
$\shc$ such that $\Hom_{\shc}(X,Y)$ is endowed with a structure of
$\cor$-supermodule for $X,Y\in\shc$ and the composition map
$\Hom_{\shc}(Y,Z)\times\Hom_{\shc}(X,Y)\to\Hom_{\shc}(X,Z)$ is
$\cor$-superbilinear.} \edf

We say that a morphism $f\cl X\to Y$ is even or odd according as $f$
belongs to the even part or the odd part of $\Hom(X,Y)$. For a
supercategory $\shc$, the category $\sd{\shc}$ is a
$1$-supercategory.

For a diagram
\eq
&&\ba{c}\xymatrix@C=7ex{
X\ar[r]^{f}\ar[d]_{\vphi}&Y\ar[d]^{\psi}\\
X'\ar[r]_{f'}&Y'}\ea \label{dia:super} \eneq with
$f\in\Hom_{\shc}(X,Y)_\eps$, $f'\in\Hom_{\shc}(X',Y')_{\eps}$ and
$\vphi\in\Hom_{\shc}(X,X')_{\eps'}$,
$\psi\in\Hom_{\shc}(Y,Y')_{\eps'}$ with $\eps,\eps'=0,1$, we say
that the diagram \eqref{dia:super} {\em supercommutes} or sometimes
{\em $(-1)^{\eps\eps'}$-commutes} if $\psi\circ
f=(-1)^{\eps\eps'}f'\circ \vphi$.

\medskip
For a pair $\shc$, $\shc'$ of super-$1$-categories, the notion of a
superfunctor from $\shc$ to $\shc'$ is naturally defined, and we do
not write it.  However, as for morphisms of functors and bifunctors,
we need a special care. \Def Let $\shc$ and  $\shc'$ be
$1$-supercategories and let $F,G\cl\shc\to\shc'$ be two
superfunctors.
 An {\em even} \ro resp.\ {\em odd}\rf\ {\em morphism}
$\vphi\cl F\to G$ is the data
associating
an even (resp.\ odd) morphism $\vphi(X)\cl F(X)\to G(X)$ to any
$X\in\shc$
such that
the diagram
$$\xymatrix@C=8ex{
F(X)\ar[r]^{F(f)}\ar[d]^{\vphi(X)}&F(Y)\ar[d]^{\vphi(Y)}\\
G(X)\ar[r]^{G(f)}&G(Y)}
$$
supercommutes for any $X,Y\in\shc$ and $f\in\Hom_\shc(X,Y)_\eps$ $(\eps=0,1)$.
\edf Then the superfunctors from $\shc$ to $\shc'$ and the morphisms
of superfunctors form a $1$-supercategory, which we denote by
$\Fcts(\shc,\shc')$.

\Def  Let $\shc$, $\shc'$, $\shc''$ be three
$1$-supercategories. A {\em superbifunctor}
$F\cl\shc\times\shc'\to\shc''$ is the data \bnum
\item
a map $\Ob(\shc)\times\Ob(\shc')\to \Ob(\shc'')$,
\item a $\cor$-linear even map $F(\scbul,Y)\cl
\Hom_{\shc}(X,X')\to \Hom_{\shc''}\bl F(X,Y),F(X',Y)\br$
for $X,X'\in\shc$ and $Y\in\shc'$,
\item a $\cor$-linear even map $F(X,\scbul)\cl
\Hom_{\shc'}(Y,Y')\to \Hom_{\shc''}\bl F(X,Y),F(X,Y')\br$
for $X\in\shc$ and $Y,Y'\in\shc'$,
\ee
such that
\bna\item
$F(\scbul,Y)\cl \shc\to\shc''$ and $F(X,\scbul)\cl \shc'\to\shc''$
are superfunctors,
\item
 as
elements of $\Hom_{\shc''}\bl F(X,Y),F(X',Y')\br$, we have
$$F(f,Y')\circ F(X,g)=(-1)^{\eps\eps'}F(X',g)\circ F(f,Y)$$ for
$X,X'\in\shc$, $f\in\Hom_{\shc}(X,X')_\eps$ and $Y,Y'\in\shc'$,
$g\in\Hom_{\shc'}(Y,Y')_{\eps'}$. \ee \edf

The following propositions are easy to verify.
\Prop
For $1$-supercategories $\shc$,$\shc'$ and $\shc''$,
the composition $(F,G)\mapsto F\cdot G$ gives a superbifunctor
$\Fcts(\shc',\shc'')\times\Fcts(\shc,\shc')\To\Fcts(\shc,\shc'')$
of $1$-supercategories.
\enprop

\Prop\label{prop:Dequiv} \hfill \bnum\item Let $F\cl \shc\to\shc'$
be a superfunctor of supercategories. Then  it induces a
superfunctor $\sd{F}\cl\sd{\shc}\to\sd{\shc'}$ of
$1$-supercategories.

 Moreover we have an equivalence of $1$-supercategories
$$\sd{\Fcts(\shc,\shc')}\isoto\Fcts(\sd{\shc},\sd{\shc'}).$$
\item
Let $F\cl \shc\times\shc' \to\shc''$ be a superbifunctor of
supercategories. Then  it induces a superbifunctor
$\sd{F}\cl\sd{\shc}\times \sd{\shc'}\to\sd{\shc''}$ of
$1$-supercategories. \ee \enprop

\Def {\rm A {\em $2$-supercategory} $\FA$ is the data of \bnum
\item a set $\FA$ of objects,
\item a $1$-supercategory $\CHom_\FA(a,a')$ for $a,a'\in\FA$,
\item a superbifunctor $\CHom_\FA(a_2,a_3)\times\CHom_\FA(a_1,a_2)\to
\CHom_\FA(a_1,a_3)$, $(b_2,b_1)\mapsto b_2b_1$ for $a_1,a_2,a_3\in\FA$,
\item an object $\one_a\in\CEnd_\FA(a)$ for $a\in\FA$,
\item a natural even isomorphism
$$\can(b_3,b_2,b_1)\cl(b_3b_2)b_1\isoto b_3(b_2b_1)$$
for $a_k\in\FA$ and $b_i\in\CHom_\FA(a_i,a_{i+1})$ $(k=1,\ldots,4$, $i=1,2,3)$,
\item natural even isomorphisms $$b \one_a\isoto b \quad \text{and}
\quad \one_{a'} b\isoto b$$ for $a,a'\in\FA$ and
$b\in\CHom_\FA(a,a')$

\ee such that the following diagrams are commutative.
$$\xymatrix@R=2ex{
& \bl (b_4b_3)b_2\br b_1 \ar[rr]^-{\can(b_4,b_3,b_2)\cdot b_1}
  \ar[dl]^(.4){\hs{2ex} \can(b_4b_3,b_2,b_1)} &&
 \bl b_4(b_3b_2)\br b_1\ar[dr]_(.4){\can(b_4,b_3b_2,b_1)\hs{3ex}} \\
(b_4b_3)(b_2b_1) \ar[drr]_{\can(b_4,b_3,b_2b_1)} &&&& b_4\bl (b_3b_2)b_1\br
  \ar[dll]^{b_4\cdot\can(b_3,b_2,b_1)}\\
&& b_4\bl b_3(b_2b_1)\br
}$$
$$\xymatrix@C=8ex@R=3ex{
(b_2 \one_a)b_1 \ar[rr]^{\can(b_2,I_a,b_1)} \ar[dr]&&
 b_2(\one_ab_1) \ar[dl]\\
& b_2b_1 }$$ } \edf

\begin{example} \hfill
\bnum
\item
Let the set of objects of $\FA$ be the set of supercategories. For
supercategories $\shc$ and $\shc'$, set
$\CHom_{\FA}(\shc,\shc')=\sd{\Fcts(\shc,\shc')}$. Then $\FA$ becomes
a $2$-supercategory.

\item
Let the set of objects of $\FA$ be the set of $\cor$-superalgebras.
Let $A$, $B$, $C$ be $\cor$-superalgebras. Set
$\CHom_{\FA}(A,B)=\sd{\Mods(B,A)}$ and define the bifunctor
$$\CHom_{\FA}(B,C)\times\CHom_{\FA}(A,B)\to\CHom_{\FA}(A,C) \quad
\text{by} \quad (K,L)\mapsto K\tens_BL.$$ Then $\FA$ is a
$2$-supercategory. (See \S\,\ref{subsec:superbimodules} below.)\ee
\end{example}

Let $\FA$ be a $2$-supercategory. The objects (resp.\ morphisms) of
$\CHom_{\FA}(a,a')$ are referred to as {\em $1$-arrows} (resp.\ {\em
$2$-arrows}). Let $b\cl a\to a'$ be a $1$-arrow. A {\em right
superadjoint} of $b$ is a $1$-arrow $b^\vee\cl a'\to a$ with even
$2$-arrows $\eps\cl bb^\vee\to\one_{a'}$ and $\eta\cl \one_a\to
b^\vee b $ such that
$$b\isoto b\one_a\To[ b\eta]bb^\vee b\To[\eps\,b]\one_{a'}b\isoto b$$
and
$$b^\vee\isoto \one_ab^\vee\To[\eta b^\vee]b^\vee bb^\vee\To[b^\vee\eps]
b^\vee\one_{a'}\isoto b^\vee$$ are the identities. If a right
superadjoint exists, then it is unique up to a unique even
isomorphism. We call $(b,b^\vee)$ a superadjoint pair and
$(\eps,\eta)$ the {\em superadjunction}.

Let $b,b'\cl a\to a'$ be a pair of $1$-arrows,
and assume that they admit right superadjoints
with superadjunctions $(\eps,\eta)$ and $(\eps',\eta')$.
Then we have an even isomorphism
$$\Hom_{\CHom(a,a')}(b,b')\isoto\Hom_{\CHom(a',a)}(b'\,{}^\vee,b^\vee)\qquad (f\mapsto f^\vee).$$
Here, $f^\vee$ is given by the composition
$$b'{}^\vee\isoto \one_ab'{}^\vee\To[\eta b'{}^\vee]b^\vee bb'{}^\vee
\To[f]b^\vee b'b'{}^\vee
\To[b^\vee\eps']
b^\vee\one_{a'}\isoto b^\vee.$$

\Prop Let $b_1,b_2,b_3$ be $1$-arrows from $a$ to $a'$. Assume that
they admit right superadjoints. For $f\in \Hom(b_1,b_2)_\eps$ and
$g\in \Hom(b_2,b_3)_{\eps'}$ with $\eps,\eps'=0,1$, we have
$$(g\circ f)^\vee=(-1)^{\eps\eps'}f^\vee\circ g^\vee.$$
\enprop \Proof Let $(\eps_k,\eta_k)$ be the superadjunction for
$b_k$ $(k=1,2,3)$. Then we have a diagram in $\CHom(a',a)$
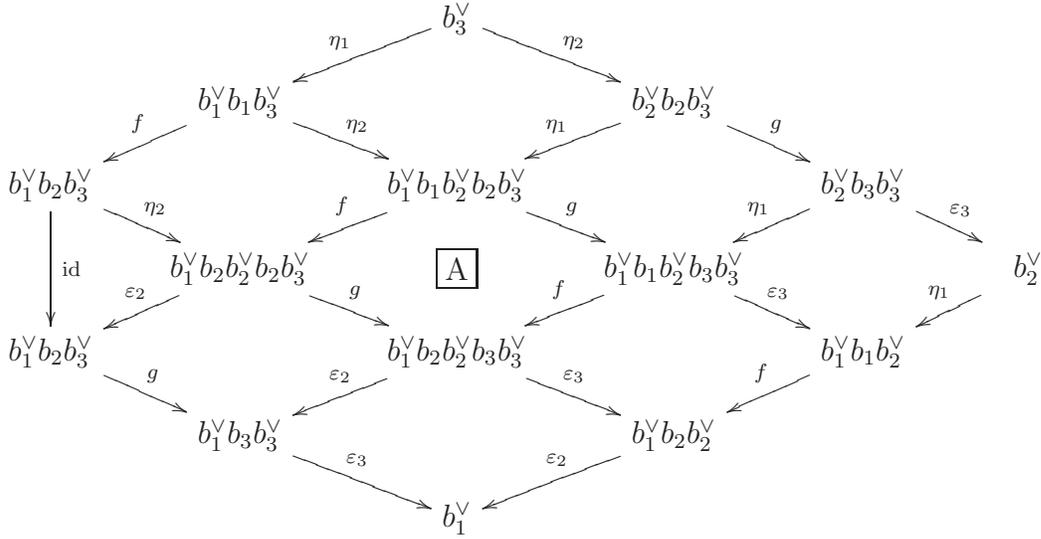
$$\xymatrix@C=4ex@R=2.2ex{
&&b^\vee_3\ar[dl]_{\eta_1}\ar[dr]^{\eta_2}\\
&b^\vee_1b_1b^\vee_3\ar[dl]_{f}\ar[dr]^{\eta_2}&&b^\vee_2b_2b^\vee_3
\ar[dl]_{\eta_1}\ar[dr]^{g}\\
b^\vee_1b_2b^\vee_3\ar[dd]^{\id}\ar[dr]^{\eta_2}&&b^\vee_1b_1b^\vee_2b_2b^\vee_3
\ar[dl]_{f}\ar[dr]^{g}\ar@{}[dd]|{\fbox{A}}&&b^\vee_2b_3b^\vee_3\ar[dl]_{\eta_1}\ar[dr]^{\eps_3}\\
&b^\vee_1b_2b^\vee_2b_2b^\vee_3\ar[dl]_{\eps_2}\ar[dr]^{g}&&b^\vee_1b_1b^\vee_2b_3b^\vee_3
\ar[dl]_{f}\ar[dr]^{\eps_3}&&{\makebox[8ex][l]{$\quad b^\vee_2$}}\ar[dl]_{\eta_1}\\
b^\vee_1b_2b^\vee_3\ar[dr]^{g}&&b^\vee_1b_2b^\vee_2b_3b^\vee_3\ar[dl]_{\eps_2}\ar[dr]^{\eps_3}&&
b^\vee_1b_1b^\vee_2\ar[dl]_{f}\\
&b^\vee_1b_3b^\vee_3\ar[dr]^{\eps_3}&&b^\vee_1b_2b^\vee_2\ar[dl]_{\eps_2}\\
&&b^\vee_1 }$$
Here, $\eps_k$ and $\eta_k$ are even morphisms. Hence
all the squares are commutative except that the central square
$\fbox{A}$ is $(-1)^{\eps\eps'}$-commutative.

By the definition, $(g\circ f)^\vee$ is the composition of the left
most arrows, and $f^\vee\circ g^\vee$ is the composition of the
rightmost arrows. Hence we obtain the desired result. \QED

\begin{remark}
As seen in Lemma~\ref{lem:Dequi} (i) and
Proposition~\ref{prop:Dequiv}, the notion of supercategories and
that of super-$1$-categories are almost equivalent. Hence, although
we can define the notion of a $2$-category using the condition that
$\CHom_\FA(a,a')$ are supercategories, those two definitions are
almost equivalent.
\end{remark}

\subsection{Superalgebras and superbimodules}\label{subsec:superbimodules}

Recall that a $\cor$-superalgebra is a $\Z_2$-graded $\cor$-algebra.
Let $A=A_0 \oplus A_1$ be a superalgebra. We denote by $\phi_{A}$
the involution of $A$ given by
$$\phi_{A}(a)=(-1)^{\epsilon}a \quad
\text{for} \ \ a\in A_\epsilon, \ \epsilon=0,1.$$ We call $\phi_A$
the {\em parity involution} of the superalgebra $A$. An {\em
$A$-supermodule} is an $A$-module with a decomposition $M= M_0
\oplus M_1$ such that $A_\epsilon M_{\epsilon'} \subset
M_{\epsilon+\epsilon'}$ ($\epsilon,\epsilon' \in \Z_2$).  For an
$A$-supermodule $M$, we denote by $\phi_M\cl M\to M$ the involution
of $M$ given by
$\phi_M|_{M_\epsilon}=(-1)^\epsilon\,\id_{M_\epsilon}$. We call
$\phi_M$ the {\em parity involution} of the $A$-supermodule $M$.
Then we have $\phi_M(ax)=\phi_{A}(a)\phi_M(x)$ for any $a\in A$ and
$x\in M$.

Let $A$ and $B$ be $\cor$-superalgebras. We define the multiplication on
the tensor product $A \otimes_\cor B$  by
\begin{align} \label{eq:tensor superal}
(a_1\otimes b_1)(a_2\otimes b_2)=(-1)^{\eps'_1\eps_2}(a_1a_2)\otimes (b_1b_2)
\end{align}
for $a_i\in A_{\eps_i}$, $b_i\in B_{\eps'_i}$
($\eps_i,\eps_{i}'=0,1$). If $M$ is an $A$-supermodule and $N$ is a
$B$-supermodule, then $M\otimes_\cor N$ has a structure of
$A\otimes_\cor B$-supermodule by
\begin{align*}
(a\otimes b)(u\otimes v)=(-1)^{\eps\eps'}(au)\otimes (bv)
\end{align*}
for $a\in A$, $b\in B_{\eps}$, $u\in M_{\eps'}$, $v\in N$ ($\eps,\eps'=0,1$).

\begin{example} \label{exa:supercategories} Let $A$ be a $\cor$-superalgebra.
\bna
\item Let $\Mod(A)$ be the category of $A$-modules.
Then  $\Mod(A)$ is endowed with a supercategory structure induced by
the parity involution $\phi_A$; i.e., for $M\in\Mod(A)$, we have
$$
\begin{aligned}
& \Pi M  \seteq\set{ \pi(x)}{x \in M }, \ \
\pi(x)+\pi(x')=\pi(x+x'),
\\
& a \cdot \pi(x)  \seteq \pi(\phi_A(a) \cdot x) \ \ (a\in A, \, x,
x' \in M).
\end{aligned}$$
The isomorphism $\xi_M\cl \Pi^2M\to M$ is given by
$\pi\left(\pi(x)\right)\mapsto x$  ($x\in M$).

\item Let $\MOD(A)$ be the category of $A$-supermodules. The
morphisms in this category are $A$-module homomorphisms which
preserve the $\Z_2$-grading. Then $\MOD(A)$ has a supercategory structure
induced by the {\em parity shift}; i.e.,
\begin{align*}
& (\Pi M)_\epsilon\seteq\set{\pi(x)}{x\in M_{1-\epsilon}} \quad  (\epsilon=0,1), \\
& a\cdot\pi(x)\seteq\pi(\phi_{A}(a)\cdot x) \quad (a\in A, \, x\in
M).
\end{align*}
The isomorphism $\xi_M\cl\Pi^2M\to M$ is also given by
$\pi\left(\pi(x)\right)\mapsto x$.
\ee
\end{example}

Let $A$ be a $\cor$-superalgebra. The {\em sign-reversed}
$\cor$-superalgebra of $A$ is defined to be the $\cor$-superalgebra
$A^\rev\seteq\set{a^\rev}{a\in A}$ which is isomorphic to $A$ as a
$\cor$-supermodule with the multiplication given by
$$a^\rev\; b^\rev=(-1)^{\eps\eps'}(ab)^\rev \quad \text{for} \ \ a \in A_{\eps}, \,
b \in A_{\eps'}, \ \eps,\eps'=0,1. $$ For an $A$-supermodule $M$,
let $M^\rev\seteq\set{u^\rev}{u\in M}$ be the $A^\rev$-module with
the action given by
$$a^\rev\; u^\rev=(-1)^{\eps\eps'}(au)^\rev \quad \text{for} \ \
a \in A_{\eps}, \, u \in M_{\eps'}, \ \eps,\eps'=0,1. $$ We remark
that if $\cor$ contains $\sqrt{-1}$, then $A^\rev$ is (non
canonically) isomorphic to $A$ by $a^\rev\mapsto (\sqrt{-1})^\eps a$
for $\eps=0,1$ and $a\in A_\eps$.

\Lemma We have equivalences of supercategories:
$$\Mod_\super(A^\rev)\simeq \Mod_\super(A)^\rev\simeq\Mod(A)^\ct.$$
\enlemma
\Proof
The right equivalence is proved in \cite[Section 2]{KKT11}.
Let $M\mapsto M^\rev$ be an equivalence of categories
from $\Mod_\super(A)$ to $\Mod_\super(A^\rev)$.
We give an isomorphism
$$(\Pi M)^\rev\simeq \Pi (M^\rev)$$
by $\bl\pi(x)\br^\rev\mapsto \pi\bl\phi_M(x)^\rev\br$. We can check
easily that it gives an equivalence of supercategories from
$\Mod_\super(A)^\rev$ to $\Mod_\super(A^\rev)$. \QED

Let $A$ be a $\cor$-superalgebra.
Let us denote by $A^\sop$ the opposite superalgebra of $A$.
By definition, it is the superalgebra $(A^\sop)_\eps\seteq\set{a^\sop}{a\in A_\eps}$ ($\eps=0,1$)
with $a^\sop\; b^\sop=(-1)^{\eps\eps'}(ba)^\sop$ for $a\in A_\eps$ and $b\in A_{\eps'}$.
Then a right $A$-supermodule $M$ may be regarded as a left $A^\sop$-supermodule by
$a^\sop\; x=(-1)^{\eps\eps'}x a$ for $a\in A_\eps$ and $x\in M_{\eps'}$.
We should not confuse $A^\sop $ with the opposite algebra $A^\op\seteq\set{a^\op}{a\in A}$
with the multiplication $a^\op\;b^\op=(ba)^\op$.
We have $A^\sop\simeq(A^\op)^\rev$.

Let $A$ and $B$ be $\cor$-superalgebras. An {\em
$(A,B)$-superbimodule} is an $(A,B)$-bimodule with a
$\Z_2$-grading compatible with the left action of $A$ and the right
action of $B$. Furthermore, we assume that $ax=xa$ for $a\in\cor$ and $x\in M$.
We denote by $\Mods(A,B)$ the category of
$(A,B)$-superbimodules. We have $\Mods(A,B)\simeq\Mods(A\tens B^\sop)$.

For an $(A,B)$-superbimodule $L$, we have a functor $F_L\cl
\MOD(B)\to\MOD(A)$ given by $N\mapsto L\otimes_BN$ for $N \in\MOD(B)$.
Then $F_L$ becomes a superfunctor with an isomorphism
$$\alm{F_L}\cl F_L\Pi N=L\otimes_B\Pi N\to \Pi F_L N=\Pi(L\otimes_BN)$$
given by $$s\otimes \pi(x)\mapsto \pi(\phi_L(s)\otimes x) \ (s\in L, \ x\in N).$$

For an $(A,B)$-superbimodule $L$, the superbimodule structure on  $\Pi L$ is given as follows:
$$a\cdot\pi(s)\cdot b=\pi(\phi_A(a)\cdot s\cdot b) \quad
\text{for all $s\in L$,  $a\in A$ and $b\in B$.}$$
Then there exists a natural isomorphism between superfunctors $\eta\cl
F_{\Pi L}\overset{\sim}{\to} \Pi\cdot F_L$. The isomorphism
$\eta_N\cl (\Pi L)\otimes_BN\overset{\sim}{\to} \Pi(L\otimes_BN)$ is
given by $\pi(s)\otimes x\mapsto \pi(s\otimes x)$. It is an
isomorphism of superfunctors since one can easily check the
commutativity of the following diagram:
$$\xymatrix@C=10ex@R=2ex{
F_{\,\Pi L}\cdot\Pi\ar[r]^-{\eta\cdot\Pi}\ar[dd]_{\alm{F_{\Pi L}}}
&\Pi\cdot F_L\cdot \Pi\ar[dr]^{\Pi\cdot\alm{F_L}}\ar[dd]^{\alm{(\Pi\cdot F_L)}}\\
&&\Pi\cdot\Pi\cdot F_L\ar[dl]^(.3){\hs{3ex}\alm{\Pi}\cdot F_L=-{\rm id}_{\Pi\cdot\Pi\cdot F_L}}\\
\Pi\cdot F_{\,\Pi L}\ar[r]^-{\Pi\cdot\eta}&\Pi\cdot\Pi\cdot F_L.
}
$$
by using the fact $\phi_{\Pi L}(\pi(s))=-\pi(\phi_L(s))$.
Summing up, we obtain
\Prop $L\mapsto F_L$ gives a superfunctor
$$\Mods(A,B)\to \Fcts(\Mods(B),\Mods(A))\simeq \Fcts(\Mod(B),\Mod(A))$$
and superbifunctors \eqn&& \Mods(A,B)\times \Mods(B)\to\Mods(A), \\
&&\Mods(A,B)\times \Mod(B)\to\Mod(A).\eneqn
\enprop

Let $A,B,C$ be $\cor$-superalgebras. For $K\in\Mods(A,B)$ and
$L\in\Mods(B,C)$, the tensor product $K\ot_B L$ has a structure of
$(A,C)$-superbimodule. We define the homomorphisms
$$\al(K,L)\cl (\Pi K)\ot_BL\isoto\Pi (K\ot_B L) \quad \text{by}
\ \ \pi(x)\tens y\mapsto \pi(x\ot y)$$
and
$$\beta(K,L)\cl K\ot_B(\Pi L)\isoto\Pi (K\ot_B L) \quad \text{by}
\ \ x\otimes \pi(y)\mapsto \pi\bl\phi_K(x)\tens y\br.$$ These
homomorphisms are well-defined and we can easily check the following
lemma. \Lemma $\scbul\ot_B\scbul:\Mods(A,B)\times
\Mods(B,C)\to\Mods(A,C)$ is a superbifunctor of supercategories.
\enlemma

We now discuss the endomorphisms of bimodules. Let $A$, $B$, $C$ be
$\cor$-superalgebras and let $L$ be an $(A\tens C,B)$-superbimodule.
Regarding $L$ as an $(A,B)$-bimodule, we obtain a superfunctor
$F_L\cl \Mod(B)\to\Mod(A)$. Thus we get  a superalgebra homomorphism
$$C\to \End_{\sd{\Fcts(\Mod(B),\Mod(A))}}(F_L)\simeq\End_{\sd{\Mod(A,B)}}(L),$$
which is given by assigning to $c\in C_\eps$ ($\eps=0,1$) the
morphism in $\Mods(A,B)$
$$L\ni x\longmapsto \pi^\eps(cx)\in\Pi^\eps L.$$

Similarly, let $K$ be an $(A,B\tens C)$-superbimodule and consider
$K$ as an $(A,B)$-bimodule to  obtain a superfunctor $F_K\cl
\Mod(B)\to\Mod(A)$. Then we get a superalgebra homomorphism
$$C\to \End_{\sd{\Fcts(\Mod(B),\Mod(A))}}(F_K)^\sop\simeq\End_{\sd{\Mod(A,B)}}(K)^\sop$$
by assigning $\psi^\sop$ to $c\in C_\eps$ ($\eps=0,1$), where
$\psi\in\Hom_{\Mod(A,B)}(K,\Pi^\eps K)$ is the morphism
$$K\ni x\longmapsto \pi^\eps(\phi_K{}^\eps(x)c)\in\Pi^\eps K.$$

\subsection{Grothendieck group} Assume that the supercategory
$(\shc,\Pi,\xi)$ is an exact category such that $\Pi$ sends the
exact sequences to exact sequences. Recall that the Grothendieck
group $[\shc]$ of $\shc$ is the abelian group generated by $[X]$ \ro
$X$ is an object of $\shc$\rf\ with the defining relations:
\begin{center}
if $0 \to X' \to X \to X''\to 0$ is an exact sequence, then $[X] = [X'] + [X'']$.
\end{center}
We denote by $\pi$ the
involution of $[\shc]$ given by $[X] \mapsto [\Pi X]$.
Then $[\shc]$ is a module over $\Z^\pi=\Z\oplus\Z\pi$.

\vskip 5mm

\section{Supercategorification via quiver Hecke superalgebras} \label{sec: Quiver Hecke}

\subsection{Quiver Hecke superalgebras} \label{subsec:quiver}
In this subsection, we recall the definition of quiver Hecke
superalgebras and their basic properties (\cite{KKT11}). We take a
graded commutative ring $\k=\bigoplus_{n \in \Z_{\ge 0}}\k_n$ as a
base ring. {\em For the sake of simplicity, we assume that $\k_0$ is
a field of characteristic different from $2$.}

Let $(\car,\P,\Pi,\Pi^{\vee})$ be a Cartan superdatum.
For $i \neq j \in I$ and $r,s \in \Z_{\ge0}$, let
$t_{i,j;(r,s)}$ be an element of $\k$ satisfying the following conditions:
\begin{align*}
& t_{i,j;(r,s)} \in \k_{-2(\alpha_i|\alpha_j)-r(\alpha_i|\alpha_i)-s(\alpha_j|\alpha_j)},
\qquad t_{i,j;(r,s)}=t_{j,i;(s,r)}, \\
& t_{i,j;(-a_{ij},0)} \in \k_0^{\times},
\qquad  t_{i,j;(r,s)}=0 \text{ if $i \in \Iod$ and $r$ is odd.}
\end{align*}
We take $t_{i,j;(r,s)}=0$ for  $i=j$.

For any $\nu \in I^n$ ($n \ge 2$), let
$$\cP_\nu\seteq \k \langle x_1,\ldots,x_n \rangle /
    \langle x_ax_b - (-1)^{\pa(\nu_a)\pa(\nu_b)}x_bx_a \rangle_{1 \le a < b \le n}$$
be the superalgebra generated by $x_k$ ($1\le k\le n$) where the
parity of the indeterminate $x_k$ is $\pa(\nu_k)$. For $i,j \in I$,
we choose an element $\cQ_{i,j}$ in $\cP_{(ij)}$ of the form
\begin{equation*}
 \cQ_{i,j}(x_1,x_2)= \sum_{r,s\in\Z_{\ge0}} t_{i,j;(r,s)} x_1^r x_2^s.
\end{equation*}

\begin{definition}[\cite{KKT11}] \label{def:Quiver Hekce superalg} \
The {\em quiver Hecke superalgebra} $R(n)$ of degree $n$
associated with a Cartan superdatum $(\car,\P,\Pi,\Pi^{\vee})$ and
$(\cQ_{i,j})_{i,j\in I}$
 is the superalgebra over $\k$ generated by $e(\nu)$ $(\nu \in I^n)$, $x_k$
 $(1 \le k \le n)$, $\tau_{a}$ $(1 \le a \le n-1)$ with the parity
 \begin{equation*}
 \pa(e(\nu))=0, \quad \pa(x_k e(\nu))= \pa(\nu_k), \quad \pa(\tau_a e(\nu))=\pa(\nu_a)\pa(\nu_{a+1})
 \end{equation*}
 subject to the following defining relations:
\begin{equation}\label{eq:R(n)}
\begin{aligned}
& e(\mu)e(\nu)=\delta_{\mu,\nu}e(\nu) \ \ \text{for} \ \mu, \ \nu
\in I^n,  \ \ 1 = \sum_{\nu \in I^n} e(\nu),\\
& x_p x_q e(\nu)= (-1)^{\pa(\nu_p)\pa(\nu_q)}x_qx_p e(\nu) \quad
\text{if} \ p \neq q, \\
& x_pe(\nu)=e(\nu)x_p, \ \ \tau_ae(\nu) = e(s_a \,\nu) \tau_a, \
\text{where} \  s_a=(a,a+1), \\
& \tau_ax_pe(\nu)=
(-1)^{\pa(\nu_p)\pa(\nu_a)\pa(\nu_{a+1})}x_p\tau_ae(\nu) \ \
\text{if} \ p \neq a, \ a+1, \\
&(\tau_ax_{a+1}\hs{-0.2ex}-(-1)^{\pa(\nu_a)\pa(\nu_{a+1})}x_{a}\tau_a)e(\nu)
\\ & \qquad \qquad =(x_{a+1}\tau_a \, - (-1)^{\pa(\nu_a)\pa(\nu_{a+1})}\tau_ax_{a})e(\nu)
= \delta_{\nu_a,\nu_{a+1}}e(\nu), \\
& \tau_a^2e(\nu)=\cQ_{\nu_a,\nu_{a+1}}(x_a,x_{a+1})e(\nu), \\
&
\tau_a\tau_be(\nu)=(-1)^{\pa(\nu_a)\pa(\nu_{a+1})\pa(\nu_b)\pa(\nu_{b+1})}
\tau_b\tau_ae(\nu) \ \ \text{if} \ |a-b|>1,\\
& (\tau_{a+1}\tau_{a}\tau_{a+1}-\tau_{a}\tau_{a+1}\tau_{a})e(\nu) \\
& =  \begin{cases}
    \dfrac{\cQ_{\nu_{a},\nu_{a+1}}(x_{a+2},x_{a+1})-
    \cQ_{\nu_{a},\nu_{a+1}}(x_{a},x_{a+1}) }{x_{a+2}-x_{a}}e(\nu)
        \quad  \text{if} \ \ \nu_{a}=\nu_{a+2} \in \Iev, \\
    (-1)^{\pa(\nu_{a+1})}(x_{a+2}-x_a)
\dfrac{\cQ_{\nu_{a},\nu_{a+1}}(x_{a+2},x_{a+1})-\cQ_{\nu_{a},\nu_{a+1}}(x_{a},x_{a+1})}{x^2_{a+2}-x^2_{a}}e(\nu) \\
    \qquad \qquad \qquad \qquad \qquad \qquad \qquad \qquad \qquad \qquad   \text{if} \ \ \nu_{a}=\nu_{a+2} \in \Iod, \\
     \qquad \qquad \qquad \qquad \qquad 0 \quad \qquad \qquad \qquad\qquad \text{ otherwise }.
    \end{cases}
\end{aligned}
\end{equation}
\end{definition}

\noindent The algebra $R(n)$ is also $\Z$-graded by setting
\begin{equation*}
\degZ(e(\nu))=0, \quad \degZ(x_k e(\nu))= (\alpha_{\nu_k}|\alpha_{\nu_k}),
\quad \degZ(\tau_a e(\nu))=-(\alpha_{\nu_a}|\alpha_{\nu_{a+1}}).
\end{equation*}

For $\beta \in \rtl^+$ with $|\beta|=n$, set
\begin{align*}
& I^\beta=\set{ \nu=(\nu_1,\ldots,\nu_n) \in I^n}{\alpha_{\nu_1}+ \cdots + \alpha_{\nu_n}=\beta }.
\end{align*}
For $\alpha,\beta \in \rtl^+$ and $m,n \in \Z_{\ge 0}$, we define
\begin{align*}
& R(m,n)=R(m)\otimes_{\k}R(n)\subset R(m+n), \\
& e(n)=\sum_{\nu \in I^n}e(\nu), \quad e(\beta)=\sum_{\nu \in I^\beta}e(\nu), \quad
       e(\alpha,\beta)=\sum_{\mu \in I^\alpha, \ \nu \in I^\beta}e(\mu,\nu), \\
& R(\beta)=e(\beta)R(n),  \quad R(\alpha,\beta)=R(\alpha)\otimes_{\k}R(\beta)\subset R(\alpha+\beta),  \\
& e(n,i^k)=\sum_{ \substack{\nu \in I^{n+k}, \\ \nu_{n+1}=\cdots=\nu_{n+k}=i} }e(\nu), \quad
e(\beta,i^k)=e(\beta,k\alpha_i).
\end{align*}
Here, $R(m)\otimes_{\k}R(n)$ is endowed with a superalgebra
structure by \eqref{eq:tensor superal} and the map
$R(m)\otimes_{\k}R(n) \to R(m+n)$ is a superalgebra homomorphism.

For an $R(m)$-supermodule $M$ and an $R(n)$-supermodule $N$, we define their {\em convolution product}
$M \circ N$ by
\begin{equation*}
M \circ N \seteq R(m+n) \otimes_{R(m,n)}(M \tens N).
\end{equation*}
\begin{proposition} [{\cite[Corollary 3.15]{KKT11}}] \label{Prop: PBW}
For each $w \in S_n$, choose a reduced expression $s_{i_1}\cdots
s_{i_\ell}$ of $w$ and write $\tau_w=\tau_{i_1}\cdots\tau_{i_\ell}$.
Then
$$\{ x_1^{a_1}\cdots x_n^{a_n} \tau_w e(\nu) \ | \ a=(a_1,\ldots,a_n) \in \Z_{\ge 0}^n, \ w \in S_n, \ \nu \in I^n \} $$
forms a basis of the free $\k$-module $R(n)$.
\end{proposition}

Let $\MOD(R(\beta))$ be the category of  arbitrary $\Z$-graded
$R(\beta)$-supermodules. Let $\PROJ(R(\beta))$ and $\REP(R(\beta))$
be the full subcategories of $\MOD(R(\beta))$ consisting of finitely
generated projective $R(\beta)$-supermodules and
$R(\beta)$-supermodules finite-dimensional over $\k_0$,
respectively. The morphisms in these categories are
$R(\beta)$-linear homomorphisms preserving the $\Z \times
\Z_2$-grading. As we have seen in Example \ref{exa:supercategories}
(b), these categories have a supercategory structure induced by the
parity shift.

\vskip 1em

{\em In the sequel, by an $R(n)$-module or $R(\beta)$-module,
we mean a $\Z$-graded $R(n)$-supermodule
or $R(\beta)$-supermodule. }

\vskip 1em For an $R(\beta)$-module $M = \bigoplus_{t \in \Z} M_t$,
let $M \langle k \rangle$ denote the $\Z$-graded $R(\beta)$-module
such that $M \langle k \rangle_t \seteq M_{k+t}$; i.e., $M \langle k
\rangle = \bigoplus_{t \in \Z} M_{k+t}$. We also denote by $q$ the
grading shift functor
$$(qM)_i=M_{i-1}.$$
The Grothendieck groups $[\PROJ(R(\beta))]$ and $[\REP (R(\beta))]$
have the $\A^\pi$-module structure given by $q[M] = [qM]$ and
$\pi[M]=[\Pi M]$, where $[M]$ denotes the isomorphism class of an
$R(\beta)$-module $M$.

Let $a=\sum\limits_{k\in\Z,\;\eps=0.1}m_{k,\eps}
q^k\pi^\eps\in\A^\pi$ with $m_{k,\eps}\in\Z_{\ge0}$. For an
$R(\beta)$-module $M$, we define \eq aM=\soplus_{k\in\Z,\;\eps=0.1}
\bl q^k\Pi^\eps M\br^{\oplus m_{k,\eps}}, \eneq so that we have
$[aM]=a[M]$.

\vskip 3mm

\subsection{Strong perfect basis of $\REP(R)$} \label{sec:
categorical strong perfect basis}

In this subsection, we study the structure of the supercategory
$\REP(R(\beta))$ based on the results of \cite{EKL} and
\cite[Section 6]{KKO12}. In those papers, the authors studied the
supercategory $\Rep(R(\beta))$, not $\REP(R(\beta))$, but their
results provide us with a good foundation.  In \cite{HW12},
Hill and Wang dealt with the supercategory $\REP(R(\beta))$ under a
certain restriction, called the (C6) condition (see \S\;{sec:QKM}).
Although $[{\rm
Rep}(R(\beta))] \simeq [\REP(R(\beta))]/(\pi-1)$ as we saw in
\cite{KKO12}, the action of $\pi$ on $\REP(R(\beta))$ is non-trivial
and will be investigated here.

Throughout this subsection, we assume that
\eq&&\parbox{70ex}%
{the ring $\k_0$ is a field of characteristic different from $2$
 and the $\k_i$'s are finite-dimensional over $\k_0$} \label{cond:k0} \eneq
Under the assumption \eqref{cond:k0}, the superalgebra $R(\beta)$
has the following properties: \eq&&\parbox{75ex}{ \bnum
\item Any simple object in $\MOD(R(\beta))$ is finite-dimensional over $\k_0$ and has an indecomposable
finitely generated projective cover (unique up to isomorphism),
\item there are finitely many simple objects in $\REP(R(\beta))$
up to $\Z$-grading shifts and isomorphisms. \ee }\label{property of
R(beta)-mod} \eneq \noindent Thus $\REP(R(\beta))$ contains all
simple $R(\beta)$-supermodules and the set of isomorphism classes of
simple $R(\beta)$-supermodules, denoted by $\Irr(R(\beta))$, forms a
{\em $\Z$-basis} of $[\REP(R(\beta))]$.

For $1 \le k < n$, let $\mathbf{b}_k\seteq \tau_k x_{k+1} \in R(n\al_i)$. It is known (\cite{EKL,HW12,KL1}) that
\eqn&&\parbox{70ex}{
\bna
\item The $\mathbf{b}_k$'s are idempotents and they satisfy the braid relations,
\item $\mathbf{b}_{w}$ is well-defined for any $w \in S_n$ by (a),
\item $\mathbf{b}(i^n)\seteq\mathbf{b}_{w_0}$ is a primitive idempotent of $R(n\al_i)$, where
$w_0$ is the longest element of $S_n$.
\ee}
\eneqn

\begin{proposition} [\cite{EKL}]\label{prop:proj}
The superalgebra $R(n\alpha_i)$ is decomposed into a direct sum of
projective indecomposable $\Z \times \Z_2$-modules\,{\rm :}
\eq\label{eq:divided} &&R(n\alpha_i) \simeq [n]_i^\pi !  P(i^{n}),
\eneq where
$$  P(i^{n})\seteq (\pi_iq_i)^{-n(n-1)/2} R(n\alpha_i) \mathbf{b}(i^n).$$
\end{proposition}
\noindent The factorial $[n]_i^\pi!$ is defined in
\eqref{def:pifact}.

Note that $P(i^{n})$ is a unique indecomposable projective
$R(n\alpha_i)$-supermodule up to isomorphism and ($\Z \times
\Z_2$)-grading shift. By \eqref{property of R(beta)-mod}, there
exists an irreducible $R(n\alpha_i)$-supermodule, denoted by
$L(i^n)$, which is unique up to isomorphism and ($\Z \times
\Z_2$)-grading shift:
\begin{equation} \label{Eqn: Def of L(i^n)}
 L(i^n)\seteq \ind^{R(n\alpha_i)}_{\k[x_1] \ot \cdots \ot \k[x_n]} {\mathbf 1},
\end{equation}
where ${\mathbf 1}$ is the simple $\k[x_1] \ot \cdots \ot \k[x_n]$-supermodule
which is isomorphic to $\k_0$.

\medskip
For $M \in \REP(R(\beta))$ and $i\in I$, define
\begin{equation} \label{eqn: crystal operators}
\begin{aligned}
&\Delta_{i^k} M = e( \beta- k\alpha_i,i^k) M \in \REP(R(\beta-k\alpha_i,k\alpha_i)), \\
& \varepsilon_i(M) = \max\{ k \ge 0 \mid \Delta_{i^k} M \ne 0 \}, \\
& E_i(M) = e(\beta-\alpha_i,i)M \in \REP(R(\beta-\alpha_i)), \\
& \tilde{e}_i(M) = \soc(E_i(M)) \in \REP(R(\beta-\alpha_i)), \\
& F_i'(M) = \ind_{\beta,\alpha_i}(M \bt L(i)) \in \REP(R(\beta+\alpha_i)),\\
& \tilde{f}_i(M) = \hd(F_i' M) \in \REP(R(\beta+\alpha_i)).
\end{aligned}
\end{equation}
Here, $\soc(M)$ means the {\em socle} of $M$ and $\hd(M)$ means the
{\em head} of $M$. We set $\eps_i(M)=-\infty$ for $M=0$. Then $E_i$
and $F'_i$ are superfunctors.

For $M = \bigoplus_{a \in \Z} (M_{a,\bar{0}} \oplus M_{a,\bar{1}}) \in \REP(R(\beta))$,
we define its $(q,\pi)${\em -dimension} and $(q,\pi)${\em -character}
 as follows:
\eq
&&\ba{l}
 \dim^\pi_q(M) \seteq \sum_{a \in \Z}
(\dim_{\k_0}M_{a,\bar{0}}+ \pi\dim_{\k_0}M_{a,\bar{1}})q^a \in\Z[q^{\pm1}]^\pi,
\\[1ex]
 \ch_q^\pi(M) \seteq \sum_{\nu \in I^\beta} \dim^\pi_q(e(\nu)M) \cdot e(\nu).
\ea\eneq

\begin{lemma} [\cite{Kle05,KL1,LV09}] \label{Lem: crystal structure}

For any $[M] \in \Irr(R(\beta))$ and $i \in I$, we have \bna
\item
$[\tilde{e}_iM] \in \Irr(R(\beta-\alpha_i))$ if $\varepsilon_i(M)>0$, and
$[\tilde{f}_iM] \in \Irr(R(\beta+\alpha_i))$.

\item $\tilde{f}_i\tilde{e}_iM \simeq M$ if $\varepsilon_i(M)>0$,
and $\tilde{e}_i\tilde{f}_iM \simeq M$.
\item $\k_nM=0$ for $n>0$ and $\k_0 \simeq {\rm End}_{R(\beta)}(M)$.
\ee
\end{lemma}

\begin{proposition} [{\cite[Proposition 6.2]{KKO12}}]
For any $[M] \in \Irr(R(\beta))$ with
$\varepsilon \seteq \varepsilon_i(M) >0$, we have
\begin{equation} \label{eqn: Perfect basis of REP}
[E_iM]= \pi_i^{1-\eps}q_i^{1-\eps}[\eps]_i^\pi[\tilde{e}_iM] +
\sum_k [N_k],
\end{equation}
where $[N_k] \in \Irr(R(\beta-\alpha_i))$ with $\eps_i(N_k) < \eps-1$.
\end{proposition}

As can be seen in the following theorem, the endofunctor $\Pi$ on
$\REP(R(\beta))$ treated in this paper is substantially different
from the one in \cite{KKO12} (cf.\ \cite[Theorem 6.4]{KKO12}).

\begin{theorem}\label{thm: Pi M }
For any $[M] \in \Irr(R(\beta))$, we have
$$ M \not\simeq \Pi M.$$
\end{theorem}

\begin{proof} It was shown in \cite[Theorem 6.4]{KKO12} that
\eqn&&\text{$S\simeq\Pi S$ for any simple $S\in \Mod(R(\beta))$.}
\eneqn Since $\REP(R(\beta))^\rev$ is equivalent to the Clifford
twist of $\on{Rep}(R(\beta))$, the assertion follows from
\cite[Lemma 2.11]{KKT11}.
\end{proof}

Let $\psi\cl R(\beta) \to R(\beta)$ be the anti-involution given by
\begin{align}
\psi(ab)=\psi(b)\psi(a), \quad \psi(e(\nu))=e(\nu), \quad
\psi(x_k)=x_k, \quad  \psi(\tau_l)=\tau_l \label{eq:anti}
\end{align}
for all $a,b \in R(\beta)$. For any $M \in \MOD(R(\beta))$, we
denote by $M^* \seteq \Hom_{\k_0}(M,\k_0)$ the $\k_0$-dual of $M$
whose left $R(\beta)$-module structure is given by $\psi$. By a
direct computation, we have
$$ (q M)^* = \Hom_{\k_0}(q  M, \k_0) \simeq q^{-1}  \Hom_{\k_0}(M, \k_0)= q^{-1}  (M^*).$$
Similarly, we have $(\Pi M)^* \simeq \Pi(M^*)$, which implies
$$ ([k]_i^\pi M)^* \simeq \Pi_i^{1-k} [k]_i^\pi (M^*) \ \ \text{ for } k \in \Z_{\ge 0}.$$
Here we set $\Pi_i\seteq\Pi^{\pa(i)}$.

\begin{proposition} \label{prop: quasi self-dual} \hfill
\bna \item For any $[M] \in \Irr(R(\beta))$ such that $\varepsilon \seteq \varepsilon_i(M) >0$, we have
$$(q_i^{1-\eps}\tilde{e}_iM)^* \simeq \Pi_i^{1-\eps}q_i^{1-\eps}\tilde{e}_i(M^*).$$
\item For any $[M] \in \Irr(R(\beta))$, there exists a pair of integer $(r_1,r_2)$ such that
$$(q^{r_2} M)^*\simeq\Pi^{r_1}q^{r_2} M. $$
\ee
\end{proposition}

\begin{proof}
Note that the duality functor $*$ commutes with the functor $E_i$.
Applying the functor $*$ to \eqref{eqn: Perfect basis of REP}, we have
$$[E_i(M^*)]=[\eps]_i^\pi [(q_i^{1-\eps}\tilde{e}_iM)^*] + \sum_{k,\ \eps_i(N^*_k) < \eps-1} [N^*_k].$$
On the other hand,
$$[E_i(M^*)]=\pi_i^{1-\eps}[\eps]_i^\pi[q_i^{1-\eps}\tilde{e}_i(M^*)]+\sum_{k,\ \eps_i(N'_k) < \eps-1} [N_k'].$$
Therefore the assertion (a) holds.

\smallskip

We will prove (b) by induction on $|\beta|$. If $|\beta|=0$, our
assertion is trivial. If $|\beta|>0$, take $i \in I$ such that
$\eps=\eps_i(M)>0$. By induction hypothesis, there exists
$(r_1',r_2')$  such that
$$\Pi^{r_1'}q^{r_2'}q_i^{1-\eps}\tilde{e}_i M \simeq (q^{r_2'}q_i^{1-\eps} \tilde{e}_iM)^*.$$
The assertion (a) implies
\begin{align*}
\Pi^{r_1'}q_i^{1-\eps}\tilde{e}_i(q^{r_2'}M) & \simeq (q_i^{1-\eps} \tilde{e}_i(q^{r_2'}M))^*
\simeq \Pi_i^{1-\eps} q_i^{1-\eps} \tilde{e}_i(q^{r_2'}M)^*,
\end{align*}
which yields
$$q_i^{1-\eps}\tilde{e}_i(q^{r_2'}M) \simeq  \Pi^{\pa(i)(1-\eps)-r_1'} q_i^{1-\eps} \tilde{e}_i(q^{r_2}M)^*.$$
Therefore, by Lemma \ref{Lem: crystal structure}\;(b), we conclude
$$\Pi^{\pa(i)(1-\eps)-r_1'}q^{r_2'}M \simeq (q^{r_2'}M)^*.$$
Thus the pair $(\pa(i)(1-\eps)-r_1', r_2')$ is the desired one.
\end{proof}

For $[M] \in \Irr(R(\beta))$, we say that $M$ is  {\em quasi-self-dual}  if
$$M^* \simeq \Pi^{\epsilon}M \quad\text{for $\epsilon =0$ or $1$.}$$
Note that, by Theorem \ref{thm: Pi M },
$\eps$ is uniquely determined by $M$.

\begin{example} For $i \in I$, we can easily check that
$$(q_i^{\frac{n(n-1)}{2}}L(i^n))^* \simeq \Pi^{\frac{n(n-1)}{2}} q_i^{\frac{n(n-1)}{2}}L(i^n).$$
Hence, for $n=2$ and $i\in\Iod$, we have
$(q_i L(i^2))^* \simeq \Pi_i (q_i L(i^2))$.
However, $(\Pi^\eps q^r L(i^2))^* $ is never isomorphic to
$\Pi^\eps q^rL(i^2)$ for any $r \in \Z$ and any $\eps=0,1$.
\end{example}

\medskip
 Let $\Irr_{\mathrm{qsd}}(R(\beta))$
be the subset of $\Irr(R(\beta))$ consisting of the isomorphism
classes of  quasi-self-dual modules in $\Irr(R(\beta))$. Then
$\Irr_{\mathrm{qsd}}(R(\beta))$ forms an $\A$-basis of the
Grothendieck group $[\REP(R(\beta))]$. Choose a subset
$\Irr_0(R(\beta)) \subset \Irr_{\mathrm{qsd}}(R(\beta))$ satisfying
the conditions: \eq &&\ba{l}
\Irr_0(R(\beta))\cap \pi\Irr_0(R(\beta))=\emptyset, \\
\Irr_{\mathrm{qsd}}(R(\beta))=\Irr_0(R(\beta))\sqcup \pi\Irr_0(R(\beta)).
\ea
\eneq
Such a subset $\Irr_0(R(\beta))$ exists
by Theorem~\ref{thm: Pi M }.

\begin{theorem} \label{Thm: categorical strong} For $\beta \in \rtl^+$,
$\Irr_0(R(\beta))$ is a strong perfect basis of $[\REP(R(\beta))]$
as an $\A^{\pi}$-module.
\end{theorem}
\begin{proof}
The statement is an immediate consequence of Proposition \ref{prop: quasi self-dual} and
\eqref{eqn: Perfect basis of REP}.
\end{proof}

\vskip 3mm

\subsection{Cyclotomic quotients} \label{sec: cycl quotient} In this
subsection, we quickly review the results on the cyclotomic quiver
Hecke superalgebras $R^\Lambda$ which were proved in \cite[Section
7, 8, 9]{KKO12}.

\bigskip

For each $i \in I$ and $k \in \Z_{\ge 0}$,
we take  $c_{i;k} \in \k_{k(\alpha_i|\alpha_i)}$ such that
{\rm(i)} $c_{i,0}=1$,  {\rm(ii)} $c_{i;k}=0$ if $i \in \Iod$ and $k$ is odd.
For $\Lambda \in \P^+$ and $i \in I$, we choose a monic
polynomial
\begin{equation} \label{Eq: cylotomic polynomial}
\begin{aligned}
a^{\Lambda}_i(u)= \sum_{k=0}^{\langle h_i,\Lambda \rangle}
c_{i;k}u^{\langle h_i,\Lambda \rangle-k}
\end{aligned}
\end{equation}
and define
$$ a^{\Lambda}(x_1) = \sum_{\nu \in I^n} a^{\Lambda}_{\nu_1}(x_1)e(\nu) \in R(n) .$$

\begin{definition} \label{Def: cyclo}
Let $\beta \in \rtl^+$ and $\Lambda\in \P^+$. The {\em
cyclotomic quiver Hecke superalgebra $R^{\Lambda}(\beta)$ at
$\beta$} is the quotient algebra
$$ R^{\Lambda}(\beta) =
\dfrac{R(\beta)}{R(\beta)a^{\Lambda}(x_1)R(\beta)}.$$
\end{definition}

We need the next proposition in  proving our main result: the
supercategorification of integrable highest weight modules.

\begin{proposition}[{\cite[Corollary 7.5]{KKO12}}] \label{Prop: Nilpotency}
For $\beta \in \rtl^+$, there exists $m$ such that
$$R^\Lambda(\beta+k\alpha_i)=0\quad \text{for any $k \ge m$.}$$
\end{proposition}

Let $\MOD(R^\Lambda(\beta))$, $\PROJ(R^\Lambda(\beta))$ and
$\REP(R^\Lambda(\beta))$ be the supercategories defined in a similar
manner as we did in \S \ref{subsec:quiver}. For each $i \in I$ and
$\beta \in \rtl^+$, we define the superfunctors
\begin{align*}
& E_i^{\Lambda}\cl \MOD(R^{\Lambda}(\beta+\alpha_i)) \to \MOD(R^{\Lambda}(\beta)), \\
& F_i^{\Lambda}\cl \MOD(R^{\Lambda}(\beta)) \to \MOD(R^{\Lambda}(\beta+\alpha_i))
\end{align*}
by
\begin{align*}
& E_i^{\Lambda}(N)=e(\beta,i)N =e(\beta,i)R^{\Lambda}(\beta+\alpha_i)
\otimes_{R^{\Lambda}(\beta+\alpha_i)}N,\\
& F_i^{\Lambda}(M)=R^{\Lambda}(\beta+\alpha_i)e(\beta,i)
\otimes_{R^{\Lambda}(\beta)}M
\end{align*}
for $M \in \MOD(R^{\Lambda}(\beta))$ and $N\in
\MOD(R^{\Lambda}(\beta+\alpha_i))$. Then $(\FL_i,E^{\Lambda}_i)$ is
a superadjoint pair (see \S\;\ref{subsec:super2}); i.e.,
$$\Hom_{R^\La(\beta+\al_i)}(F^{\Lambda}_iM,N)\simeq\Hom_{R^\La(\beta)}(M,E^{\Lambda}_iN).$$

Set $n=|\beta|$. There exist natural transformations: \eqn &&
x_{E^{\Lambda}_i}\cl E^{\Lambda}_i \to \Pi_i q^{-2}_iE_i^{\Lambda},
\hs{20ex} x_{\FL_i}\cl \FL_i \to \Pi_i q^{-2}_i\FL_i,  \\
&&\tau_{\EL_{ij}}\cl \EL_i\EL_j \to \Pi^{\pa(i)\pa(j)}
q^{(\alpha_i|\alpha_j)} \EL_j\FL_i, \hs{6ex} \tau_{\FL_{ij}}\cl
\FL_i\FL_j \to \Pi^{\pa(i)\pa(j)} q^{(\alpha_i|\alpha_j)}\FL_j\FL_i
\eneqn induced by \bna
\item the left multiplication by $x_{n+1}$ on the kernel $e(\beta,i)R^\La(\beta+\al_i)$ of the functor $\EL_i$,
\item the right multiplication by $x_{n+1}$ on the kernel
$R^\La(\beta+\alpha_i)e(\beta,i)$ of the functor
$\FL_i$,
\item the left multiplication by $\tau_{n+1}$ on the kernel
$e(\beta,i,j)R^\La(\beta+\al_i+\al_j)$ of the functor $\EL_i \FL_j$,
\item the right multiplication by $\tau_{n+1}$ on the kernel $R^\La(\beta+\alpha_i+\alpha_j)e(\beta,j,i)$
of the functor $\FL_i\FL_j$.
\end{enumerate}

For $\gamma$ with $|\gamma|=n$ and $\nu\in I^\gamma$,
let us denote by
$$\EL_\nu=\EL_{\nu_1}\cdots\EL_{\nu_n}\cl \MOD(R^\La(\beta+\gamma))
\to\MOD(R^\La(\beta)).$$
Then $x_{\EL_i}$'s and $\tau_{\EL_{ij}}$'s
induce a superalgebra homomorphism
$$R(\gamma)\to
\End_{\sd{\Fcts\bl\MOD(R^\La(\beta+\gamma),\,\MOD(R^\La(\beta)\br}}
\bl\soplus_{\nu\in I^\gamma}\EL_\nu\br.
$$
(Recall the discussion at the end of
\S\,\ref{subsec:superbimodules}.) Under this homomorphism,
$e(\nu)\in R(\gamma)$ is sent to the projection to the factor
$\EL_\nu$, $x_ke(\nu)$ is sent to $\EL_{\nu_1}\cdots
x_{\EL_{\nu_k}}\cdots \EL_{\nu_n}$, and $\tau_ke(\nu)$ is sent to
$\EL_{\nu_1}\cdots\tau_{\EL_{\nu_k,\nu_{k+1}}} \cdots \EL_{\nu_n}$.
Here, we have forgotten the grading.

Similarly,
let us denote by
$$\FL_\nu=\FL_{\nu_n}\cdots\FL_{\nu_1}\cl \MOD(R^\La(\beta))
\to\MOD(R^\La(\beta+\gamma)).$$
Then $x_{\FL_i}$'s and
$\tau_{\FL_{ij}}$'s induce a superalgebra homomorphism
$$R(\gamma)\to
\End_{\sd{\Fcts\bl\MOD(R^\La(\beta)),\,\MOD(R^\La(\beta+\gamma))\br}}
\bl\soplus_{\nu\in I^\gamma}\FL_\nu\br^\sop,
$$
where $e(\nu)\in R(\gamma)$ is sent to the projection to the factor
$\FL_\nu$, $x_ke(\nu)$ is sent to $\FL_{\nu_n}\cdots
x_{\FL_{\nu_k}}\cdots\FL_{\nu_1}$, and $e(\nu)\tau_k$ is sent to
$\FL_{\nu_n}\cdots\tau_{\FL_{\nu_{k+1},\nu_{k}}} \cdots\FL_{\nu_1}$.

By the superadjunction, $\tau_{\EL_{ij}}$ induces  a natural
transformation
$$\FL_j\EL_i \to \Pi^{\pa(i)\pa(j)}q^{(\alpha_i|\alpha_j)} \EL_i\FL_j.$$
Set
\begin{align*}
\PROJ(R^{\Lambda})=\bigoplus_{\beta \in \rtl^+}\PROJ(R^{\Lambda}(\beta)),
\quad\REP(R^{\Lambda})=\bigoplus_{\beta \in \rtl^+}\REP(R^{\Lambda}(\beta)).
\end{align*}

\begin{theorem}[{\cite[Theorem 8.9]{KKO12}}] \label{Thm: injective}
The functors $E^{\Lambda}_i$ and $F^{\Lambda}_i$ are well-defined
exact superfunctors on $\PROJ(R^{\Lambda})$  and $\REP(R^{\Lambda})$.
Hence they induce the endomorphisms $\mathsf{E}_i$ and
$\mathsf{F}_i$ on the Grothendieck groups $[\PROJ(R^{\Lambda})]$ and
$[\REP(R^{\Lambda})]$\,{\rm :}
\eqn &&
\xymatrix@C=13ex{[\PROJ(R^\Lambda(\beta))]\ar@<.8ex>[r]^-{\mathsf{F}_i\seteq[F_{i}^{\Lambda}]}
&[\PROJ(R^\Lambda(\beta+\alpha_i))]
\ar@<.8ex>[l]^-{\mathsf{E}_i\seteq[E_{i}^{\Lambda}]}
},\\
&&\xymatrix@C=13ex{[\REP(R^\Lambda(\beta))]
\ar@<.8ex>[r]^-{\mathsf{F}_i\seteq[F_{i}^{\Lambda}]}
&[\REP(R^\Lambda(\beta+\alpha_i))]\ar@<.8ex>[l]^-{\mathsf{E}_i\seteq[E_{i}^{\Lambda}]}.
} \eneqn
\end{theorem}

\begin{theorem}[{\cite[Theorem 9.1, Theorem 9.6]{KKO12}}] \label{Thm: Main}
There exist natural isomorphisms of endofunctors on
$\MOD(R^{\Lambda}(\beta))$ given below\,{\rm :}
\begin{equation}  \label{Eq: The com rel 0}
\begin{aligned}
& E^{\Lambda}_iF^{\Lambda}_j \overset{\sim}{\to}
 q^{-(\alpha_i|\alpha_j)}\Pi^{\pa(i)\pa(j)}F^{\Lambda}_j E^{\Lambda}_i \quad
  \text{ if $i \neq j$}, \\
& \Pi_iq_i^{-2}F^{\Lambda}_iE^{\Lambda}_i \oplus
\bigoplus^{\langle h_i,\Lambda-\beta \rangle-1}_{k=0}\Pi_i^k q_i^{2k} \overset{\sim}{\to}
E^{\Lambda}_iF^{\Lambda}_i \quad
\text{ if $\langle h_i,\Lambda-\beta \rangle \ge 0$}, \\
& \Pi_iq_i^{-2}F^{\Lambda}_iE^{\Lambda}_i \overset{\sim}{\to}
E^{\Lambda}_iF^{\Lambda}_i \oplus \bigoplus^{-\langle h_i,\Lambda-\beta \rangle-1}_{k=0}\Pi_i^{k+1}q_i^{-2k-2} \quad \text{ if $\langle h_i,\Lambda-\beta \rangle < 0$}.
\end{aligned}
\end{equation}
\end{theorem}

\vskip 3mm

\subsection{Supercategorification} \label{sec: supercategorification}

As our main results, we show that $\REP(R^{\Lambda})$ and $\REP(R)$
provide a supercategorification of $\Vs_{\A^\pi}(\Lambda)^\vee$ and
$\Us^-_{\A^\pi}(\g)^\vee$, respectively. In this subsection, we
assume that the condition \eqref{cond:k0} is satisfied; i.e., $\k_0$
is a field and the $\k_i$'s are finite-dimensional over $\k_0$.

By \eqref{property of R(beta)-mod} and Lemma \ref{Lem: crystal structure}(c), we have a perfect pairing
\eq&& [\PROJ(R^\Lambda)] \times [\REP(R^\Lambda)] \to \A^\pi
\label{def:coupl}
\eneq
given by
$$ ([P],[M]) \mapsto  \dim_q^\pi(P^\psi \ot_{R^{\Lambda}} M),$$
which implies that $ [\PROJ(R^\Lambda)]$ and $[\REP(R^\Lambda)]$ are
$\A^\pi$-dual to each other. Here, $P^\psi$ is the right
$R^{\Lambda}$-module obtained from $P$ by applying the
anti-involution $\psi$ (see \eqref{eq:anti}).

Let $\mathsf{E}_i$ and $\mathsf{F}_i$ be the endomorphisms on
$[\PROJ(R^\Lambda)]$ or $[\REP(R^\Lambda)]$ given in
Theorem~\ref{Thm: injective}. Then we can check easily that they are
adjoint to each other. For example, we have
\eqn(\FL_iP)^\psi\tens_{R^\La(\beta+\al_i)}M&=& \bl
R^\La(\beta+\al_i)e(\beta,i)\tens_{R^\La(\beta)}P)\br^\psi
\tens_{R^\La(\beta+\al_i)}M\\
&\simeq& P^\psi\tens_{R^\La(\beta)} e(\beta,i) R^\La(\beta+\al_i)
\tens_{R^\La(\beta+\al_i)}M\\ &\simeq&
P^\psi\tens_{R^\La(\beta)}\EL_iM \eneqn for $P\in
\REP(R^\La(\beta))$ and $M\in \REP(R^\La(\beta+\al_i))$.

Let us show that  $\mathsf{E}_i$ and $\mathsf{F}_i$ induce
$\Us_{\A^\pi}(\g)$-module structures on $[\PROJ(R^\Lambda)]$ and
$[\REP(R^\Lambda)]$. The natural isomorphisms given in $\eqref{Eq:
The com rel 0}$ can be written as follows:
\begin{equation}  \label{Eq: The com rel}
\begin{aligned}
& \mathsf{E}_i\mathsf{F}_j = q^{-(\alpha_i|\alpha_j)}\pi^{\pa(i)\pa(j)}\mathsf{F}_j\mathsf{E}_i \quad
\text{ if } i \neq j, \\
& \mathsf{E}_i\mathsf{F}_i = q_i^{-2} \pi_i \mathsf{F}_i\mathsf{E}_i
+ \dfrac{1-(q_i^{2}\pi_i)^{\langle h_i,\Lambda-\beta \rangle}}{1-q_i^{2}\pi_i} \quad
\text{ if } \langle h_i,\Lambda-\beta \rangle \ge 0, \\
& \mathsf{E}_i\mathsf{F}_i + \dfrac{1-(q_i^{2}\pi_i)^{\langle
h_i,\Lambda-\beta \rangle}}{q_i^{2}\pi_i-1} = q_i^{-2} \pi_i
\mathsf{F}_i\mathsf{E}_i  \quad \text{ if } \langle
h_i,\Lambda-\beta \rangle < 0
\end{aligned}
\end{equation}
on $[\PROJ(R^\Lambda)]$ or $[\REP(R^\Lambda)]$.

Let $\tilde{\mathsf{K}}_i$ be an endomorphism on $[\PROJ(R^\Lambda)]$ and $[\REP(R^\Lambda)]$ defined by
$$\tilde{\mathsf{K}}_i|_{[\PROJ(R^\Lambda(\beta))]} \seteq (q^2_i\pi_i)^{\langle h_i,\Lambda-\beta \rangle},
\quad  \tilde{\mathsf{K}}_i|_{[\REP(R^\Lambda(\beta))]} \seteq
(q^2_i\pi_i)^{\langle h_i,\Lambda-\beta \rangle}.$$ Then \eqref{Eq:
The com rel} can be rephrased as
\begin{equation}  \label{Eq: summarized com rel}
\mathsf{E}_i\mathsf{F}_j - q^{-(\alpha_i|\alpha_j)}\pi^{\pa(i)\pa(j)} \mathsf{F}_i\mathsf{E}_j
= \delta_{i,j} \dfrac{1-\tilde{\mathsf{K}}_i}{1-q_i^{2}\pi_i},
\end{equation}
which coincides with one of the defining relations in Definition
\ref{dfn:Uqsg}.

We now define the  superfunctors
\begin{align*}
& {F^\Lambda_i}^{\{n \}}: \MOD(R^\Lambda(\beta)) \to \MOD(R^\Lambda(\beta+n\alpha_i)), \\
& {E^\Lambda_i}^{\{ n \}}: \MOD(R^\Lambda(\beta+n\alpha_i)) \to \MOD(R^\Lambda(\beta)),
\end{align*}
by
\begin{align*}
& {F^\Lambda_i}^{\{ n \}}(M) = R^\Lambda(\beta+n\alpha_i) e(\beta,i^n)
\ot_{R^\Lambda(\beta)\ot R(n\alpha_i)}\bl M \tens P(i^{n})\br,\\
& {E^\Lambda_i}^{\{ n \}}(N) = \bl R^\Lambda(\beta)\ot P(i^{n})^\psi\br
\ot_{R^\Lambda(\beta)\ot R(n\alpha_i)}e(\beta,i^n)N
\end{align*}
for $M \in \MOD(R^\Lambda(\beta))$ and $N \in
\MOD(R^\Lambda(\beta+n\alpha_i))$. Then Proposition~\ref{prop:proj} implies that
$$[n]^\pi_i!{E^\Lambda_i}^{\{ n \}} \simeq (\EL_i)^n \quad \text{and} \quad
[n]^\pi_i!{F^\Lambda_i}^{\{ n \}} \simeq (\FL_i)^n.$$ Note that \eq
&&
\parbox{70ex}{
\be[(i)]
\item the actions of $\mathsf{E}_i$ on $[\PROJ(R^\Lambda)]$ and $[\REP(R^\Lambda)]$ are locally nilpotent,
\item by Proposition \ref{Prop: Nilpotency}, the actions of $\mathsf{F}_i$ on $[\PROJ(R^\Lambda)]$ and $[\REP(R^\Lambda)]$ are locally nilpotent,
\item if $\beta \neq 0$ and $M \in \REP(R^\Lambda(\beta))$ does not vanish, then there exists $i \in I$
such that $\mathsf{E}_i[M] \neq 0$, \label{eq: unique hwv}
\item $\mathsf{E}_i$ and $\mathsf{F}_i$ are the transpose of each other
with respect to the coupling \eqref{def:coupl}. Indeed we have
$P^\psi \ot_{R^{\Lambda}}\mathsf{F}_i M \simeq (\mathsf{E}_iP)^\psi
\ot_{R^{\Lambda}}M$ and $P^\psi \ot_{R^{\Lambda}}\mathsf{E}_i M
\simeq (\mathsf{F}_iP)^\psi \ot_{R^{\Lambda}}M$. \ee
}\label{prop:gro} \eneq By Proposition~\ref{prop: KMPY96},
\eqref{Eq: summarized com rel} and \eqref{prop:gro}, the
endomorphisms $\mathsf{E}_i$ and $\mathsf{F}_i$ satisfy the Serre
relations in Definition \ref{dfn:Uqsg}, which gives a
$\Us_{\A^\pi}(\g)$-module structure on $[\PROJ(R^\Lambda)]$ and
$[\REP(R^\Lambda)]$.

Let $\Irr(R^\Lambda(\beta))$ be the set of isomorphism classes of
simple $R^\Lambda(\beta)$-supermodules. Using the fully faithful
functor $\REP(R^\Lambda(\beta))\monoto\REP(R(\beta))$, we define a
subset $\Irr_0(R^\Lambda(\beta))$ of  $\Irr(R^\Lambda(\beta))$ by
$$\Irr_0(R^\Lambda(\beta)) = \Irr_0(R(\beta))\cap
[\REP(R^\Lambda(\beta))].$$ Set $\Irr_0(R^\Lambda) \seteq
\bigsqcup_{\beta \in \rtl^+} \Irr_0(R^\Lambda(\beta))$. Then Theorem
\ref{Thm: categorical strong} implies that $\Irr_0(R^\Lambda)$ is a
strong perfect basis of $[\REP(R^\Lambda)]$. Therefore, by Theorem
\ref{Thm:recognition thm} and \eqref{prop:gro}(iii), we obtain the
following supercategorification theorem.

\begin{theorem}\label{th:main1}
Let  $\Lambda \in \P^+$.
\bnum
\item $[\REP(R^\Lambda)]$ and $[\PROJ(R^\Lambda)]$ are
$\Us_{\A^\pi}(\g)$-modules.
\item
$\Irr_0(R^\Lambda)$ is a strong perfect basis of
$[\REP(R^\Lambda)]$.
\item There exist isomorphisms of $\Us_{\A^\pi}(\g)$-modules
\begin{align*}
\Vs_{\A^\pi}(\Lambda)^\vee \simeq [\REP(R^\Lambda)] \quad \text{ and
} \quad \Vs_{\A^\pi}(\Lambda) \simeq [\PROJ(R^\Lambda)].
\end{align*}
In particular, $\Vs_{\A^\pi}(\Lambda)$ and
$\Vs_{\A^\pi}(\Lambda)^\vee$ are free $\A^\pi$-modules. \ee
\end{theorem}

Set
$$[\PROJ(R)]\seteq \soplus_{\beta \in \rtl^+}[\PROJ(R(\beta))], \qquad
[\REP(R)]\seteq \soplus_{\beta \in \rtl^+}[\REP(R(\beta))].$$ We
denote by ${\rm B}^{\mathrm{low}}_{\A^\pi}(\g)$ (resp.\ ${\rm
B}^{\mathrm{up}}_{\A^\pi}(\g)$) the $\A^\pi$-subalgebra of $\Bg$
generated by $e_i'$ and $f_i^{\{n \}}$ (resp.\ ${e_i'}^{\{n \}}$ and
$f_i$) for all $i \in I$ and $n \in \Z_{> 0}$. Then, by a similar
argument given in \cite[Corollary 10.3]{KKO12}, we have:

\begin{corollary}\label{cor:main2} \hfill
\bnum
\item $[\REP(R)]$ and $[\PROJ(R)]$
have a structure of ${\rm B}^{\mathrm{up}}_{\A^\pi}(\g)$-module and
${\rm B}^{\mathrm{low}}_{\A^\pi}(\g)$-module, respectively.

\item There exist isomorphisms
$$\Us^-_{\A^\pi}(\g)^\vee \simeq [\REP(R)] \quad \text{ and } \quad  \Us^-_{\A^\pi}(\g) \simeq [\PROJ(R)]$$
as a ${\rm B}^{\mathrm{up}}_{\A^\pi}(\g)$-module and a ${\rm
B}^{\mathrm{low}}_{\A^\pi}(\g)$-module, respectively. In particular,
$\Us^-_{\A^\pi}(\g)$ and $\Us^-_{\A^\pi}(\g)^\vee$ are free
$\A^\pi$-modules. \ee
\end{corollary}

\begin{corollary}\label{cor:main3}
Let $M,M' \in \REP(R(\beta))$. If $\ch^\pi_q(M)=\ch^\pi_q(M')$, then $[M]=[M']$. In particular, if $M$ and
$M'$ are simple, then $M \simeq M'$.
\end{corollary}

\vskip 3mm

\subsection{Quantum Kac-Moody algebras}\label{sec:QKM}

In \cite{HW12}, Hill and Wang proposed
a condition on a Cartan
superdatum

\vskip 2mm

\ (C6)\quad the integer  $\dg_i$ is odd if and only if $i \in \Iod$.

\vskip 2mm

Under the condition (C6), we claim that there are equivalences of
categories
$$\Mod^\P(\Uqsg) \simeq \Mod^\P(\Uqsug) \simeq \Mod^\P(\Uqg), $$
where $\Uqg$ is the usual quantum Kac-Moody algebra with a parameter
$v$ (which will be set to be $\sqrt{\pi}q$).

Let us recall the definition of quantum Kac-Moody algebras. For $n
\in \Z_{\ge 0}$, set
$$[ n ]^v_i=[n]_{v^{\dg_i},v^{-\dg_i}} \quad \text{and} \quad
{\genfrac{[}{]}{0pt}{0}{n}{m}}_v=
{\genfrac{[}{]}{0pt}{0}{n}{m}}_{v^{\dg_i},v^{-\dg_i}}.$$
The {\it
quantum Kac-Moody algebra} $\Uqg$ associated with a Cartan datum
$(\car,\P,\Pi,\Pi^\vee)$ is the  $\Q(v)$-algebra generated by $e_i$,
$f_i$ and $K_i^{\pm 1}$ $(i \in I)$ subject to the following
defining relations:
\begin{align*}
&K_iK_j=K_jK_i, \quad
 K_i e_j K_i^{-1} = v^{ \dg_ia_{ij} } e_j, \quad K_i f_j K_i^{-1}= v^{ -\dg_i a_{ij} } f_j, \\
& e_if_j- f_je_i =\delta_{i,j} \dfrac{K_i-K^{-1}_i}{v^{\dg_i}-v^{-\dg_i}} \ \ (i,j \in I), \\
& \sum_{k=0}^{1-a_{ij}}
(-1)^k{\genfrac{[}{]}{0pt}{0}{1-a_{ij}}{k}}_v f^{{ 1-a_{ij}-k } }_i
f_j f_i^{ { k } } = 0 \ \ (i \neq j), \\
& \sum_{k=0}^{1-a_{ij}} (-1)^k
{\genfrac{[}{]}{0pt}{0}{1-a_{ij}}{k}}_v e^{{ 1-a_{ij}-k} }_i e_j
e_i^{{ k } } =0 \ \ (i \neq j).
\end{align*}
Hence $\Q[\sqrt{\pi}]\tens\Uqg$ is nothing but the algebra $\FS$
with $p_{ii}\theta_{ii}^{-1}=v^{2\dg_i}$. Recall  that the algebra
$\Uqsug$ is equal to $\FS$ with $p_{ii}\theta_{ii}^{-1}=q_i^2\pi_i$.

Assume that the condition  (C6) is satisfied and set
$v=q\sqrt{\pi}$. Then we have
$$
v^{2\dg_i}=(q\sqrt{\pi})^{2\dg_i}=q_i^2\pi^{\dg_i}=q_i^2\pi_i.$$
Therefore, combining with Theorem \ref{thm:Oint}, we obtain \eq&&
\ba{rl} \Mod^{\P}(\Q[\sqrt{\pi}]\tens_{\Q[\pi]}\Uqsg) & \simeq \Mod^\P(\Uqsug) \simeq \Mod^\P\bl
\Q[\sqrt{\pi}]\tens\Uqg\br, \\
\Oint^{\P}(\Q[\sqrt{\pi}]\tens_{\Q[\pi]}\Uqsg) & \simeq
\Oint^\P(\Uqsug) \simeq \Oint^\P\bl\Q[\sqrt{\pi}]\tens\Uqg\br.
\ea\eneq


\begin{thebibliography}{99}

\bibitem[BKM98]{BKM98} G.~Benkart, S.-J.~Kang and D.~Melville,
{\em Quantized enveloping algebras for Borcherds superalgebras},
Trans. Amer. Math. Soc. {\bf 350} (1998), 3297--3319.

\bibitem[BeKa07]{BerKaz07}
A.~Berenstein, D.~Kazhdan, \emph{Geometric and unipotent crystals.
{I}{I},
  {F}rom unipotent bicrystals to crystal bases}, Quantum groups, 13--88,
  Contemp. Math., {\bf 433}, Amer. Math. Soc., Providence, RI, 2007.


\bibitem[CHW13]{CHW} S.~Clark, D.~Hill, Weiqiang Wang,
{\em Quantum Supergroups I. Foundations}, arXiv:1301.1665.

\bibitem[EKL11]{EKL} A.~Ellis, M.~Khovanov, A.~Lauda,
{\em The odd nilHecke algebra and its diagrammatics},
arXiv:1111.1320.

\bibitem[HW12]{HW12}
D.~ Hill, W.~ Wang, {\em Categorification of quantum Kac-Moody
superalgebras}, arXiv:1202.2769.

\bibitem[Kac77]{Kac77}
V.~Kac, \emph{Lie superalgebras}, Adv. Math. {\bf 26} (1977), 8--96.

\bibitem[Kac90]{Kac90}
V.~Kac, \emph{Infinite-dimensional Lie Algebras}, Cambridge University Press, 1990.

\bibitem[KK11]{KK11}
S.-J. Kang, M.~Kashiwara, \emph{Categorification of highest weight
modules via {K}hovanov-{L}auda-{R}ouquier algebras},
Invent. math. {\bf 190} (2012), 699--742.

\bibitem[KKO12]{KKO12}
S.-J. Kang, M. Kashiwara, S.-j. Oh, \emph{Supercategorification of quantum Kac-Moody algebras},
arXiv:1206.5933

\bibitem[KKT11]{KKT11} S.-J.~Kang, M.~Kashiwara, S.~Tsuchioka,
{\em Quiver Hecke superalgebras}, arXiv:1107.1039.

\bibitem[KOP11a]{KOP11a}
S.-J. Kang, S.-j. Oh, E. Park, \emph{Perfect bases for integrable
modules over generalized Kac-Moody algebras}, Algebr. Represent. Theory
{\bf 14} (2011) no. 3, 571--587.

\bibitem[Kash91]{Kash91}
M.~Kashiwara, \emph{On crystal bases of the $q$-analogue of universal
  enveloping algebras}, Duke Math. J. \textbf{63} (1991), no.~2, 465--516.

\bibitem[KMPY96]{KMPY96} M.~Kashiwara, T.~Miwa, J.~Petersen, C.-M.~Yung,
\emph{Perfect crystals and $q$-deformed Fock spaces}, Selecta Math.
(N.S.) \textbf {2} (1996), 415--499.

\bibitem[KS06]{KS06}
M .Kashiwara, P. Schapira,
{\em Categories and Sheaves},
Grundlehren der mathematischen Wissenschaften {\bf 332},
Springer-Verlag Berlin Heidelberg (2006).

\bibitem[KT91]{KT91}
S.~M.~Khoroshkin, V. N. Tolstoy,
{\em Universal $R$-matrix for quantized (super)algebras},
Comm. Math. Phys. {\bf 141} (1991), no. 3, 599--617.

\bibitem[KL09]{KL1}
M.~Khovanov, A.~Lauda, \emph{A diagrammatic approach to
categorification of quantum groups {I}}, Represent. Theory
\textbf{13} (2009), 309--347.

\bibitem[KL11]{KL2}
M.~Khovanov, A.~Lauda, \emph{A diagrammatic approach to
categorification of quantum groups {II}}, Trans. Amer. Math. Soc.
{\bf 363} (2011), 2685--2700.

\bibitem[Kle05]{Kle05}
A.~Kleshchev, \emph{{L}inear and {P}rojective {R}epresentations of {S}ymmetric
  {G}roups}, Cambridge Tracts in Math. {\bf 163}, Cambridge University Press,
  Cambridge, 2005.


\bibitem[LV09]{LV09}
A.~Lauda, M.~Vazirani, \emph{Crystals from categorified quantum groups},
  Adv. Math. \textbf{228} (2011), no.~2, 803--861.

\bibitem[Lus93]{Lus93} G.~Lusztig. \emph{Introduction to Quantum Groups},
Progress in
Mathematics \textbf{110}, Birkh\"auser Boston Inc., Boston, MA, 1993.

\bibitem[R08]{R08}
R.~Rouquier, \emph{2-{K}ac-{M}oody algebras}, arXiv:0812.5023
(2008).

\bibitem[W10]{W10}
B.~Webster, \emph{Knot invariants and higher dimensional
representation theory I: diagrammatic and geometric categorification
of tensor products}, arXiv:1001.2020 (2010).

\end{thebibliography}
\end{document}